%% file: bowersBowersPratt-Cauchy.tex
\documentclass[11pt]{amsart}

\usepackage{amsmath}
\usepackage{amssymb}
\usepackage{amsfonts}
\usepackage[mathscr]{eucal}
\usepackage{graphicx}
\usepackage{epstopdf}
\usepackage{verbatim}
\usepackage{amscd}
\usepackage{upgreek}
\usepackage{dsfont}
\usepackage{caption}
\usepackage{amsthm}
\usepackage{color}
\usepackage{hyperref}
\usepackage[all]{xy}
\usepackage{pdfsync}
\usepackage{array}

\usepackage{pgf,tikz}
\usepackage{mathrsfs}
\usetikzlibrary{arrows}

\usepackage{subcaption}

\usepackage{ccfonts,eulervm}
\usepackage[T1]{fontenc}

\usetikzlibrary{shadows}
\usetikzlibrary{shapes}
\usetikzlibrary{decorations}
\usetikzlibrary{arrows,decorations.markings}   
\usetikzlibrary{patterns}

\makeindex
\numberwithin{equation}{section}

\newtheorem{Theorem}{Theorem}[section]
\newtheorem{Corollary}[Theorem]{Corollary}
\newtheorem{Proposition}[Theorem]{Proposition}
\newtheorem{Lemma}[Theorem]{Lemma}
\newtheorem*{MainTheorem}{Main Theorem}

\newtheorem*{Hyperideal}{Theorem}
\theoremstyle{definition}
\newtheorem*{Definition}{Definition}

\newtheorem{Remark}[Theorem]{Remark}

\usepackage{geometry}  % See geometry.pdf to learn the many layout options.
\usepackage[parfill]{parskip}  % Activate to begin paragraphs with empty 
\usepackage[percent]{overpic}   % to put overlay latex on images -- 2 stage 
\title[Rigidity of \textit{c}-polyhedra]{Rigidity of Circle Polyhedra in the $2 $-Sphere and of\\Hyperideal Polyhedra in Hyperbolic $3$-Space}

\author{John C. Bowers}
\address{Department of Computer Science, James Madison University, 
Harrisonburg VA 22807}
\email{bowersjc@jmu.edu}

\author{Philip L. Bowers}
\address{Department of Mathematics, The Florida State University, 
Tallahassee FL 32306}
\email{bowers@math.fsu.edu}

\author{Kevin Pratt}
\address{Department of Mathematics, University of Connecticut, Storrs CT 06269}
\email{kevin.pratt@uconn.edu}

\date{\today} % Activate to display a given date or no date

\begin{document}

\begin{abstract}
We generalize Cauchy's celebrated theorem on the global rigidity of convex polyhedra in Euclidean $3$-space $\mathbb{E}^{3}$ to the context of circle polyhedra in the $2$-sphere $\mathbb{S}^{2}$. We prove that any two convex and proper non-unitary \textit{c}-polyhedra with M\"obius-congruent faces that are consistently oriented are M\"obius-congruent. Our result implies the global rigidity of convex inversive distance circle packings in the Riemann sphere as well as that of certain hyperideal hyperbolic polyhedra in $\mathbb{H}^{3}$.
\end{abstract}

\subjclass[2010]{52C26}
\keywords{circle packing, inversive distance, hyperbolic geometry, hyperideal polyhedra}

\maketitle

\input{sec_introduction.tex}

\input{sec_preliminaries.tex}

\input{sec_cauchy.tex}

\input{sec_greenblackpolygonsNEW.tex}

\input{sec_mainresult.tex}

\input{sec_hyperideal.tex}

\end{document}

%% file: sec_introduction.tex
\section*{Introduction}
Start with an abstract oriented triangulation of the $2$-sphere $\mathbb{S}^2$ in which each edge is labeled with a non-negative real number. A {\em circle realization} of this triangulation is a drawing of circles on the sphere, one circle for each vertex, that respects orientation, and with the inversive distance between adjacent circles equal to the label for that edge. When connecting adjacent circle centers by geodesic arcs results in a geodesic triangulation of the $2$-sphere, this circle realization is called an {\em inversive distance circle packing}.\footnote{Perhaps a better term would be \textit{circle pattern} instead of \textit{packing} since adjacent circles may in fact be disjoint. We very soon will dispense with packings altogether and consider more general \textit{circle frameworks} and, specifically, \textit{circle polyhedra}.}  

The question addressed in this paper is whether these inversive distance circle packings of $\mathbb{S}^2$, and more generally these circle realizations, are uniquely determined by the underlying triangulation and the inversive distances between adjacent circles. Of course, uniqueness is with respect to the group of orientation-preserving, circle-preserving transformations of the $2$-sphere, the group of M\"obius transformations $\text{M\"ob}(\mathbb{S}^{2})$. In \cite{Bowers:2004bg}, Bowers and Stephenson questioned this uniqueness for general surfaces. In the case of packings on closed, orientable Euclidean or hyperbolic surfaces, Guo \cite{Guo:2011kf} verified their local rigidity, and later Luo \cite{Luo:2011ex} verified their global rigidity, as long as adjacent circles overlap in angles at most $\pi/2$. Surprisingly, such packings are not unique on $\mathbb{S}^2$. Ma and Schlenker \cite{Ma:2012hl} produced counter-examples by constructing pairs of circle packings on $\mathbb{S}^2$ for the octahedral graph realizing the same inversive distance edge data, but for which there is no M\"obius transformation taking one pattern onto the other. Their construction used an infinitesimally flexible Euclidean polyhedron, embeddings in de Sitter space $\mathbb{S}^{3}_{1}$, and the Pogorelov map between different geometries. In \cite{bowersBowers:2016}, the first two authors of the present work produced families of counter-examples inspired by Ma and Schlenker using only the inversive geometry of the sphere. 

Ma and Schlenker's result was surprising particularly because the famous Koebe-Andre'ev-Thurston Circle Packing Theorem implies the uniqueness of such packings when all inversive distance edge labels are in the closed unit interval $[0,1]$. In this case, adjacent circles always overlap, and in  angles between 0 (tangent circles) and $\pi/2$ (orthogonal overlap). In contrast, our interest is in studying packings where adjacent circles may be disjoint, where inversive distances are greater than unity. Of particular interest to us are \textit{edge-separated} packings in which all adjacent circles are disjoint. These packings always have \textit{ortho-circles}, meaning that any three mutually adjacent circles have a unique ortho-circle, a circle orthogonal to all three. In this paper we will allow adjacent circles to overlap as long as this ortho-circle property is preserved. We even will allow adjacent circles to overlap in an angle greater than $\pi/2$ but we will not allow tangencies. Our study then is of \textit{non-unitary circle packings with ortho-circles}, inversive distance circle packings for which all inversive distance edge labels take their values in $(-1, 1) \cup (1, \infty)$ and for which each face admits an ortho-circle. The constructions given in \cite{bowersBowers:2016} show that in general such packings are not unique. In contrast to this, our main result shows that if we restrict ourselves to {\em convex} non-unitary circle packings with ortho-circles, then the Bowers--Stephenson question is once more answered in the affirmative.
\begin{MainTheorem}\label{thm:main}
	Let $\mathcal{C}$ and $\mathcal{C}'$ be two non-unitary, inversive distance circle packings with ortho-circles for the same oriented edge-labeled triangulation of the $2$-sphere $\mathbb{S}^{2}$. If $\mathcal{C}$ and $\mathcal{C}'$ are convex and proper, then there is a M\"obius transformation $T:\mathbb{S}^2\rightarrow\mathbb{S}^2$ such that $T(\mathcal{C}) = \mathcal{C}'$. 
\end{MainTheorem}
Our interest really is in the generalization of circle packings that we call \textit{c}-frameworks and the circle realizations of labeled graphs, and in particular in a generalization of $3$-dimensional compact polyhedra to `polyhedral' patterns of circles in the $2$-sphere. This in fact is a major goal of this paper---to introduce these new concepts that generalize inversive distance circle packings to circle realizations, to define \textit{c}-polyhedra (with the circle packings of the Main Theorem as but special examples of \textit{c}-polyhedra), and to study the global rigidity of these \textit{c}-polyhedra. Finding the right notion for generalizing boundedness for polyhedra, which we term \textit{properness} for \textit{c}-polyhedra, from Euclidean geometry to circle geometry will take some effort. The most important ingredients for this study of rigidity are those of the convexity and the properness of \textit{c}-polyhedra, and these definitions and the pursuit of the just right notions of convexity and properness will take up a significant portion of the paper in Section~\ref{sec:prelims}, \textit{Preliminaries}, before we are able to prove this more general version of the Main Theorem.
\begin{MainTheorem}[General]
	Any two convex and proper non-unitary \textit{c}-polyhedra with M\"obius-congruent faces that are based on the same oriented abstract spherical polyhedron and are consistently oriented\,\footnote{Without the consistent orientation assumption, the \textit{c}-polyhedra are inversive-congruent.} are M\"obius-congruent.
\end{MainTheorem}
The Main Theorem coupled with the Ma-Schlenker example of~\cite{Ma:2012hl} and the examples of ~\cite{bowersBowers:2016} show that the uniqueness of inversive distance circle packings, and more generally, of \textit{c}-polyhedra is exactly analogous to that of Euclidean polyhedra---convex and bounded polyhedra in $\mathbb{E}^{3}$ are prescribed uniquely by their edge lengths and face angles whereas non-convex or unbounded polyhedra are not. The proof of this for convex and bounded Euclidean polyhedra is Cauchy's celebrated rigidity theorem \cite{Cauchy1813}, which is reviewed in Section~\ref{Section:CRT}. Our proof here follows Cauchy's original argument, which splits the proof into two components---a combinatorial lemma and a geometric lemma. Cauchy's combinatorial lemma deals with a certain labeling of the edges of any graph on a sphere, and applies to our setting. The geometric lemma, known as \textit{Cauchy's Arm Lemma}, requires that a polygon with certain properties be defined for each vertex of the polyhedron, and fails to apply here. The main work of this paper is in describing and analyzing a family of hyperbolic polygons that we call {\em green-black polygons} that are defined for each vertex of a \textit{c}-polyhedron in a M\"obius-invariant manner. We develop an analog of Cauchy's Arm Lemma for convex green-black polygons in Section~\ref{sec:greenblackpolygons} and use it to prove the Main Theorem in Section~\ref{Section:Proof}.

Though our interest in the Bowers-Stephenson question arises from our study of inversive distance packings, a second motivation for studying this topic stems from the equivalence between inversive distance circle packings, or more generally \textit{c}-polyhedra, and certain generalized hyperideal polyhedra in hyperbolic $3$-space $\mathbb{H}^{3}$. Our Main Theorem implies that the hyperideal polyhedra associated to convex \textit{c}-polyhedra are globally rigid. These hyperideal polyhedra are generalizations of those that Bao and Bonahon studied in~\cite{BaoBonahon:2002}. 
\begin{Hyperideal}\label{Theorem:Hyper}
	In Klein's projective model for the hyperbolic space in which $\mathbb{H}^{3}$ is identified with the unit open ball $B^{3}$ in $\mathbb{E}^{3} \subset \mathbb{R}\mathbb{P}^{3}$, let $P$ be the intersection with $\mathbb{H}^{3}$ of the compact, convex, polyhedron $P'$ in $\mathbb{R}\mathbb{P}^{3}$, all of whose vertices lie outside the closed unit ball $\mathbb{H}^{3} \cup \partial \mathbb{H}^{3} = B^{3}\cup \mathbb{S}^{2}$ and each of whose faces meets the open ball $\mathbb{H}^{3}$. If $P$ is proper and non-unitary, then $P$ is globally rigid, unique up to isometries of $\mathbb{H}^{3}$. (See Section~\ref{Section:Hyperideal} for the definitions.)
\end{Hyperideal}
To say that $P$ is \textit{globally rigid} is to say that $P$ is determined uniquely by the combinatorics of $P'$ and the hyperbolic isometry classes of the faces of the polyhedron $P$ (which are determined by the faces of the polyhedron $P^{*}$ dual to $P'$; see Section~\ref{Section:Hyperideal}). This means that combinatorics and isometry classes of faces determine uniquely the angles of intersection of adjacent faces that meet in $\mathbb{H}^{3}$, or their hyperbolic distance from one another if they do not meet in $\mathbb{H}^{3}$, again exactly analogous to bounded, convex Euclidean polyhedra.

This generalizes rigidity results of Bao-Bonahon~\cite{BaoBonahon:2002} and Rousset~\cite{Rousset:2004} on hyperideal polyhedra that require that, either all edges of $P'$ meet the unit ball $\mathbb{H}^{3}$, or all edges of $P'$ lie ``beyond infinity.''  Of course, Bao and Bonahon are able to prove more, namely, they give an existence result that characterizes in terms of dihedral angles and combinatorics those hyperideal polyhedra where all adjacent faces meet in $\mathbb{H}^{3}$. We do not provide such a characterization in our context, when some adjacent faces meet in $\mathbb{H}^{3}$ and others do not. The question of existence of these generalized hyperideal polyhedra is equivalent to the Bowers-Stephenson question of the existence of circle patterns with preassigned inversive distances between adjacent circles. When one allows for separated circles, this problem of existence becomes quite a bit more delicate than when adjacent circles are required to meet, with the possibility of even the local assignment of inversive distances about a central circle having no realization. These issues and the preceding theorem along with several corollaries, as well as a bit of background history of the rigidity of hyperbolic polyhedra, will be discussed in the final Section~\ref{Section:Hyperideal}.

%% file: sec_preliminaries.tex
\section{Preliminaries}\label{sec:prelims}
Circle packings have been studied primarily in the cases where adjacent circles intersect nontrivially. When adjacent circles are allowed to be separated, with absolute inversive distances larger than unity, new difficulties arise from several directions. Existence of packings becomes problematic, even locally. When working on the $2$-sphere, the facts that there are two complementary disks that a circle bounds and that the centers and radii of these disks fail to be M\"obius invariants present difficulties in even defining the right notion of a M\"obius-invariant circle packing. These difficulties suggest that the more natural setting for working with circle configurations on $\mathbb{S}^{2}$ with prescribed inversive distances is obtained by replacing circles by oriented circles, and circle packings by oriented circle realizations. We find that a more efficient language for oriented circle realizations on the $2$-sphere than that of the traditional circle packing language is one obtained by adapting some of the terminology of rigid Euclidean frameworks or linkages to the setting of circle configurations in $\mathbb{S}^{2}$.\footnote{The first two authors expand upon this point of view and present a general framework for studying the rigidity of circle realizations in~\cite{Bowers:rigidityOfCircle:techReport}.} In this rather lengthy preliminary section, we recall the necessary properties from the inversive geometry of the $2$-sphere needed for our result, we introduce oriented circle realizations and circle polyhedra, we describe appropriate notions of convexity and properness for circle  polyhedra, and we review some of and develop further the elementary geometry of hyperideal hyperbolic polygons and polyhedra.

\subsection{Inversive geometry of the Riemann sphere} We use the two usual models for the Riemann sphere, the round unit sphere $\mathbb{S}^{2}$ in Euclidean $3$-space $\mathbb{E}^{3}$ and its image under stereographic projection, the extended complex plane $\widehat{\mathbb{C}} = \mathbb{C} \cup \{\infty\}$. $\mathrm{Inv}(\mathbb{S}^{2})$ denotes the inversive group of the $2$-sphere generated by reflections in the circles of $\mathbb{S}^{2}$, and $\text{M\"ob}(\mathbb{S}^{2})$ its index two subgroup of M\"obius transformations. The absolute inversive distance between two circles in $\mathbb{S}^{2}$ is a M\"obius invariant of the placement of the two circles in the $2$-sphere. More useful for us is the general inversive distance that serves as an invariant for the placement of two relatively oriented circles in the $2$-sphere. We give two equivalent definitions, the first using the algebra of the extended plane $\widehat{\mathbb{C}}$ and the second using the intrinsic spherical metric of the $2$-sphere $\mathbb{S}^{2}$. Each has its advantages, as becomes apparent as the discussion advances.

For the definitions, note that an oriented circle determines a unique closed \textit{companion} or \textit{spanning disk} that the circle bounds. Indeed, assuming fixed orientations for $\mathbb{S}^{2}$ and $\widehat{\mathbb{C}}$ that are compatible via stereographic projection, the companion disk determined by the oriented circle $C$ is the closed complementary disk $D$ (of the two available) whose positively oriented boundary $\partial^{+} D = C$, where of coures the orientation of $D$ is inherited from that of $\mathbb{S}^{2}$ or $\widehat{\mathbb{C}}$. This is described colloquially by saying that $D$ lies to the left of $C$ as one traverses $C$ along the direction of its orientation. 

 \begin{Definition}[General inversive distance]
Let $C_{1}$ and $C_{2}$  be oriented circles in the extended plane $\widehat{\mathbb{C}}$ bounding their respective companion disks $D_{1}$ and $D_{2}$, and let $C$ be any oriented circle mutually orthogonal to $C_{1}$ and $C_{2}$.  Denote the points of intersection of $C$ with $C_{1}$ as $z_{1},z_{2}$ ordered so that the oriented sub-arc of $C$ from $z_{1}$ to $z_{2}$ lies in the disk $D_{1}$.  Similarly denote the ordered points of intersection of $C$ with $D_{2}$ as $w_{1},w_{2}$.  The \textit{general inversive distance} between $C_{1}$ and $C_{2}$, denoted as $\langle C_{1},C_{2}\rangle$, is defined in terms of the cross ratio
\begin{equation*}
[z_{1},z_{2};w_{1},w_{2}]=\frac{(z_{1}-w_{1})(z_{2}-w_{2})}{(z_{1}-z_{2})(w_{1}-w_{2})}
\end{equation*}
\noindent by
\begin{equation*}
\langle C_{1},C_{2}\rangle =2[z_{1},z_{2};w_{1},w_{2}]-1.
\end{equation*}
Subsequently, we drop the adjective \textit{general} and refer to the inversive distance $\langle C_{1}, C_{2} \rangle$ with its absolute value $|\langle C_{1}, C_{2} \rangle |$ the \textit{absolute inversive distance}.\footnote{The second author first learned of defining inversive distance in this way from his student, Roger Vogeler. He has looked for this in the literature and, unable to find it can only surmise that it is original with Prof.~Vogeler. The definition appeared in~\cite{Bowers:2003kr} in 2003.}
\end{Definition}

\begin{figure}

\begin{tikzpicture}
   
   \clip (-1.1, -6.5) rectangle (4.3, 3.7);
   
	% Top Left
	\begin{scope}[decoration = {markings, mark = at position 0.25 with {\arrow{>}}, mark = at
     position 0.75 with {\arrow{>}}, }]
     \draw[postaction = decorate] (0,2.4) circle (1.0cm);
   \end{scope}
   
   \begin{scope}[decoration = {markings, mark = at position 0.25 with {\arrow{>}}, mark = at
     position 0.75 with {\arrow{>}}, }]
     \draw[postaction = decorate, pattern=north west lines,  pattern color=lightgray] (0,2.4) circle (0.5cm);
   \end{scope}
   
   \node at (0,1.1) {$d < -1$};
   
   % Top Right
   
   \begin{scope}[decoration = {markings, mark = at position 0.25 with {\arrow{>}}, mark = at
     position 0.75 with {\arrow{>}}, }]
     \draw[postaction = decorate, pattern=north west lines,  pattern color=lightgray] (3,2.4) circle (1.0cm);
   \end{scope}
   
   \begin{scope}[decoration = {markings, mark = at position 0.25 with {\arrow{<}}, mark = at
     position 0.75 with {\arrow{<}}, }]
     \draw[postaction = decorate, fill=white] (3,2.4) circle (0.5cm);
   \end{scope}
   
   \node at (3,1.1) {$d > 1$};

	% 2 Left
   \begin{scope}[decoration = {markings, mark = at position 0.25 with {\arrow{>}}, mark = at
     position 0.75 with {\arrow{>}}, }]
     \draw[postaction = decorate, pattern=north west lines,  pattern color=lightgray] (0.5,-0.5) circle (0.5cm);
   \end{scope}
   
   \begin{scope}[decoration = {markings, mark = at position 0.25 with {\arrow{>}}, mark = at
     position 0.75 with {\arrow{>}}, }]
     \draw[postaction = decorate] (0,-0.5) circle (1.0cm);
   \end{scope}
   
   \node at (0,-1.8) {$d = -1$};
   
   % 2 Right
   
   \begin{scope}[decoration = {markings, mark = at position 0.25 with {\arrow{>}}, mark = at
     position 0.75 with {\arrow{>}}, }]
     \draw[postaction = decorate, pattern=north west lines,  pattern color=lightgray] (3,-0.5) circle (1.0cm);
   \end{scope}
   
   \begin{scope}[decoration = {markings, mark = at position 0.25 with {\arrow{<}}, mark = at
     position 0.75 with {\arrow{<}}, }]
     \draw[postaction = decorate, fill=white] (3.5,-0.5) circle (0.5cm);
   \end{scope}
   
   \node at (3,-1.8) {$d = 1$};
   
	% 3 Left
   \begin{scope}[decoration = {markings, mark = at position 0.25 with {\arrow{<}}, mark = at
     position 0.75 with {\arrow{<}}, }]
     \draw[postaction = decorate] (0.5,-3) circle (0.5cm);
   \end{scope}

   \begin{scope}[decoration = {markings, mark = at position 0.25 with {\arrow{>}}, mark = at
     position 0.75 with {\arrow{>}}, }]
     \draw[postaction = decorate, pattern=north west lines,  pattern color=lightgray] (-0.5,-3) circle (0.5cm);
   \end{scope}
   
   \node at (0,-3.9) {$d = -1$};

   % 3 Right
   
   \begin{scope}[decoration = {markings, mark = at position 0.25 with {\arrow{>}}, mark = at
     position 0.75 with {\arrow{>}}, }]
     \draw[postaction = decorate] (2.5,-3) circle (0.5cm);
   \end{scope}
   
   \begin{scope}[decoration = {markings, mark = at position 0.25 with {\arrow{>}}, mark = at
     position 0.75 with {\arrow{>}}, }]
     \draw[postaction = decorate, fill=white] (3.5,-3) circle (0.5cm);
   \end{scope}
   
   \node at (3,-3.9) {$d = 1$};

	% 4 Left
   
   \begin{scope}[decoration = {markings, mark = at position 0.25 with {\arrow{>}}, mark = at
     position 0.75 with {\arrow{>}}, }]
     \draw[postaction = decorate, pattern=north west lines,  pattern color=lightgray] (-0.45,-5.1) circle (0.55cm);
   \end{scope}
   
   \begin{scope}[decoration = {markings, mark = at position 0.25 with {\arrow{<}}, mark = at
     position 0.75 with {\arrow{<}}, }]
     \draw[postaction = decorate, fill=white] (0.45,-5.1) circle (0.55cm);
   \end{scope}
   
   \draw (-0.45,-5.1) circle (0.55cm);
   
   \node at (0,-6.0) {$-1 < d < 0$};
   
   % 4 Right
   
   \path [pattern=north west lines,  pattern color=lightgray] (1.8,-4.35) rectangle (4.2, -5.85);
   
   \begin{scope}[decoration = {markings, mark = at position 0.25 with {\arrow{<}}, mark = at
     position 0.75 with {\arrow{<}}, }]
     \draw[postaction = decorate, fill=white] (2.55,-5.1) circle (0.55cm);
   \end{scope}
   
   \begin{scope}[decoration = {markings, mark = at position 0.25 with {\arrow{<}}, mark = at
     position 0.75 with {\arrow{<}}, }]
     \draw[postaction = decorate, fill=white] (3.45,-5.1) circle (0.55cm);
     
   \draw[postaction = decorate] (2.55,-5.1) circle (0.55cm);
   \end{scope}
   
   \node at (3,-6.0) {$0 < d < 1$};
   
\end{tikzpicture}

\caption{Inversive distances $d=\langle C_{1}, C_{2} \rangle$. The shaded regions are the intersections $D_{1}\cap D_{2}$, the points common to the spanning disks $D_{1}$ and $D_{2}$ for both circles $C_{1}$ and $C_{2}$.}
\label{fig:ID}
\end{figure}
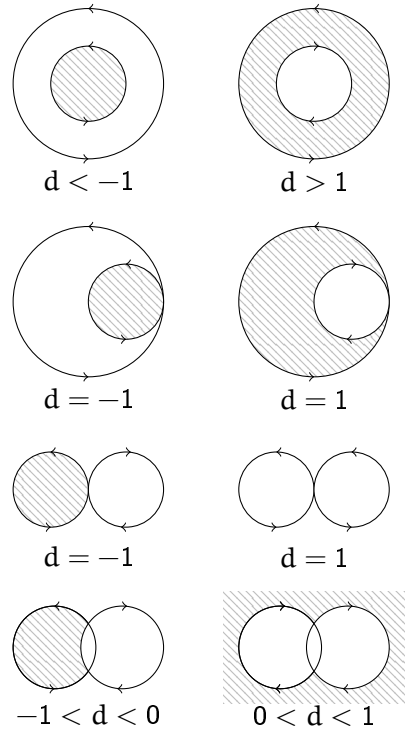
Recall that cross ratios of ordered $4$-tuples of points in $\widehat{\mathbb{C}}$ are invariant under M\"{o}bius transformations and that there is a M\"obius transformation taking an ordered set of four points of $\widehat{\mathbb{C}}$ to another ordered set of four if and only if the cross ratios of the sets agree.  This implies that which circle $C$ orthogonal to both $C_{1}$ and $C_{2}$ is used in the definition is irrelevant as a M\"{o}bius transformation that set-wise fixes $C_{1}$ and $C_{2}$ can be used to move any one orthogonal circle to another.  Which one of the two orientations on the orthogonal circle $C$ is used is irrelevant as the cross ratio satisfies $[z_{1},z_{2};w_{1},w_{2}]=[z_{2},z_{1};w_{2},w_{1}]$.  This equation also shows that the inversive distance is preserved when the orientation of both circles is reversed so that it is only the relative orientation of the two circles that is important for the definition. In fact, the general inversive distance is a relative conformal measure of the placement of an oriented circle pair on the Riemann sphere. By this we mean that two oriented circle pairs are inversive equivalent if and only if their inversive distances agree. All of this should cause one to pause to develop some intuition about how companion disks may overlap with various values of inversive distances. See Fig.~\ref{fig:ID} for some corrections to possible misconceptions. Finally, the inversive distance is symmetric with $\langle C_{1},C_{2}\rangle  = \langle C_{2},C_{1}\rangle$ since $[z_{1},z_{2};w_{1},w_{2}]=[w_{1},w_{2};z_{1},z_{2}]$.

The inversive distance is real since the cross ratio of points lying on a common circle is real and, in fact, every real value is realized as the inversive distance of some oriented circle pair. It is easy to give an intuitive understanding of the distance when $C_{1}$ and $C_{2}$ meet. In this case, the oriented angle $\alpha$ of overlap may be defined unambiguously as the angle between the unit tangent vectors to the circles at a point of overlap formed by one tangent pointing along the orientation of its parent circle and the other pointing against the orientation of its parent circle. This is precisely the \textit{lune angle} formed by the \textit{lune} $D_{1} \cap D_{2}$. The inversive distance is then $\langle C_{1}, C_{2} \rangle = \cos \alpha$. In particular, when the circles are tangent, the inversive distance is $\pm 1$, and when orthogonal, the inversive distance is zero. When the circles do not meet, the inversive distance lies outside the interval $[-1,1]$ and this is the range where the definitions in terms of hyperbolic, Euclidean, and spherical geometry give more geometric insight. We will give the spherical definition next, but we comment first that the general inversive distance is less than $-1$ if and only if one of the disks $D_{1}$ and $D_{2}$ is contained in the interior of the other.\footnote{We should mention that some of the older literature use $-\langle C_{1}, C_{2} \rangle$ for the inversive distance; for example~\cite{Sch:1962}.}  We refer the reader to~\cite{Bowers:2003kr} for a more detailed treatment. Notice that if the orientation of only one member of a circle pair is reversed, the inversive distance merely changes sign.  This follows from the immediate relation $[z_{1},z_{2};w_{2},w_{1}]=1-[z_{1},z_{2};w_{1},w_{2}]$. Despite its name, the inversive distance is not a metric as it fails to be non-negative and fails to satisfy the triangle inequality.\footnote{Some authors, perhaps more aptly, call the inversive \textit{distance} the inversive \textit{product} of $C_{1}$ and $C_{2}$.}

 The second definition is entirely in terms of the spherical metric.
 
\begin{Definition}[Inversive distance in the spherical metric]
In the $2$-sphere $\mathbb{S}^{2}$, the inversive distance may be expressed as\footnote{In both Luo~\cite{Luo:2011ex} and Ma-Schlenker~\cite{Ma:2012hl} there is a typo in the expression for the spherical formula for inversive distance. They report the negative of this formula.} 
\begin{equation}\label{EQ:sphereID} 
	\langle C_{1}, C_{2} \rangle = \frac{-\cos \sphericalangle ( p_{1},p_{2} ) + \cos(r_{1}) \cos(r_{2})}{\sin(r_{1}) \sin(r_{2})} =\frac{- p_{1}\cdot p_{2}  + \cos(r_{1}) \cos(r_{2})}{\sin(r_{1}) \sin(r_{2})}.
\end{equation}
Here, $\sphericalangle (p_{1},p_{2}) = \cos^{-1} (p_{1} \cdot p_{2})$ denotes the spherical distance between the centers, $p_{1}$ and $p_{2}$, of the respective companion disks, $p_{1} \cdot p_{2}$ the usual Euclidean inner product between the unit vectors $p_{1}$ and $p_{2}$, and $r_{1}$ and $r_{2}$ the respective spherical radii of the companion disks. Note that $r_{i} = \cos^{-1} (p_{i} \cdot q_{i})$ for any point $q_{i}$ on the circle $C_{i}$, for $i=1,2$.
\end{Definition}
Verifying the equivalence of the two definitions is an exercise in the use of trigonometric identities after a standard placement of $C_{1}$ and $C_{2}$ on $\mathbb{S}^{2}$ followed by stereographic projection. This standard placement is obtained by finding the unique great circle $C$ orthogonal to both $C_{1}$ and $C_{2}$ and then rotating the sphere so that this great circle is the equator, which then stereographically projects to the unit circle in the complex plane. We leave the details to the reader. In Section~\ref{sec:poincaredisk} we will recall the standard expression for inversive distance greater than unity in terms of hyperbolic geometry.

\subsection{Circle packings, frameworks, and realizations} We here generalize the language of circle packing and patterns of triangulations and quadrangulations of the sphere to that of circle realizations of oriented circle frameworks.\footnote{The material in this section is extended and further developed in~\cite{Bowers:rigidityOfCircle:techReport}.} Let $G$ be a graph, by which we mean a set of vertices $V= V(G)$ and simple edges $E = E(G)$. We disallow both loops and multiple edges. An oriented edge incident to the initial vertex $u$ and terminal vertex $v$ is denoted as $uv$, and $-uv$ means the oppositely oriented edge $vu$. We use the same notation, $uv$, to denote an un-oriented edge, context making the meaning clear. A \textit{circle framework with adjacency graph $G$}, or \textit{\textit{c}-framework} for short, is a collection $\mathcal{C} = \{ C_{u} : u \in V(G) \}$ of oriented circles in $\mathbb{S}^{2}$ indexed by the vertex set of $G$. We denote this by $G(\mathcal{C})$. Two \textit{c}-frameworks $G(\mathcal{C})$ and $G(\mathcal{C}')$ are \textit{equivalent} if $\langle C_{u}, C_{v} \rangle = \langle C_{u}', C_{v}'\rangle$ whenever $uv$ is an edge of $G$. Let $H$ be a subgroup of $\mathrm{Inv}(\mathbb{S}^{2})$. Two collections $\mathcal{C}$ and $\mathcal{C}'$ of oriented circles indexed by the same set are \textit{$H$-equivalent} or \textit{$H$-congruent} provided there is a mapping $T\in H$ such that $T(\mathcal{C}) = \mathcal{C}'$, respecting the common indexing and the orientations of the circles. When $H$ is not so important they are \textit{inversive-equivalent} or \textit{inversive-congruent}, and when $T$ can be chosen to be a M\"obius transformation, they are \textit{M\"obius-equivalent} or \textit{M\"obius-congruent}. The global rigidity theory of \textit{c}-frameworks concerns conditions on $G$ or $G(\mathcal{C})$ that ensure that the equivalence of the \textit{c}-frameworks $G(\mathcal{C})$ and $G(\mathcal{C}')$ guarantees their $H$-equivalence. Often we restrict our attention to \textit{c}-frameworks in a restricted collection $\mathscr{F}$ of \textit{c}-frameworks. In this paper, $\mathscr{F}$ is the collection of non-unitary, convex and proper \textit{c}-polyhedra, defined subsequently, and our interest is in M\"obius equivalence.

We develop next some terminology to describe how much separation there is between two oriented circles. Two oriented circles $C_{1}$ and $C_{2}$ are \textit{coupled} if one of the circles is contained in the companion disk of the other circle. This may occur in two distinct ways. It may be that one of the companion disks is contained in the other. This occurs exactly when $\langle C_{1} ,C_{2} \rangle \leq -1$, and in this case the pair remains coupled when both their orientations are reversed. On the other hand, it may be that neither companion disk is contained in the other as in the first row of Fig.~\ref{fig:ID}. In this case $\langle C_{1} ,C_{2} \rangle \geq 1$ and the two circles become uncoupled when both orientations are reversed. It follows then that $C_{1}$ and $C_{2}$ are \textit{uncoupled} precisely when either $C_{1} \cap C_{2}$ has two points, or the interiors of their companion disks are disjoint. If $C_{1}$ and $C_{2}$ are uncoupled, then $\langle C_{1} ,C_{2} \rangle > -1$. A stronger form of separation is that of segregation. The oriented circles $C_{1}$ and $C_{2}$ are \textit{segregated} if they they are uncoupled and \textit{overlap by at most $\pi/2$}. In particular, either the companion disks are disjoint, or when not disjoint, then the two bounding circles meet in an oriented angle of at most $\pi/2$.  In this case $\langle C_{1} ,C_{2} \rangle \geq 0$. Even stronger, the circles are \textit{separated} if the companion disks are disjoint, and in this case $\langle C_{1} ,C_{2} \rangle > 1$. Notice that $C_{1}$ and $C_{2}$ are uncoupled if segregated, and segregated if separated. A collection of oriented circles is respectively uncoupled, segregated, or separated if every distinct pair of the collection is uncoupled, segregated, or separated.

The oriented circles $C_{u}$ and $C_{v}$ of a \textit{c}-framework $G(\mathcal{C})$ are \textit{adjacent} provided $uv$ is an edge of $G$. The \textit{c}-framework $G(\mathcal{C})$
\begin{enumerate}
\item[(i)] is \textit{edge-uncoupled} if each pair of adjacent circles is uncoupled (implying that $\langle C_{u}, C_{v} \rangle > -1$ for all edges $uv$ of $G$);
\item[(ii)] is \textit{edge-segregated}  if each pair of adjacent circles is segregated so that the companion disks of all adjacent circles overlap by at most $\pi/2$ (implying that $\langle C_{u}, C_{v} \rangle \geq 0$ for all edges $uv$ of $G$); 
\item[(iii)] is \textit{edge-separated} if each pair of adjacent circles is separated so that the companion disks of all adjacent circles are disjoint (implying that $\langle C_{u}, C_{v} \rangle > 1$ for all edges $uv$ of $G$);
\item[(iv)] is \textit{non-unitary} if $\langle C_{u}, C_{v} \rangle \neq \pm 1$ for all edges $uv$ of $G$ (companion disks of adjacent circles are not tangent);
\item[(v)] has \textit{deep overlaps} if there is at least one edge $uv$ for which $\langle C_{u}, C_{v} \rangle < 0$ (implying that it is not edge-segregated).
\end{enumerate}

\begin{Definition}[Labeled graph and circle realization]
	An \textit{edge-label} is a real-valued function $\beta : E(G) \to \mathbb{R}$ defined on the edge set of $G$, and $G$ together with an edge-label $\beta$ is denoted as $G_{\beta}$ and called an \textit{edge-labeled graph}. The \textit{c}-framework $G(\mathcal{C})$ is a \textit{circle realization} of the edge-labeled graph $G_{\beta}$ provided $\langle C_{u}, C_{v} \rangle = \beta (uv)$ for every edge $uv$ of $G$, and we denote it as $G_{\beta}(\mathcal{C})$. 
\end{Definition}

Circle packings are circle realizations of edge-labeled graphs that arise as the $1$-skeletons of oriented triangulations of the $2$-sphere that also satisfy certain properties that ensure that the realizations of the triangular boundaries of faces respect orientation. The general definition allows for branch vertices and configurations of circles in which the open geodesic triangles cut out by connecting centers of adjacent circles overlap. There are subtleties in which we have no interest, so we adapt a restricted definition that corresponds to the circle packings that arise from spherical polyhedral metrics on triangulated surfaces. These are circle realizations of the edge-labeled $1$-skeleton $G_{\beta}= K^{(1)}_{\beta}$ of an oriented triangulation $K$ of $\mathbb{S}^{2}$ that produce oriented geodesic triangulations\footnote{By this we mean that the orientation of the geodesic triangulation determined by the packing is consistent with the orientation on $K$.} of the 2-sphere when adjacent circle centers are connected by geodesic arcs. The assumption here is that the centers of no two adjacent circles are antipodal, so that there is a unique geodesic arc connecting them, and that the centers of three circles corresponding to the vertices of a face of $K$ do not lie on a great circle. Now this causes no particular problems when all adjacent circles overlap nontrivially, the traditional playing field of circle packing, but does cause some real concern when adjacent circles may have inversive distance greater than unity. For example, a circle realization $\mathcal{C}$ may produce a geodesic triangulation of the sphere by connecting adjacent centers while its M\"obius image $T(\mathcal{C})$ may not. This is traced directly to the fact that neither circle centers nor radii, nor geodesic arcs, are M\"obius invariants in the inversive geometry of the sphere. This behavior does not occur for inversive distance circle packings of the Euclidean or hyperbolic planes (and surfaces), precisely because circle centers and geodesics are invariant under automorphisms and radii are invariant up to scale in Euclidean geometry and invariant in hyperbolic geometry. Our belief is that using centers and radii of circles in inversive geometry should be avoided except where these can be used to simplify computations (as in the use of the spherical definition of inversive distance). Our shift in this paper then is from inversive distance circle packings to inversive distance circle realizations. We are less concerned with possible underlying geodesic triangulations and more concerned with M\"obius-invariant quantities. For example, rather than working with a geodesic face formed by connecting the centers of three mutually adjacent circles, we are more interested in the existence of an ortho-circle, a circle mutually orthogonal to the three, which is a M\"obius invariant. Though our initial motivation was circle packing as reflected in the Main Theorem, our real interest has evolved to circle realizations as reflected in the general version of the Main Theorem.

\subsection{Circle polyhedra}
A precise definition of a \textit{circle polyhedron}, or \textit{\textit{c}-polyhedron} for short, is given subsequently, but it is, essentially, a \textit{c}-framework $G(\mathcal{C})$ where $G = P^{(1)}$, the $1$-skeleton of a Euclidean $3$-dimensional polyhedron $P$. There are conditions imposed that correspond to the fact that the faces of a Euclidean polyhedron are planar polygons. To describe this we borrow from the incidence geometry of the space of circles. It turns out that the space of all oriented circles on the $2$-sphere is topologically a twice punctured $3$-sphere homeomorphic to $\mathbb{S}^{2} \times \mathbb{R}$ that double covers the space of un-oriented circles, topologically a punctured projective space $\mathbb{RP}^{3*}$. What is more significant is that there is a natural incidence geometry of points, lines, and planes in each of these spaces of circles that mimics the incidence geometry of Euclidean space $\mathbb{E}^{3}$ in some significant ways. This is developed rather fully in~\cite{Bowers:isoInversive:techReport} and put to use in studying iso-inversive embeddings. We need not develop these incidence geometries here, but we do want to borrow one of the types of planes from the incidence geometry of circle space to use as a natural notion of a planar set of circles in $\mathbb{S}^{2}$. These are called \textit{hyperbolic \textit{c}-planes} in~\cite{Bowers:isoInversive:techReport} (as opposed to \textit{parabolic} and \textit{elliptic \textit{c}-planes}) and, in the older literature, are examples of bundles of circles. We will call these \textit{\textit{c}-planes}. A \textit{c}-plane $\Pi_{O}$ is determined by its \textit{generating circle} $O$, an un-oriented circle in $\mathbb{S}^{2}$, and is defined as the collection of all oriented circles in $\mathbb{S}^{2}$ that meet $O$ orthogonally, or what is the same, all circles $C$ for which $\langle C, O \rangle = 0$ when $O$ is given either orientation. Any collection $\mathcal{C}$ of oriented circles is said to be \textit{\textit{c}-planar} provided all the circles of the collection belong to a common \textit{c}-plane, or what is the same, all the circles of $\mathcal{C}$ have a common \textit{ortho-circle}, a circle $O$ orthogonal to each circle of $\mathcal{C}$. When the \textit{c}-planar collection $\mathcal{C}$ has at least three distinct elements that are not coaxial, the collection has a unique ortho-circle $O$ determining the unique \textit{spanning plane} $\Pi_{O}$ of $\mathcal{C}$.\footnote{We should point out that three distinct circles need not have a common ortho-circle, and therefore need not lie on a \textit{c}-plane as defined here. Nonetheless, three distinct circles that are not coaxial always lie on a unique \textit{c}-plane as defined in~\cite{Bowers:isoInversive:techReport}, as the three may determine a \textit{parabolic} or \textit{elliptic} \textit{c}-plane instead of a \textit{hyperbolic} one.}

An \textit{abstract polyhedral graph} $G$ is the $1$-skeleton of an \textit{abstract spherical polyhedron} $P$, by which we mean an oriented cellular decomposition of the $2$-sphere into a finite number of combinatorial polygons\footnote{We do not allow bi-gons, only $n$-gons for $n\geq 3$.} for which distinct polygonal faces are either disjoint, or meet at a set of vertices and full edges.\footnote{In the literature a \textit{polyhedral graph} is one isomorphic to the $1$-skeleton of a convex polyhedron in $\mathbb{E}^{3}$. By Steinitz's Theorem, these are characterized as precisely the 3-vertex-connected planar graphs. Our abstract polyhedral graphs are more general in that we allow for $1$-skeletons of non-convex polyhedra, for example, or for two faces to share two or more non-contiguous edges.} We can think of $P$ as set of topological polygons glued together along full edges to form a topological sphere. Writing $P= ( V, E , F )$, where $V= V(P)$ is the finite set of vertices, $E=E(P)$ is the finite set of edges, and $F=F(P)$ is the finite set of combinatorial polygons, we have $G = (V, E)$. Each $n$-gon face $f\in F$ may be denoted by listing its vertices in cyclic order as $f = v_{1}v_{2}\dots v_{n}$, respecting the orientation. Notice that if two faces share the two adjacent vertices $u$ and $v$, then one face determines the oriented edge $uv$, while the other, the oriented edge $vu$.
\begin{Definition}[Circle polyhedron]
	Let $G = P^{(1)}$ be an abstract polyhedral graph where $P$ is an oriented abstract spherical polyhedron. The \textit{c}-framework $G(\mathcal{C})$ is called a \textit{circle polyhedron based on $P$}, or a \textit{\textit{c}-polyhedron} for short, provided it is edge-uncoupled and, for each $n$-gon face $f=u_{1}\dots u_{n}$ of $P$, the \textit{\textit{c}-face} $\mathcal{C}_{f} = \{ C_{u_{i}}: i= 1, \dots , n\}$ is not coaxial but is \textit{c}-planar. We denote the unique ortho-circle for $\mathcal{C}_{f}$ as $O_{f}$. 
\end{Definition}
The primary interest of this paper is in the global rigidity of \textit{c}-polyhedra. The Ma-Schlenker example of~\cite{Ma:2012hl} and the Ma-Schlenker \textit{c}-octahedra constructed in ~\cite{bowersBowers:2016} show that in general \textit{c}-polyhedra are not globally rigid with respect to the inversive group of the $2$-sphere. We next define a notion of convexity for \textit{c}-polyhedra. The Ma-Schlenker \textit{c}-octahedra are not convex and our Main Theorem shows it is this lack of convexity that allows for the non-rigidity of these examples.

\begin{Definition}[Convex \textit{c}-polyhedron]
	Let $G(\mathcal{C})$ be a \textit{c}-polyhedron based on the abstract spherical polyhedron $P$ with face set $F=F(P)$. For any face $f\in F$ we say that $G(\mathcal{C})$ is \textit{convex with respect to $f$} provided the ortho-circle $O_{f}$ may be oriented so that every circle $C_{u} \in \mathcal{C}$ is segregated from the oriented ortho-circle $O_{f}$. $G(\mathcal{C})$ then is \textit{convex} provided it is convex with respect to each of its faces so that the ortho-circles $O_{f}$ for $f\in F$ may be oriented so that every circle $C_{u} \in \mathcal{C}$ is segregated from every oriented ortho-circle $O_{f}$ for $f\in F$. The oriented ortho-circles for $f\in F$ that meet this condition are denoted as $O_{f}^{+}$. To avoid unnecessary pathology, we make the further assumption that the circles corresponding to three consecutive vertices on the face $f$ are never coaxial.\footnote{This is the analogue with convex Euclidean polyhedra of avoiding three vertices lying on a line. The middle vertex is unnecessary.}
\end{Definition}
%\begin{Proposition}
%	If $G(\mathcal{C})$ is a convex \textit{c}-polyhedron based on $P$, then $G$ is 3-vertex-connected, so by Steinitz's Theorem is isomorphic to the $1$-skeleton of a convex polyhedron in $\mathbb{E}^{3}$.
%\end{Proposition}
%\begin{Proposition}\label{Prop:uncoupledO-C}
%	 If $G(\mathcal{C})$ is a convex \textit{c}-polyhedron based on $P$ and $uv$ is an edge of $P$ incident to the two faces $f, g \in F(P)$, then the oriented ortho-circles $O_{f}^{+}$ and $O_{g}^{+}$ are uncoupled. In particular, any two oriented ortho-circles to adjacent \textit{c}-faces $\mathcal{C}_{f}$ and $\mathcal{C}_{g}$ that are disjoint are in fact separated.
%\end{Proposition}
We have stated the Main Theorem in terms of M\"obius rigidity instead of inversive rigidity. For this we need to impose a condition to make sure that two \textit{c}-polyhedra based on the same abstract spherical polyhedron $P$ are oriented consistently, both with the orientation of $P$. We can insure this with the application of the antipodal map if needed.
\begin{Proposition}\label{prop:faceorder}
	Let $G(\mathcal{C})$ be a convex \textit{c}-polyhedron based on $P$.  The either (i) for every oriented $n$-gon face $f = u_{1}\dots u_{n}$ of $P$, the circles $C_{u_{1}}, \dots , C_{u_{n}}$ are met in that order as one progresses around $O_{f}^{+}$ in the direction of its orientation, starting at $C_{u_{1}}$; or (ii) for every oriented $n$-gon face $f = u_{1}\dots u_{n}$ of $P$, the circles $C_{u_{1}}, \dots , C_{u_{n}}$ are met in that order as one progresses around $O_{f}^{+}$ in the direction opposite of its orientation, starting at $C_{u_{1}}$.
\end{Proposition}
By replacing the \textit{c}-polyhedron $G(\mathcal{C})$ by its image $G(\mathcal{C}^{*})$ under the antipodal mapping of the sphere $\mathbb{S}^{2}$ if necessary, and then reversing all circle orientations, we may assume that property (i) of Proposition~\ref{prop:faceorder} holds. We then say that the \textit{c}-polyhedron $G(\mathcal{C})$ is \textit{oriented consistently} with the orientation of the base polyhedron $P$ for $G(\mathcal{C})$, which will be assumed without comment in the remainder of the paper. Without the assumption of consistent orientation, we would need to work in terms of inversive rigidity, allowing for an orientation reversing map to match orientations before applying a M\"obius transformation.

Before pursuing the definition of properness for \textit{c}-polyhedra, we pause to develop prerequisite terminology for hyperideal polygons in the hyperbolic plane.

% convex implies interstice condition, star of as vertex, flower of circles.

\subsection{Hyperideal polygons and complex angles}\label{sec:poincaredisk}
We assume a fixed orientation on the hyperbolic plane $\mathbb{H}^{2}$, and when using an open disk $D$ in $\mathbb{S}^{2}$ or $\widehat{\mathbb{C}}$ as a Poincar\'e-disk model of $\mathbb{H}^{2}$, we assume the orientation of $\mathbb{H}^{2} = D$ is inherited from that of $\mathbb{S}^{2}$.  Let $\ell_{1}$ and $\ell_{2}$ be oriented lines in the hyperbolic plane $\mathbb{H}^{2}$. We define the \textit{complex angle} $\theta = \theta(\ell_{1}, \ell_{2})$ between $\ell_{1}$ and $\ell_{2}$. When the lines meet at the point $p$, $\theta = \theta (\ell_{1}, \ell_{2})$ is the angle between the unit tangent vectors to the lines at $p$ formed by one tangent pointing along the orientation of its parent line and the other pointing against the orientation of its parent line. This is precisely the \textit{lune angle} formed by the \textit{lune} $h_{1} \cap h_{2}$, where $h_{i}$ is the oriented half-plane with $\partial ^{+} h_{i} = \ell_{i}$, the half-plane whose positively oriented boundary is $\ell_{i}$, for $i =1,2$. When $\ell_{1}$ and $\ell_{2}$ are parallel, meeting at an ideal point $p \in \partial \mathbb{H}^{2}$, the angle $\theta =\theta (\ell_{1}, \ell_{2})$ is either zero or $\pi$, depending on the relative orientations. When $\ell_{1}$ and $\ell_{2}$ are ultra-parallel, there is no point of intersection, ideal or finite. When the orientations are \textit{consistent} with one another, meaning that neither of $h_{1}$ and $h_{2}$ is contained in the other, we define $\theta(\ell_{1}, \ell_{2}) = {i} \,d_{\mathbb{H}^{2}} (\ell_{1}, \ell_{2})$, this also called the \textit{imaginary angle} between the lines. Of course $i$ is the imaginary unit and $d_{\mathbb{H}^{2}} (\ell_{1}, \ell_{2})$ is the hyperbolic distance between the lines, the length of the unique geodesic arc orthogonal to both $\ell_{1}$ and $\ell_{2}$, marked as the green dotted lines of Fig.~\ref{fig:hyperlines6}. When the orientations are inconsistent, the complex angle is a \textit{phase shift} of the imaginary angle by $\pi$, by which we mean that $\theta (\ell_{1}, \ell_{2}) = \pi + i d_{\mathbb{H}^{2}}(\ell_{1},\ell_{2})$. In a Poincar\'e disk model of the hyperbolic plane (which includes the upper-half-plane model), for $i=1,2$, $\ell_{i}$ is the intersection with the disk of an oriented circle $C_{i}$ in $\mathbb{C}$ that meets the boundary of the disk orthogonally. In this case, the inversive distance satisfies $\langle C_{1}, C_{2} \rangle = \cos \theta(\ell_{1}, \ell_{2})$. Notice that the range of complex angles is the curve $\Theta = i\mathbb{R}_{+} \cup [0 ,\pi\,] \cup (\pi + i\mathbb{R}_{+} )$ in the complex plane\footnote{$\mathbb{R}_{\pm} = \{ x\in \mathbb{R} : \pm x >0 \}$, respectively the sets of positive and negative real numbers.} on which the complex cosine function $\cos$ is a homeomorphism onto the real line $\mathbb{R}$. When we use the inverse function $\cos^{-1}$, we mean the inverse of this homeomorphism.

\begin{figure}
\centering
\begin{subfigure}[t]{0.45\textwidth}
	\centering
	\begin{tikzpicture}[line cap=round,line join=round,>=triangle 45,x=0.5cm,y=0.5cm]
\clip(-3.5,-3.5) rectangle (3.5,3.5);
\draw [line width=1.7pt] (0.,-0.02) circle (1.64cm);
\draw [shift={(3.1876595412987103,2.4888869320995317)},line width=1.0pt] plot[domain=2.866682812109121:4.750106932851553,variable=\t]({1.*2.3868990320691204*cos(\t r)+0.*2.3868990320691204*sin(\t r)},{0.*2.3868990320691204*cos(\t r)+1.*2.3868990320691204*sin(\t r)});
\draw [shift={(-2.611735369123523,2.455413256516329)},line width=1.0pt] plot[domain=4.377652143067875:6.05145783060787,variable=\t]({1.*1.4800109556579497*cos(\t r)+0.*1.4800109556579497*sin(\t r)},{0.*1.4800109556579497*cos(\t r)+1.*1.4800109556579497*sin(\t r)});
\draw [shift={(-1.1447914011904725,3.26)},line width=1.0pt] plot[domain=-1.5939397725248288:0.,variable=\t]({1.*1.1447914011904727*cos(\t r)+0.*1.1447914011904727*sin(\t r)},{0.*1.1447914011904727*cos(\t r)+1.*1.1447914011904727*sin(\t r)});
\draw [shift={(2.5328000152792263,-2.9092597645064924)},line width=1.0pt] plot[domain=1.267562174466263:2.640078689604578,variable=\t]({1.*2.001124160114651*cos(\t r)+0.*2.001124160114651*sin(\t r)},{0.*2.001124160114651*cos(\t r)+1.*2.001124160114651*sin(\t r)});
\draw [shift={(-0.7874568535451341,-4.2497979049415955)},line width=1.0pt] plot[domain=0.9736786681253566:1.9183851272700465,variable=\t]({1.*2.784399147543959*cos(\t r)+0.*2.784399147543959*sin(\t r)},{0.*2.784399147543959*cos(\t r)+1.*2.784399147543959*sin(\t r)});
\draw [shift={(-3.1525200300893337,-1.7028016674997803)},line width=1.0pt] plot[domain=0.04999814551398941:1.6529638491565948,variable=\t]({1.*1.4183807642006712*cos(\t r)+0.*1.4183807642006712*sin(\t r)},{0.*1.4183807642006712*cos(\t r)+1.*1.4183807642006712*sin(\t r)});
\end{tikzpicture}
\caption{The convex hyperideal polygon defined by six hyperbolic lines.}
\label{fig:hyperpoly6}
\end{subfigure}
\quad
\begin{subfigure}[t]{0.45\textwidth}
	\centering
	\definecolor{qqccqq}{rgb}{0.,0.69215686274509803,0.2}
\begin{tikzpicture}[line cap=round,line join=round,>=triangle 45,x=0.5cm,y=0.5cm]
\clip(-3.5,-3.5) rectangle (3.5,3.5);
\draw [line width=1.7pt] (0.,-0.02) circle (1.64cm);
\draw [shift={(-1.1447914011904745,3.26)},line width=1.0pt] plot[domain=-2.469986497336053:0.,variable=\t]({1.*1.1447914011904736*cos(\t r)+0.*1.1447914011904736*sin(\t r)},{0.*1.1447914011904736*cos(\t r)+1.*1.1447914011904736*sin(\t r)});
\draw [shift={(-2.6117353691235228,2.4554132565163282)},line width=1.0pt] plot[domain=-1.905533164111711:0.38831871817246555,variable=\t]({1.*1.480010955657949*cos(\t r)+0.*1.480010955657949*sin(\t r)},{0.*1.480010955657949*cos(\t r)+1.*1.480010955657949*sin(\t r)});
\draw [shift={(-3.1525200300893315,-1.702801667499778)},line width=1.0pt] plot[domain=-0.6723293713358327:1.6529638491565963,variable=\t]({1.*1.4183807642006683*cos(\t r)+0.*1.4183807642006683*sin(\t r)},{0.*1.4183807642006683*cos(\t r)+1.*1.4183807642006683*sin(\t r)});
\draw [shift={(-0.7874568535451343,-4.249797904941598)},line width=1.0pt] plot[domain=0.5197946225862194:2.2536745215993577,variable=\t]({1.*2.784399147543961*cos(\t r)+0.*2.784399147543961*sin(\t r)},{0.*2.784399147543961*cos(\t r)+1.*2.784399147543961*sin(\t r)});
\draw [shift={(2.5328000152792267,-2.909259764506493)},line width=1.0pt] plot[domain=1.2675621744662633:3.3135307514135377,variable=\t]({1.*2.0011241601146517*cos(\t r)+0.*2.0011241601146517*sin(\t r)},{0.*2.0011241601146517*cos(\t r)+1.*2.0011241601146517*sin(\t r)});
\draw [shift={(3.1876595412987103,2.4888869320995317)},line width=1.0pt] plot[domain=2.866682812109121:4.750106932851553,variable=\t]({1.*2.3868990320691204*cos(\t r)+0.*2.3868990320691204*sin(\t r)},{0.*2.3868990320691204*cos(\t r)+1.*2.3868990320691204*sin(\t r)});
\draw [shift={(-3.667217832595169,0.4569294489406403)},line width=1.0pt,dotted,color=qqccqq]  plot[domain=-0.6438439087145857:0.37090862791196577,variable=\t]({1.*1.7080832330336952*cos(\t r)+0.*1.7080832330336952*sin(\t r)},{0.*1.7080832330336952*cos(\t r)+1.*1.7080832330336952*sin(\t r)});
\draw [shift={(0.6224598461281102,3.477252036412751)},line width=1.0pt,dotted,color=qqccqq]  plot[domain=3.962242862889318:4.8637080421743715,variable=\t]({1.*1.3637551342652905*cos(\t r)+0.*1.3637551342652905*sin(\t r)},{0.*1.3637551342652905*cos(\t r)+1.*1.3637551342652905*sin(\t r)});
\draw [shift={(3.7312779859576266,-0.47264848672367893)},line width=1.0pt,dotted,color=qqccqq]  plot[domain=2.6675970598300136:3.42672191285405,variable=\t]({1.*1.8354634458428318*cos(\t r)+0.*1.8354634458428318*sin(\t r)},{0.*1.8354634458428318*cos(\t r)+1.*1.8354634458428318*sin(\t r)});
\end{tikzpicture}
	\caption{The green dotted segments denote the hyperideal vertices, the unique shortest paths between consecutive non-intersecting lines.}
\label{fig:hyperlines6}
\end{subfigure}

\caption{A convex hyperideal $6$-gon.}
\label{fig:hyperidealpolygon6}
\end{figure}
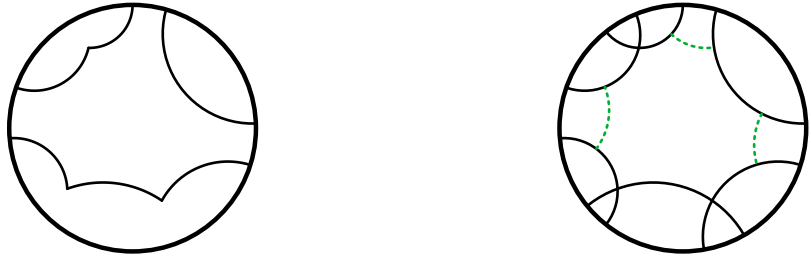
Let $H = \{h_1, \dots, h_n\}$ be a set of half-planes in $\mathbb{H}^2$ such that the region $P = h_1\cap  \dots \cap h_n$ is non-empty and the boundary line $\ell_{i}$ of each $h_i$ supports a non-empty segment, ray, or line on the boundary of $P$. The lines $\ell_{i}$ for $i = 1, \dots, n$ are oriented consistent with the orientation $P$ inherits from $\mathbb{H}^{2}$ so that $\partial^{+} h_{i} = \ell_{i}$. If $P$ is bounded, then it is a compact convex polygon in $\mathbb{H}^2$. Whether or not $P$ is bounded, we call $P$ a {\em convex hyperideal polygon}. As is usual, we sometimes refer to the oriented polygonal boundary $\partial^{+} P$ as a convex hyperideal polygon. An example is shown in Fig.~\ref{fig:hyperpoly6}. Suppose that the half-planes are indexed so that $\ell_{1} , \dots , \ell_{n}$ are in cyclic counter-clockwise order along $P$. Two consecutive lines $\ell_{i-1}, \ell_{i}$ may intersect either at a point in $\mathbb{H}^2$ forming a {\em finite vertex} of $P$, or at an ideal point at infinity forming an {\em ideal vertex}, or not at all. In the latter case, the hyperbolic segment orthogonal to the two consecutive lines is called the \textit{hyperideal vertex} of $P$ determined by $\ell_{i-1}$ and $\ell_{i}$. The complex angle $\theta_{i} = \theta(\ell_{i-1}, \ell_{i})$ is the \textit{complex interior angle} of the polygon $P$ determined by $\ell_{i-1}$ and $\ell_{i}$; see Fig.~\ref{fig:hyperlines6} . In Section~\ref{Section:Green-black} we use hyperideal polygons to construct \textit{green-black} polygons that then are used in Section~\ref{sec:greenblackpolygons} to generalize the Cauchy Arm Lemma for use in the proof of the Main Theorem.

Moving up one dimension, the \textit{complex dihedral angle} between two oriented hyperbolic planes in the hyperbolic $3$-space $\mathbb{H}^{3}$ is defined analogously to the complex angle between oriented lines. When the oriented planes $h_{1}$ and $h_{2}$ meet, $\theta = \theta (h_{1}, h_{2} )$ is the dihedral angle between them, the angle between their normal vectors at a point of intersection, one along the orientation of the plane and the other against the orientation. When the planes meet at an ideal point, the complex dihedral angle is zero or $\pi$ depending on relative orientations. When the planes are disjoint, the complex dihedral angle between them is the imaginary angle $\theta = \theta (h_{1}, h_{2} ) =  i\,d_{\mathbb{H}^{3}}(h_{1}, h_{2})$, or this phase-shifted by $\pi$, where $d_{\mathbb{H}^{3}}(h_{1}, h_{2})$ is the hyperbolic distance between the planes, the length of the unique geodesic arc orthogonal to both $h_{1}$ and $h_{2}$. The presence or absence of the phase shift depends on relative orientations. In either the Poincar\'e ball or the upper-half-space model of $\mathbb{H}^{3}$, the oriented boundaries of the planes $h_{1}$ and $h_{2}$ are oriented circles $C_{1}$ and $C_{2}$ on the sphere at infinity, and $\langle C_{1}, C_{2} \rangle = \cos \theta (h_{1}, h_{2})$. This motivates the next definition.
\begin{Definition}[Complex dihedral angle]
	Let $G(\mathcal{C})$ be a convex \textit{c}-polyhedron based on the abstract spherical polyhedron $P$ with face set $F=F(P)$. For each pair of faces $f, g \in F$ that share an edge, define the \textit{complex dihedral angle} between the \textit{adjacent} \textit{c}-faces $\mathcal{C}_{f}$ and $\mathcal{C}_{g}$ at that edge to be $\cos^{-1} \langle O_{f}^{+}, O_{g}^{+} \rangle$.
\end{Definition}

By the preceding discussion, the complex dihedral angle between the adjacent \textit{c}-faces $\mathcal{C}_{f}$ and $\mathcal{C}_{g}$ is the complex dihedral angle between the oriented hyperbolic planes $h_{f}$ and $h_{g}$ in the Poincar\'e ball model of hyperbolic $3$-space, where $h_{f}$ is the oriented hyperbolic plane in the ball $B^{3} = \mathbb{H}^{3}$ whose oriented boundary at infinity is $O_{f}^{+}$.

\begin{Theorem}\label{Thm:Congruence ConditionI}
	Let $G(\mathcal{C})$ and $G(\mathcal{C}')$ be two convex non-unitary \textit{c}-polyhedra both based on the abstract spherical polyhedron $P$. Then $G(\mathcal{C})$ and $G(\mathcal{C}')$ are M\"obius-congruent with one another if and only if the following conditions hold.
	\begin{enumerate}
		\item[(i)] $G(\mathcal{C})$ and $G(\mathcal{C}')$ have M\"obius-congruent \textit{c}-faces. This means that for each face $f$ of $P$, there is an M\"obius transformation $T\in \text{M\"ob}(\mathbb{S}^{2})$ such that $T(\mathcal{C}_{f}) = \mathcal{C}_{f}'$, taking corresponding oriented circles to corresponding oriented circles. 
		\item[(ii)] Corresponding complex dihedral angles between adjacent \textit{c}-faces agree; equivalently, for each pair of adjacent faces $f,g \in F(P)$, $\langle O_{f}^{+},O_{g}^{+} \rangle = \langle {O'_{f}}^{+},{O'_{g}}^{+} \rangle$, where the prime denotes the oriented ortho-circles for $G(\mathcal{C}')$.
	\end{enumerate}
\end{Theorem}
\begin{proof}
	The forward direction is obvious. For the reverse direction it is enough to verify the following: If $f, g \in F(P)$ are faces that share an edge $uv$ and $T\in \text{M\"ob}(\mathbb{S}^{2})$ is a M\"obius transformation such that $T(\mathcal{C}_{f}) = \mathcal{C}_{f}'$, taking corresponding oriented circles to corresponding oriented circles, then $T(\mathcal{C}_{g}) = \mathcal{C}_{g}'$, taking corresponding oriented circles to corresponding oriented circles. To verify this, let $f$, $g$, $uv$, and $T$ be as in the preceding sentence and observe that $T(O_{f}^{+}) = {O'_{f}}^{+}$. This uses property (i) of Proposition~\ref{prop:faceorder} to ensure that $T$ preserves the orientation on the ortho-circles. The circles $C_{u}'$ and $C_{v}'$ belong to a unique coaxial family $\mathcal{A}$ of circles in $\mathbb{S}^{2}$ and the ortho-circles $O_{f}'$ and $O_{g}'$ belong to its orthogonal complement $\mathcal{A}^{\perp}$, the collection of circles orthogonal to the circles in the family $\mathcal{A}$. Now $T(O_{f}^{+}) = {O_{f}'}^{+}$ and the question we ask is what of the image $T(O_{g}^{+})$: does $T(O_{g}^{+}) = {O_{g}'}^{+}$? Note first that $T(O_{g})$, since it is orthogonal to both $C_{u}'$ and $C_{v}'$, belongs to the family $\mathcal{A}^{\perp}$. Since the \textit{c}-polyhedra are non-unitary, the coaxial family $\mathcal{A}$ is not parabolic, and hence neither is $\mathcal{A}^{\perp}$. This implies that one of $\mathcal{A}$ and $\mathcal{A}^{\perp}$ is elliptic and the other hyperbolic.
	
	The fact we use here is that if one chooses any two distinct circles $O$ and $A$ in an elliptic or hyperbolic coaxial family and orients them, say to obtain the oriented circles $O^{+}$ and $A^{+}$, then there is exactly one circle $B$ distinct from $O$ and $A$ in that same coaxial family that admits an orientation to give an oriented circle $B^{+}$ so that the inversive distances satisfy $\langle O^{+}, A^{+}\rangle = \langle O^{+}, B^{+} \rangle$. Of course then also $\langle O^{+}, A^{-}\rangle = \langle O^{+}, B^{-} \rangle$, where the minus-superscript denotes the opposite orientation. By this observation, there is precisely one oriented circle $B^{+}$ other than $A^{+} = {O_{g}'}^{+}$ of the orthogonal complement $\mathcal{A}^{\perp}$ whose inversive distance to $T(O_{f}^{+}) = {O_{f}'}^{+}$ is $\langle {O_{f}'}^{+}, {O_{g}'}^{+} \rangle$. It follows that $T(O_{g}^{+})$ is equal to either $B^{+}$ or to ${O_{g}'}^{+}$, and our claim is that the convexity of the \textit{c}-polyhedra $G(\mathcal{C})$ and  $G(\mathcal{C}')$ precludes $B^{+}$ as the image of $O_{g}^{+}$. We argue by contradiction. Suppose that $T({O_{g}'}^{+}) = B^{+}$.  Let $C$ be an oriented circle in the \textit{c}-face $\mathcal{C}_{f}$ that is not in the coaxial family determined by $C_{u}$ and $C_{v}$ so that its image $C' = T(C) \notin \mathcal{A}$. Since $G(\mathcal{C})$ is convex, $C$ is segregated from $O_{g}^{+}$ and therefore its image $T(C) = C'$ is segregated from the image $T(O_{g}^{+}) = B^{+}$. But also since $T(\mathcal{C}_{f}) = \mathcal{C}_{f}'$, $C' = T(C)$ is an oriented circle in $G(\mathcal{C}')$ and since the \textit{c}-polyhedron $G(\mathcal{C}')$ is convex, $C'$ is segregated from ${O_{g}'}^{+}$. Thus the oriented circle $C'= T(C)$ is orthogonal to ${O_{f}'}^{+}$ and segregated from both $B^{+}$ and $A^{+} = {O_{g}'}^{+}$. Since $C$ is orthogonal to the ortho-circle ${O_{f}'}^{+}$, this violates Lemma~\ref{Lem:3coaxialcircles} that is stated and proved below.

	We have verified that $T(O_{g}^{+}) = {O_{g}'}^{+}$. Now let $T' \in \text{M\"ob}(\mathbb{S}^{2})$ be a M\"obius transformation such that $T'(\mathcal{C}_{g}) = \mathcal{C}_{g}'$, taking corresponding oriented circles to corresponding oriented circles. Let $a$ and $b$ be the points of intersection of $C_{u}$ with the ortho-circle $O_{g}$ and $c$ and $d$ the points of intersection of $C_{v}$ with $O_{g}$. Since the maps $T$ and $T'$ both take $C_{u}$ to $C_{u}'$, $C_{v}$ to $C_{v}'$, and $O_{g}^{+}$ to ${O_{g}'}^{+}$, preserving all orientations, they must agree on the four points $a$, $b$, $c$, and $d$. But then $T=T'$ and we have $T(\mathcal{C}_{g}) = T'(\mathcal{C}_{g}) = \mathcal{C}_{g}'$.	
\end{proof}
\begin{Lemma}\label{Lem:3coaxialcircles}
	Let $\mathcal{A}$ be a hyperbolic or elliptic coaxial family. Let $O$, $A$, and $B$ be three pairwise distinct circles that belong to the orthogonal complement $\mathcal{A}^{\perp}$. Orient the circles to obtain  $O^{+}$, $A^{+}$ and $B^{+}$ and let $C$ be an oriented circle that does not belong to $\mathcal{A}$. Suppose that $\langle O^{+}, A^{+} \rangle = \langle O^{+}, B^{+} \rangle$. If $O$ and $C$ are orthogonal and $C$ is segregated from $A^{+}$, then $C$ is not segregated from $B^{+}$.
\end{Lemma}
\begin{proof}

\begin{figure}
\centering
\begin{subfigure}[t]{0.45\textwidth}
\begin{tikzpicture}[line cap=round,line join=round,>=triangle 45,x=1.0cm,y=1.0cm]
\clip(-4.5,-0.5) rectangle (2.3,4.5);
\draw(-2.1818017241379315,1.6928534482758617) circle (2.004901797087114cm);
\draw(-0.4736688436419516,1.7097656550134461) circle (1.0103268130050396cm);
\draw(0.5140275862068969,1.719544827586207) circle (1.3962952018472123cm);
\draw (-2.5079016479109884,3.6710571930095393)-- (-2.2356909424491627,3.5335364136737515);
\draw (-2.246768206836871,3.8326225521418826)-- (-2.5079016479109884,3.6710571930095393);
\draw (-1.5767315100972494,1.371700237798299)-- (-1.4806796406060723,1.6279761454501356);
\draw (-1.4806796406060723,1.6279761454501356)-- (-1.3157286264827226,1.4270644858377441);
\draw (1.7031729687369748,1.86471455643474)-- (1.827878926603495,2.1922338861010093);
\draw (1.827878926603495,2.1922338861010093)-- (2.0564540349884113,1.9530348229975991);
\begin{scriptsize}
\draw [fill=black] (-0.44,0.7) circle (1.5pt);
\draw[color=black] (-0.462564740187225,0.27908866667210125) node {a};
\draw [fill=black] (-0.46,2.72) circle (1.5pt);
\draw[color=black] (-0.48051619949058993,3.1333706959071224) node {b};
\draw[color=black] (-3.406604065939072,3.8334776087383537) node {$O^+$};
\draw[color=black] (-1.8627785658496891,2.4332637830758905) node {$B^+$};
\draw[color=black] (1.2248724343290776,3.2769823703340415) node {$A^+$};
\end{scriptsize}
\end{tikzpicture}
\caption{}
\end{subfigure}
\quad
\begin{subfigure}[t]{0.45\textwidth}
\begin{tikzpicture}[line cap=round,line join=round,>=triangle 45,x=1.0cm,y=1.0cm]
\clip(-5.,0.) rectangle (0.2,5.5);
\draw [line width=0.4pt,domain=-5.:0.2] plot(\x,{(--17.7568-0.*\x)/7.16});
\draw [domain=-5.:0.2] plot(\x,{(-2.1248-3.34*\x)/2.16});
\draw [domain=-5.:0.2] plot(\x,{(--10.3168--2.48*\x)/1.92});
\draw(-2.78,2.48) circle (1.42cm);
\draw [line width=0.4pt] (-2.78,2.48)-- (0.82,2.48);
\draw [line width=0.4pt] (-0.89,2.48) -- (-0.98,2.36);
\draw [line width=0.4pt] (-0.89,2.48) -- (-0.98,2.6);
\draw (-2.9991932335718934,1.4993754066363048)-- (-4.608770331815226,-0.5796616785946656);
\draw (-3.868259967957658,0.37683087472135973) -- (-3.9107294832214374,0.5425002450746614);
\draw (-3.868259967957658,0.37683087472135973) -- (-3.697234082165681,0.37721348296697893);
\draw (-2.926560311480798,3.541625666826789)-- (-4.120013146916795,5.387057366065785);
\draw (-3.580306220694571,4.5525105449629) -- (-3.409926549677437,4.537652291226568);
\draw (-3.580306220694571,4.5525105449629) -- (-3.636646908720157,4.391030741666007);
\begin{scriptsize}
\draw [fill=black] (-2.78,2.48) circle (1.5pt);
\draw[color=black] (-3.44,3.47) node {C};
\draw[color=black] (-0.64,2.09) node {$O^+$};
\draw[color=black] (-4.4,0.87) node {$A^+$};
\draw[color=black] (-4.4,4.95) node {$B^+$};
\end{scriptsize}
\end{tikzpicture}
\caption{}
\end{subfigure}

\caption{Case 1.}
\label{fig:3coaxialCase1}
\end{figure}
%\begin{figure}
%\includegraphics[width=0.5\textwidth]{NeededFigures/P14}
%\caption{Case 1.}
%\label{fig:3coaxialCase1}
%\end{figure}
\textit{Case 1: $\mathcal{A}$ is hyperbolic.}
	Since $\mathcal{A}$ is hyperbolic, its orthogonal complement $\mathcal{A}^{\perp}$ is elliptic and as the circles $O$, $A$, and $B$ belong to $\mathcal{A}^{\perp}$, they pass through two distinct common points, say $a$ and $b$. Fig.~\ref{fig:3coaxialCase1}(a) pictures the oriented circles $O^{+}$, $A^{+}$ and $B^{+}$ with $\langle O^{+}, A^{+} \rangle = \langle O^{+}, B^{+} \rangle$. Applying appropriate M\"obius transformations to normalize, we may assume that $a = 0$, $b= \infty$, $O^{+}$ is the real axis of $\widehat{\mathbb{C}}$ oriented towards the right, and $A^{+}$ and $B^{+}$ are oriented lines through the origin both meeting the real axis at the same acute angle, but oriented so that $\langle O^{+}, A^{+} \rangle = \langle O^{+}, B^{+} \rangle$, as in Fig.~\ref{fig:3coaxialCase1}(b). Since $C$ is orthogonal to $O$ and not a member of the coaxial family $\mathcal{A}$, it is either a circle in the normalized picture centered on the real axis at a point $x \neq 0$, or a vertical line orthogonal to the real axis. Now if $C$ is tangent to or disjoint from $A$, then it is tangent to or disjoint from $B$, and hence is contained in the companion disk of either $A^{+}$ or $B^{+}$. Therefore $C$ is coupled with one of $A^{+}$ and $B^{+}$ and cannot be segregated from both $A^{+}$ and $B^{+}$. We may assume then that the radius $r$ of $C$ is greater than $|x|/\sqrt{2}$ ($r=\infty$ when a $C$ is a line), forcing $C$ to intersect each of $A$ and $B$ in exactly two distinct points (one of which may be $\infty$).  Since  $x\neq 0$, the smaller angle of intersection of $C$ with both $A$ and $B$ is $\theta < \pi/2$. From the symmetry of the normalized picture of Fig.~\ref{fig:3coaxialCase1}(b), whether $C$ is oriented clockwise or counterclockwise, the oriented angle of $C$ with one of $A^{+}$ and $B^{+}$ is $\theta$, and with the other is $\pi - \theta$. Since $\theta < \pi/2$, we have $\pi - \theta > \pi/2$ so $C$ is not segregated from one of $A^{+}$ and $B^{+}$. We conclude that if $C$ is in fact segregated from $A^{+}$, then it cannot be segregated from $B^{+}$. 

\begin{figure}
\begin{tikzpicture}[line cap=round,line join=round,>=triangle 45,x=1.0cm,y=1.0cm]
\clip(-3.5,-3.5) rectangle (5.25,4.);
\draw (2.56,0.)-- (3.56,0.);
\draw (3.2203165345021394,0.) -- (3.06,-0.20612125864560785);
\draw (3.2203165345021394,0.) -- (3.06,0.20612125864560785);
\draw (0.,-3.1)-- (0.,3.48);
\draw (-3.,0.)-- (2.56,0.);
\draw(-0.74,-1.52) circle (1.1131380542330545cm);
\draw(-0.74,1.3) circle (1.1131380542330545cm);
\draw (0.3702370572051252,-1.439688079401289)-- (0.34710634720835687,-1.7593242938797773);
\draw (0.3471005365289329,-1.7594045898834434) -- (0.1530880408944277,-1.5846289736262087);
\draw (0.3471005365289329,-1.7594045898834434) -- (0.5642553635190545,-1.6143833996548584);
\draw (0.37172628260292906,1.2439553717862022)-- (0.36111963815884285,1.463131450807897);
\draw (0.35867378964856317,1.513672551642388) -- (0.5723032836820346,1.3635066308100363);
\draw (0.35867378964856317,1.513672551642388) -- (0.1605426370797367,1.3435801917840637);
\begin{scriptsize}
\draw[color=black] (4.04208769707304,0.10142627777782323) node {$O^+$};
\draw[color=black] (0.7441475587433157,2.910782691910552) node {$C$};
\draw [fill=black] (-0.74,1.3) circle (1.5pt);
\draw[color=black] (-1.2101873380446693,1.3839585537949386) node {$i$};
\draw [fill=black] (-0.74,-1.52) circle (1.5pt);
\draw[color=black] (-1.332333269093918,-1.3948613775754781) node {$-i$};
\draw[color=black] (-2.828620924447219,-1.4253978603377906) node {$B^+$};
\draw[color=black] (-2.950766855496468,1.3839585537949386) node {$A^+$};
\end{scriptsize}
\end{tikzpicture}
\caption{Case 2.}
\label{fig:3coaxialCase2}
\end{figure}
	\textit{Case 2: $\mathcal{A}$ is elliptic.}
	Since $\mathcal{A}$ is elliptic, its orthogonal complement $\mathcal{A}^{\perp}$ is hyperbolic and as the circles $O$, $A$, and $B$ belong to $\mathcal{A}^{\perp}$, after applying an appropriate M\"obius transformation to normalize, we may assume that $O^{+}$ is the real axis of $\widehat{\mathbb{C}}$ oriented toward the right and $A$ and $B$ are circles centered at $\pm i$ of equal radii $r<1$, as in Fig.~\ref{fig:3coaxialCase2}. To satisfy that $\langle O^{+}, A^{+} \rangle = \langle O^{+}, B^{+} \rangle$, one of $A^{+}$ and $B^{+}$ is oriented clockwise and the other counterclockwise.  Since $C$ is orthogonal to $O$, it is a circle in the normalized picture centered on $O$ or a vertical line orthogonal to the real axis. If a circle, we may apply an elliptic M\"obius transformation that set-wise fixes $O$, $A$, and $B$ so that $C$ is a vertical line. Since $C$ is not a member of the coaxial family $\mathcal{A}$, $C$ is a vertical line that misses the origin. The orientations on $A^{+}$ and $B^{+}$, one clockwise and the other counterclockwise, then imply that $C$ meets $A$ and $B$ in two distinct points, since otherwise it is coupled with one of $A^{+}$ and $B^{+}$. But now, just as in the preceding paragraph, the oriented angle of $C$ with one of $A^{+}$ and $B^{+}$ is $\theta < \pi/2$, and with the other is $\pi - \theta > \pi/2$, so $C$ is not segregated from one of $A^{+}$ and $B^{+}$. We conclude that if $C$ is in fact segregated from $A^{+}$, then it cannot be segregated from $B^{+}$. 
\end{proof}

The content of the Main Theorem is that when the convex \textit{c}-polyhedra are non-unitary and proper, whose definition awaits, item (ii) of Theorem~\ref{Thm:Congruence ConditionI} is superfluous as it will be seen to follow as a consequence of item (i).

\begin{figure}
\begin{subfigure}[t]{0.45\textwidth}
\definecolor{yqyqyq}{rgb}{0.5019607843137255,0.5019607843137255,0.5019607843137255}
\begin{tikzpicture}[line cap=round,line join=round,>=triangle 45,x=1.0cm,y=1.0cm]
\clip(-4.,-4.) rectangle (4.,4.);
\fill[color=yqyqyq,fill=yqyqyq,fill opacity=0.10000000149011612] (-0.7218962237322837,0.) -- (0.6261250253543001,0.) -- (0.8867891081587809,0.23537874901246347) -- (1.2945035955881155,0.5603655125553741) -- (1.7758802553415025,0.8863382763079066) -- (2.0335487789565776,1.0384108325295003) -- (2.3801788407282816,1.2208488040605854) -- (2.675248223212492,1.3575886505854433) -- (2.502342793666282,1.6547750732309288) -- (2.255277338613824,1.9783134549244075) -- (1.9360549475046303,2.2916568766381333) -- (1.5316162624663088,2.5795642315222036) -- (1.0902078781120599,2.7948965602509155) -- (0.6760305845130185,2.9228381153945193) -- (0.138923652347681,2.99678164349997) -- (-0.2806430977018987,2.9868443969702008) -- (-0.7742234758964897,2.898375063612499) -- (-1.1793449902354023,2.758467943262466) -- (-1.552258016727583,2.567195950741772) -- (-2.101588401791455,2.1408704279931654) -- (-2.3658957979099173,1.8445967237941727) -- (-2.6671691798099433,1.373393085126015) -- (-2.3068877622031008,1.2057424971093131) -- (-1.9483543537943855,1.006206562694461) -- (-1.6644893041717281,0.8223042933181905) -- (-1.34,0.58) -- (-0.9784918643399432,0.26245661218427285) -- cycle;
\draw(0.,0.) circle (3.cm);
\draw [shift={(-4.842141318528024,-3.7715739881440986)}] plot[domain=0.15264845104875174:1.1708433296897374,variable=\t]({1.*5.58580253761566*cos(\t r)+0.*5.58580253761566*sin(\t r)},{0.*5.58580253761566*cos(\t r)+1.*5.58580253761566*sin(\t r)});
\draw [shift={(5.510261289439234,-5.146787655542229)}] plot[domain=1.9818314319013437:2.798743928696315,variable=\t]({1.*7.095365404078314*cos(\t r)+0.*7.095365404078314*sin(\t r)},{0.*7.095365404078314*cos(\t r)+1.*7.095365404078314*sin(\t r)});
\draw (-3.,0.)-- (3.,0.);
\draw (0.14182821714502836,0.) -- (0.,-0.18235056490075102);
\draw (0.14182821714502836,0.) -- (0.,0.18235056490075102);
\draw (-1.6644893041717217,0.8223042933181861)-- (-1.34,0.58);
\draw (-1.3886036877517873,0.6162935913817428) -- (-1.6113485088710253,0.5550423353724266);
\draw (-1.3886036877517873,0.6162935913817428) -- (-1.3931407953006965,0.8472619579457592);
\draw (1.4341316737077845,0.6609121966495488)-- (1.7758802553415176,0.8863382763079164);
\draw (1.7233975082400972,0.8517193179139516) -- (1.7054126406556467,0.621407537416014);
\draw (1.7233975082400972,0.8517193179139516) -- (1.5045992883936532,0.9258429355414504);
\draw (-0.2755925938930677,1.9323377694549733) node[anchor=north west] {$P$};
\draw (-3.3823059218317844,2.283531450004568) node[anchor=north west] {$\ell_1$};
\draw (-3.7875293993890082,0.6356226412718549) node[anchor=north west] {$\ell_2$};
\draw (-1.6803673160914439,-2.822284367216461) node[anchor=north west] {$\ell_3$};
\end{tikzpicture}
\caption{}
\end{subfigure}
\quad
\begin{subfigure}[t]{0.45\textwidth}
\definecolor{ffffff}{rgb}{1.,1.,1.}
\definecolor{qqccqq}{rgb}{0.,0.8,0.}
\definecolor{yqyqyq}{rgb}{0.5019607843137255,0.5019607843137255,0.5019607843137255}
\begin{tikzpicture}[line cap=round,line join=round,>=triangle 45,x=1.0cm,y=1.0cm]
\clip(-4.,-4.) rectangle (4.,4.);
\fill[color=yqyqyq,fill=yqyqyq,fill opacity=0.10000000149011612] (-0.8143376073705848,0.) -- (0.8184509075423592,0.) -- (0.9081984683330226,0.6704884795897088) -- (1.018969788190466,1.1486924144430237) -- (1.1868301920651836,1.669637696179419) -- (1.2910168695083666,1.9313505129568593) -- (1.44526842191119,2.266324344169953) -- (1.5883387185570297,2.5370232493900073) -- (2.4966792621314027,1.6510298346352827) -- (2.250172988258842,1.9738359546130195) -- (1.9316730905199657,2.2864701887787344) -- (1.5281497681779292,2.573725925352301) -- (1.0877404197315017,2.788570894220757) -- (0.674500530140123,2.916222880311707) -- (0.13860922760619346,2.989999052647882) -- (-0.2800079205239747,2.980084297005146) -- (-0.7724711823730441,2.8918151955503673) -- (-1.176675788573821,2.7522247258167334) -- (-1.2047715689340701,2.7400423937415557) -- (-1.2802214797495213,2.705612650694989) -- (-1.337150167859901,2.6779350778921165) -- (-1.5378210539856745,2.56795896749368) -- (-1.3638774154213866,2.2297202124578512) -- (-1.2080608054641848,1.8692230427985765) -- (-1.0983122174470046,1.5646650511935005) -- (-0.923950157984626,0.9057716263100045) -- (-0.8933168745007469,0.7422329561163529) -- cycle;
\fill[color=yqyqyq,fill=yqyqyq,fill opacity=0.10000000149011612] (0.061022648425962,0.) -- (-0.8143376073705848,0.) -- (-0.8198994914419904,-0.5854779667550621) -- cycle;
\fill[color=ffffff,fill=ffffff,fill opacity=1.0] (0.061022648425962,0.) -- (0.8862692843350407,0.5484749452839581) -- (0.8184509075423592,0.) -- cycle;
\draw(0.,0.) circle (2.9932101250685386cm);
\draw [shift={(-6.5782833103832,-0.23800945923435257)}] plot[domain=-0.43563335952855375:0.5079639881721255,variable=\t]({1.*5.7688576486131655*cos(\t r)+0.*5.7688576486131655*sin(\t r)},{0.*5.7688576486131655*cos(\t r)+1.*5.7688576486131655*sin(\t r)});
\draw [shift={(6.974304277252246,-0.4827347842714867)}] plot[domain=2.630593196288861:3.514380164519066,variable=\t]({1.*6.174752106871515*cos(\t r)+0.*6.174752106871515*sin(\t r)},{0.*6.174752106871515*cos(\t r)+1.*6.174752106871515*sin(\t r)});
\draw (-1.0983122174470061,1.5646650511935065)-- (-0.923950157984626,0.9057716263100045);
\draw (-0.9776138597621632,1.10856035434609) -- (-1.1739771676659572,1.1921246313827736);
\draw (-0.9776138597621632,1.10856035434609) -- (-0.848285207765675,1.2783120461207371);
\draw (1.0638162258379298,1.3043563089307628)-- (1.1868301920651874,1.6696376961794304);
\draw (1.1671379491365477,1.611162955175451) -- (1.2849651480348725,1.4332351937458239);
\draw (1.1671379491365477,1.611162955175451) -- (0.9656812698682454,1.5407588113643684);
\draw (-0.2571572331015243,1.9112974770090485) node[anchor=north west] {$P$};
\draw (-2.0539723766245346,3.633245322885268) node[anchor=north west] {$\ell_1$};
\draw (-3.1769818413264157,-1.4078193853320702) node[anchor=north west] {$\ell_2$};
\draw (1.3150560174811097,-2.5308288500339526) node[anchor=north west] {$\ell_3$};
\draw (-2.4739992447514423,-1.684824794980805)-- (2.511391916195831,1.6285630156917625);
\draw [line width=1.2pt,color=qqccqq] (-0.8143376073705848,0.)-- (0.8184509075423592,0.);
\draw [color=yqyqyq] (0.061022648425962,0.)-- (-0.8143376073705848,0.);
\draw [color=yqyqyq] (-0.8198994914419904,-0.5854779667550621)-- (0.061022648425962,0.);
\draw [color=ffffff] (0.061022648425962,0.)-- (0.8862692843350407,0.5484749452839581);
\draw [line width=2.pt,color=qqccqq] (-0.8143376073705848,0.)-- (0.8184509075423592,0.);
\draw (-0.8198994914419904,-0.5854779667550621)-- (0.8862692843350407,0.5484749452839581);
\draw (-1.8071420802981402,-1.2416185692652273)-- (-0.8198994914419904,-0.5854779667550621);
\draw (-1.2044044188266618,-0.841027410802665) -- (-1.2202796837461634,-1.0538407399230927);
\draw (-1.2044044188266618,-0.841027410802665) -- (-1.406761887993968,-0.7732557960971964);
\begin{scriptsize}
\draw[color=ffffff] (0.8284185827769611,0.23926116289735716) node {$v_1$};
\end{scriptsize}
\end{tikzpicture}
\caption{}
\end{subfigure}
\caption{Improper green-black polygons.}
\label{fig:remedy}
\end{figure}
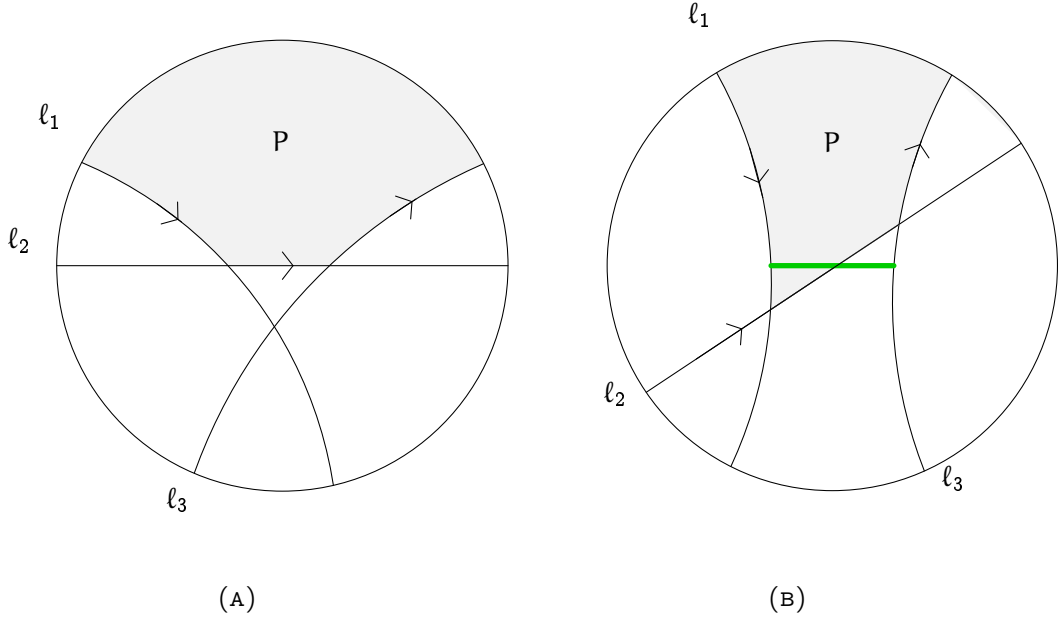

\subsection{Green-black polygons and proper \textit{c}-polyhedra}\label{Section:Green-black}
Fig.~\ref{fig:hyperidealpolygon6} might give the reader the impression that a convex hyperideal polygon may be ``completed'' to a compact convex hyperbolic polygon by adding the hyperideal vertices as new sides. Fig.~\ref{fig:remedy} should provide a remedy to this impression. Fig.~\ref{fig:remedy}(a) shows three oriented lines $\ell_{1}$, $\ell_{2}$, and $\ell_{3}$ that pairwise intersect at three finite vertices so there is no hyperideal vertex to exploit. With the orientations shown, the convex region $P$ is unbounded. Of course if the orientations are reversed on the three lines, then these oriented lines cut out a bounded convex polygon, a triangle. More serious is the example pictured in Fig.~\ref{fig:remedy}(b) in which the hyperideal vertex between $\ell_{1}$ and $\ell_{3}$ fails to lie in the convex region $P$ as it meets $\ell_{2}$ at a point separating its meeting points with $\ell_{1}$ and $\ell_{3}$. This inspires the next definition. 
\begin{Definition}[Proper hyperideal polygon]
The convex hyperideal polygon $P$ determined by the cyclically ordered oriented lines $\ell_{1}, \dots, \ell_{n}$ is said to be \textit{proper} provided the following two conditions are met.
\begin{enumerate}
\item Any hyperideal vertex, say for instance the hyperbolic line segment $s_{i,i+1}$ meeting the two consecutive lines $\ell_{i}$ and $\ell_{i+1}$ orthogonally, lies in the the convex hyperideal polygon $P$, as pictured in Fig.~\ref{fig:hyplines}.
\item The oriented lines $\ell_{1}, \dots, \ell_{n}$ along with any hyperideal vertices form the boundary of a bounded or compact convex polygon $P'$ contained in $P$, as in Fig.~\ref{fig:greenblackpoly}.
\end{enumerate}
\end{Definition}
We will construct a bounded convex hyperbolic polygon whose sides and vertices are colored either green or black from a proper convex hyperideal polygon.
\begin{Definition}[Green-black polygon]
	Let $P$ be a bounded convex hyperbolic polygon in $\mathbb{H}^{2}$ with vertices $p_{1}, \dots, p_{n}$ listed in cyclic order respecting orientation. We denote this by $P= p_{1}\cdots p_{n}$. Color each edge and each vertex of $P$ either green or black. Then $P$ is called a \textit{green-black polygon} provided
	\begin{enumerate}
		\item each green edge is adjacent to two black edges;
		\item each vertex incident to a green edge is colored black, and the remaining vertices are colored green;
		\item adjacent green and black edges are orthogonal.
	\end{enumerate}
\end{Definition}
The reason for the color coding is that the version of the Cauchy Arm Lemma we will prove in Section~\ref{sec:greenblackpolygons} considers green-black polygons for which the black side lengths are fixed while the green side lengths vary, and the black angles are fixed at right angles while the green ones vary---black indicates fixed measurements, green variable measurements.

We construct a green-black polygon from a proper convex hyperideal polygon $P$ with no ideal vertices by the following procedure. First color all of the edges of $P$, some of which are rays or lines in $\mathbb{H}^{2}$, black. For each non-intersecting pair of consecutive support lines $\ell_{i}, \ell_{i+1}$, add the hyperideal vertex $s_{i, i+1}$ as a green colored segment along the unique shortest path between $\ell_{i}$ and $\ell_{i+1}$ as in Fig.~\ref{fig:greenblackpoly}. Each line $\ell_{i}$ now supports a line segment, either a finite edge of $P$, or a truncation of a line or ray on the boundary of $P$ to a line segment. The resulting polygon $P'$ is compact and convex and every green edge has as neighbors two black edges, both of which it meets at right angles. Also color the vertices---if a vertex is adjacent to two black edges, color it green, otherwise black.
\begin{figure}
\centering
\begin{subfigure}[t]{0.3\textwidth}
	\centering
	\begin{tikzpicture}[line cap=round,line join=round,>=triangle 45,x=0.5cm,y=0.5cm]
\clip(-3.5,-3.5) rectangle (3.5,3.5);
\draw [line width=1.7pt] (0.,-0.02) circle (1.64cm);
\draw [shift={(3.1876595412987103,2.4888869320995317)},line width=1.0pt] plot[domain=2.866682812109121:4.750106932851553,variable=\t]({1.*2.3868990320691204*cos(\t r)+0.*2.3868990320691204*sin(\t r)},{0.*2.3868990320691204*cos(\t r)+1.*2.3868990320691204*sin(\t r)});
\draw [shift={(-2.611735369123523,2.455413256516329)},line width=1.0pt] plot[domain=4.377652143067875:6.05145783060787,variable=\t]({1.*1.4800109556579497*cos(\t r)+0.*1.4800109556579497*sin(\t r)},{0.*1.4800109556579497*cos(\t r)+1.*1.4800109556579497*sin(\t r)});
\draw [shift={(-1.1447914011904725,3.26)},line width=1.0pt] plot[domain=-1.5939397725248288:0.,variable=\t]({1.*1.1447914011904727*cos(\t r)+0.*1.1447914011904727*sin(\t r)},{0.*1.1447914011904727*cos(\t r)+1.*1.1447914011904727*sin(\t r)});
\draw [shift={(2.5328000152792263,-2.9092597645064924)},line width=1.0pt] plot[domain=1.267562174466263:2.640078689604578,variable=\t]({1.*2.001124160114651*cos(\t r)+0.*2.001124160114651*sin(\t r)},{0.*2.001124160114651*cos(\t r)+1.*2.001124160114651*sin(\t r)});
\draw [shift={(-0.7874568535451341,-4.2497979049415955)},line width=1.0pt] plot[domain=0.9736786681253566:1.9183851272700465,variable=\t]({1.*2.784399147543959*cos(\t r)+0.*2.784399147543959*sin(\t r)},{0.*2.784399147543959*cos(\t r)+1.*2.784399147543959*sin(\t r)});
\draw [shift={(-3.1525200300893337,-1.7028016674997803)},line width=1.0pt] plot[domain=0.04999814551398941:1.6529638491565948,variable=\t]({1.*1.4183807642006712*cos(\t r)+0.*1.4183807642006712*sin(\t r)},{0.*1.4183807642006712*cos(\t r)+1.*1.4183807642006712*sin(\t r)});
\end{tikzpicture}
\caption{A convex hyperideal $6$-gon}
\label{fig:hyppoly}
\end{subfigure}
\;
\begin{subfigure}[t]{0.3\textwidth}
	\centering
	\definecolor{qqccqq}{rgb}{0.,0.69215686274509803,0.2}
\begin{tikzpicture}[line cap=round,line join=round,>=triangle 45,x=0.5cm,y=0.5cm]
\clip(-3.5,-3.5) rectangle (3.5,3.5);
\draw [line width=1.7pt] (0.,-0.02) circle (1.64cm);
\draw [shift={(-1.1447914011904745,3.26)},line width=1.0pt] plot[domain=-2.469986497336053:0.,variable=\t]({1.*1.1447914011904736*cos(\t r)+0.*1.1447914011904736*sin(\t r)},{0.*1.1447914011904736*cos(\t r)+1.*1.1447914011904736*sin(\t r)});
\draw [shift={(-2.6117353691235228,2.4554132565163282)},line width=1.0pt] plot[domain=-1.905533164111711:0.38831871817246555,variable=\t]({1.*1.480010955657949*cos(\t r)+0.*1.480010955657949*sin(\t r)},{0.*1.480010955657949*cos(\t r)+1.*1.480010955657949*sin(\t r)});
\draw [shift={(-3.1525200300893315,-1.702801667499778)},line width=1.0pt] plot[domain=-0.6723293713358327:1.6529638491565963,variable=\t]({1.*1.4183807642006683*cos(\t r)+0.*1.4183807642006683*sin(\t r)},{0.*1.4183807642006683*cos(\t r)+1.*1.4183807642006683*sin(\t r)});
\draw [shift={(-0.7874568535451343,-4.249797904941598)},line width=1.0pt] plot[domain=0.5197946225862194:2.2536745215993577,variable=\t]({1.*2.784399147543961*cos(\t r)+0.*2.784399147543961*sin(\t r)},{0.*2.784399147543961*cos(\t r)+1.*2.784399147543961*sin(\t r)});
\draw [shift={(2.5328000152792267,-2.909259764506493)},line width=1.0pt] plot[domain=1.2675621744662633:3.3135307514135377,variable=\t]({1.*2.0011241601146517*cos(\t r)+0.*2.0011241601146517*sin(\t r)},{0.*2.0011241601146517*cos(\t r)+1.*2.0011241601146517*sin(\t r)});
\draw [shift={(3.1876595412987103,2.4888869320995317)},line width=1.0pt] plot[domain=2.866682812109121:4.750106932851553,variable=\t]({1.*2.3868990320691204*cos(\t r)+0.*2.3868990320691204*sin(\t r)},{0.*2.3868990320691204*cos(\t r)+1.*2.3868990320691204*sin(\t r)});
\draw [shift={(-3.667217832595169,0.4569294489406403)},line width=1.0pt,dotted,color=qqccqq]  plot[domain=-0.6438439087145857:0.37090862791196577,variable=\t]({1.*1.7080832330336952*cos(\t r)+0.*1.7080832330336952*sin(\t r)},{0.*1.7080832330336952*cos(\t r)+1.*1.7080832330336952*sin(\t r)});
\draw [shift={(0.6224598461281102,3.477252036412751)},line width=1.0pt,dotted,color=qqccqq]  plot[domain=3.962242862889318:4.8637080421743715,variable=\t]({1.*1.3637551342652905*cos(\t r)+0.*1.3637551342652905*sin(\t r)},{0.*1.3637551342652905*cos(\t r)+1.*1.3637551342652905*sin(\t r)});
\draw [shift={(3.7312779859576266,-0.47264848672367893)},line width=1.0pt,dotted,color=qqccqq]  plot[domain=2.6675970598300136:3.42672191285405,variable=\t]({1.*1.8354634458428318*cos(\t r)+0.*1.8354634458428318*sin(\t r)},{0.*1.8354634458428318*cos(\t r)+1.*1.8354634458428318*sin(\t r)});
\end{tikzpicture}
	\caption{The green dotted segments denote the hyperideal vertices between consecutive non-intersecting lines.}
\label{fig:hyplines}
\end{subfigure}
\;
\begin{subfigure}[t]{0.3\textwidth}
	\centering
	\definecolor{qqccqq}{rgb}{0.,0.69215686274509803,0.2}
\begin{tikzpicture}[line cap=round,line join=round,>=triangle 45,x=0.5cm,y=0.5cm]
\clip(-3.5,-3.5) rectangle (3.5,3.5);
\draw [line width=1.7pt] (0.,-0.02) circle (1.64cm);
\draw [shift={(-3.667217832595169,0.4569294489406403)},line width=1.0pt,color=qqccqq]  plot[domain=-0.6438439087145857:0.37090862791196577,variable=\t]({1.*1.7080832330336952*cos(\t r)+0.*1.7080832330336952*sin(\t r)},{0.*1.7080832330336952*cos(\t r)+1.*1.7080832330336952*sin(\t r)});
\draw [shift={(0.622459846128109,3.4772520364127475)},line width=1.0pt,color=qqccqq]  plot[domain=3.9622428628893167:4.863708042174372,variable=\t]({1.*1.3637551342652872*cos(\t r)+0.*1.3637551342652872*sin(\t r)},{0.*1.3637551342652872*cos(\t r)+1.*1.3637551342652872*sin(\t r)});
\draw [shift={(3.7312779859576266,-0.47264848672367893)},line width=1.0pt,color=qqccqq]  plot[domain=2.6675970598300136:3.42672191285405,variable=\t]({1.*1.8354634458428318*cos(\t r)+0.*1.8354634458428318*sin(\t r)},{0.*1.8354634458428318*cos(\t r)+1.*1.8354634458428318*sin(\t r)});
\draw [shift={(-0.7874568535451341,-4.2497979049415955)},line width=1.0pt]  plot[domain=0.9736786681253566:1.9183851272700465,variable=\t]({1.*2.784399147543959*cos(\t r)+0.*2.784399147543959*sin(\t r)},{0.*2.784399147543959*cos(\t r)+1.*2.784399147543959*sin(\t r)});
\draw [shift={(-2.611735369123523,2.455413256516329)},line width=1.0pt]  plot[domain=5.083297608296656:6.05145783060787,variable=\t]({1.*1.4800109556579497*cos(\t r)+0.*1.4800109556579497*sin(\t r)},{0.*1.4800109556579497*cos(\t r)+1.*1.4800109556579497*sin(\t r)});
\draw [shift={(-1.144791401190471,3.26)},line width=1.0pt]  plot[domain=4.689245534654757:5.533039189684212,variable=\t]({1.*1.1447914011904687*cos(\t r)+0.*1.1447914011904687*sin(\t r)},{0.*1.1447914011904687*cos(\t r)+1.*1.1447914011904687*sin(\t r)});
\draw [shift={(3.1876595412987063,2.4888869320995304)},line width=1.0pt]  plot[domain=3.292911715379478:4.238393386624912,variable=\t]({1.*2.3868990320691164*cos(\t r)+0.*2.3868990320691164*sin(\t r)},{0.*2.3868990320691164*cos(\t r)+1.*2.3868990320691164*sin(\t r)});
\draw [shift={(2.5328000152792254,-2.9092597645064915)},line width=1.0pt]  plot[domain=1.8559255860591526:2.6400786896045783,variable=\t]({1.*2.0011241601146503*cos(\t r)+0.*2.0011241601146503*sin(\t r)},{0.*2.0011241601146503*cos(\t r)+1.*2.0011241601146503*sin(\t r)});
\draw [shift={(-3.1525200300893332,-1.7028016674997797)},line width=1.0pt]  plot[domain=0.049998145513988956:0.9269524180803116,variable=\t]({1.*1.4183807642006707*cos(\t r)+0.*1.4183807642006707*sin(\t r)},{0.*1.4183807642006707*cos(\t r)+1.*1.4183807642006707*sin(\t r)});
\end{tikzpicture}
\caption{The green-black polygon defined by the convex hyperideal polygon in fig.~\ref{fig:hyppoly}}
\label{fig:greenblackpoly}
\end{subfigure}
\caption{Construction of a green-black polygon from a proper convex hyperideal polygon.}
\label{fig:greenblackpolygonconstruction}
\end{figure}
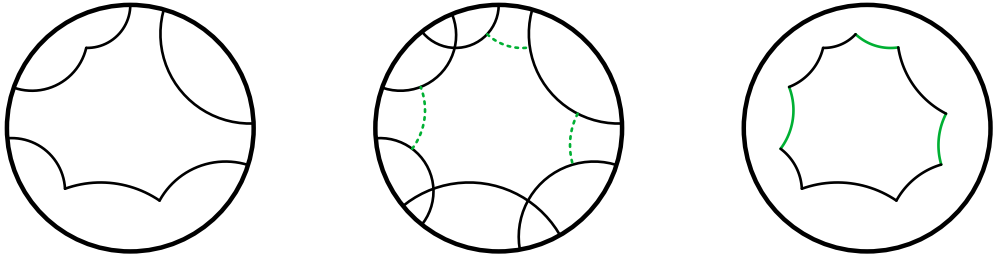

The analogue of boundedness for \textit{c}-polyhedra turns out to be a bit more subtle than one might at first suspect. The vertices of bounded polyhedra in $\mathbb{E}^{3}$ have compact links while the links of vertices at infinity of unbounded polyhedra may fail to be compact. Moreover, convex bounded polyhedra have convex links. We need a condition to force this behavior on convex \textit{c}-polyhedra.\footnote{One can in fact describe convex circle configurations that mimic unbounded polyhedra in $\mathbb{E}^{3}$ and fail to be globally rigid.} The subtlety arises precisely because we allow adjacent circles to be separated. When dealing with circle packings with ortho-circles where adjacent circles meet, either tangent or overlapping in an angle in the range $[0,\pi/2]$, the analogue of compact and convex links is satisfied automatically. When adjacent circles are allowed to separate, we need a condition that guarantees the behavior that the boundedness of convex Euclidean polyhedra forces upon neighborhoods of the vertices of the polyhedra. For this we introduce the next definition, which serves to force compact behavior around the vertices. 
\begin{Definition}[Proper \textit{c}-polyhedron, and \textit{c}-links]
Let $G(\mathcal{C})$ be a non-unitary convex \textit{c}-polyhedron based on the oriented abstract spherical polyhedron $P$ with vertex set $V=V(P)$. For any vertex $v\in V$, give the interior of the companion disk $D_{v}$ to $C_{v}$ a complete hyperbolic metric of constant curvature $-1$ making $D_{v}$ a model of the hyperbolic plane with its circle at infinity. Let $f_{1}, \dots, f_{n}$ be the faces adjacent to $v$ ordered cyclically about $v$ with respect to the orientation of $P$. Let $\ell_{i}$ be the hyperbolic line in $D_{v}$ determined by the orthogonal intersection $D_{v} \cap O_{f_{i}}^{+}$, but oriented oppositely to that of the ortho-circle $O_{f_{i}}^{+}$. Then $G(\mathcal{C})$ is \textit{proper at} $v$ if the oriented lines $\ell_{1}\dots, \ell_{n}$ are the support lines of a proper convex hyperideal polygon $P(v)$ in $D_{v}$. The \textit{c}-polyhedron $G(\mathcal{C})$ then is \textit{proper} provided it is proper at each of its vertices. The green-black polygon $L(v)$ determined by the proper hyperideal polygon $P(v)$ is the \textit{c-link} of the vertex circle $C_{v}$.
\end{Definition}

Consistent orientation of our \textit{c}-polyhedron $G(\mathcal{C})$ with its oriented abstract spherical polyhedron $P$ implies that lines $\ell_{i}$ in the definition of properness are ordered cyclically about $P(v)$, meaning that the cyclic ordering that $P(v)$ gives to the support lines is exactly that given to the faces adjacent to $v$ by the orientation of $P$. The relationship between convexity and properness is a bit subtle. We suspect that if the \textit{c}-polyhedron $G(\mathcal{C})$ is edge segregated so that it avoids deep overlaps, then convexity of $G(\mathcal{C})$ implies its properness. We have yet to verify our suspicions.

%This notion of properness serves double duty. It guarantees compact linkages about the vertices, but it also, in that $\ell_{1}, \dots, \ell_{n}$ is cyclically ordered with respect to both the ordering of the faces adjacent to $v$ and with respect to the orientation of $P(v)$, guarantees that faces lie about $C_{v}$ in the right order. We suspect the explicit requirement of this ordering condition is unnecessary, that it is a consequence of the fact that $P(v)$ is proper. This is but one of several natural questions that arise about the relationships between convexity and properness. For another, we suspect that if the \textit{c}-polyhedron $G(\mathcal{C})$ is edge segregated so that it avoids deep overlaps, then convexity of $G(\mathcal{C})$ implies its properness. We have yet to verify our suspicions. 

\subsection{A congruence condition for \textit{c}-polyhedra}
Two green-black polygons are {\em compatible} when they have the same number of green and black edges given in the same order.  When also every pair of corresponding black edges have the same hyperbolic length, the polygons are {\em black-edge-congruent}. Note here that the complex interior angles---the real angles at green vertices and the imaginary unit times the lengths of the green edges---may differ between two black-edge-congruent polygons. We use the plain term {\em congruent} to mean the corresponding complex angles also agree. 

We end this rather lengthy preliminary section with an important observation. If the non-unitary \textit{c}-polyhedron $G(\mathcal{C})$ is both convex and proper, then the \textit{c}-links are M\"obius invariants. This means that if $T\in \text{M\"ob}(\mathbb{S}^{2})$ and $v$ is a vertex of $G$,   then the \textit{c}-link of $C_{v}$ in $G(\mathcal{C})$ and that of $T(C_{v})$ in $T(G(\mathcal{C}))$ are congruent. This is obvious from the fact that the M\"obius transformation $T$ restricts to an isometry of $D^{\circ}_{v}$ with $T(D^{\circ}_{v})$ when both open disks are given the complete hyperbolic metric as in the definition of \textit{c}-link. Here $D^{\circ}$ means the open interior of the closed disk $D$. This goes the other direction as well and allows us to replace condition (ii) in Theorem~\ref{Thm:Congruence ConditionI} by an equivalent one on \textit{c}-links, as in Corollary~\ref{CongruenceConditionII} below.  
\begin{Lemma}\label{Lem:black-edgeCongruent}
	Let $G(\mathcal{C})$ and $G(\mathcal{C}')$ be two proper, convex, non-unitary \textit{c}-polyhedra both based on the abstract spherical polyhedron $P$. If $G(\mathcal{C})$ and $G(\mathcal{C}')$ have M\"obius-congruent \textit{c}-faces, then corresponding \textit{c}-links are black-edge-congruent.
\end{Lemma}
\begin{proof} 
\begin{figure}
\definecolor{qqccqq}{rgb}{0.,0.8,0.}
\definecolor{uuuuuu}{rgb}{0.26666666666666666,0.26666666666666666,0.26666666666666666}
\definecolor{ccqqqq}{rgb}{0.8,0.,0.}
\begin{tikzpicture}[line cap=round,line join=round,>=triangle 45,x=1.5cm,y=1.5cm]
\clip(-0.5,-1.75) rectangle (5.25,2.25);
\draw (0.,0.)-- (2.9748704687854977,0.);
\draw [dash pattern=on 1pt off 1pt,color=ccqqqq] (1.5571486301158692,0.) circle (1.4655733330540457cm);
\draw(0.8775867882917218,-0.8296144826179065) circle (0.6631288716114426cm);
\draw(2.702341935537465,-0.5822398615103799) circle (1.251261152776472cm);
\draw (0.,0.)-- (4.499187934422767,0.);
\draw [shift={(2.702341935537466,0.)},line width=1.2pt,color=qqccqq]  plot[domain=0.:3.141592653589793,variable=\t]({1.*0.5973635207034871*cos(\t r)+0.*0.5973635207034871*sin(\t r)},{0.*0.5973635207034871*cos(\t r)+1.*0.5973635207034871*sin(\t r)});
\draw [shift={(3.673093107571656,0.)},dash pattern=on 1pt off 1pt]  plot[domain=0.9677063691655743:2.4787473407614624,variable=\t]({1.*0.7651892982383443*cos(\t r)+0.*0.7651892982383443*sin(\t r)},{0.*0.7651892982383443*cos(\t r)+1.*0.7651892982383443*sin(\t r)});
\draw [shift={(4.29715532147481,0.)},dash pattern=on 1pt off 1pt]  plot[domain=1.5057434244221606:2.7576635596104087,variable=\t]({1.*1.4787111144837157*cos(\t r)+0.*1.4787111144837157*sin(\t r)},{0.*1.4787111144837157*cos(\t r)+1.*1.4787111144837157*sin(\t r)});
\draw [shift={(-0.22299834024373993,0.)},dash pattern=on 1pt off 1pt]  plot[domain=0.567788602614441:1.1864559066145806,variable=\t]({1.*1.3054148226541504*cos(\t r)+0.*1.3054148226541504*sin(\t r)},{0.*1.3054148226541504*cos(\t r)+1.*1.3054148226541504*sin(\t r)});
\draw [shift={(-1.2676354845266227,0.)},dash pattern=on 1pt off 1pt]  plot[domain=0.316260367092939:0.9131536181482248,variable=\t]({1.*2.2571661068356597*cos(\t r)+0.*2.2571661068356597*sin(\t r)},{0.*2.2571661068356597*cos(\t r)+1.*2.2571661068356597*sin(\t r)});
\draw (-0.31573653570558763,2.1983751472769346) node[anchor=north west] {$O_g$};
\draw (-0.32746205467202816,1.553471604122706) node[anchor=north west] {$O'_g$};
\draw (0.32916700744864186,0.9202935799349178) node[anchor=north west] {$p_u$};
\draw [color=ccqqqq](1.5017189040926955,1.5182950472233843) node[anchor=north west] {$O$};
\draw (2.322505231743533,1.1196274023644066) node[anchor=north west] {$p_w$};
\draw (2.0528182955154004,0.005703100552557145) node[anchor=north west] {$a$};
\draw (3.1198405214614895,0.029154138485438186) node[anchor=north west] {$b$};
\draw (3.330899862857419,-1.0613191253935301) node[anchor=north west] {$C_w$};
\draw (0.17673526088491487,-1.0730446443599708) node[anchor=north west] {$C_u$};
\draw (4.608981430199438,0.22848796091492704) node[anchor=north west] {$C_v$};
\draw (4.151686190508257,0.8382149471698341) node[anchor=north west] {$O'_g$};
\draw (4.468275202602151,1.7410799075857541) node[anchor=north west] {$O_g$};
\draw [shift={(1.5571486301158692,0.)},line width=1.2pt,color=ccqqqq]  plot[domain=0.5487566059437886:2.339946533230891,variable=\t]({1.*0.9770488887026971*cos(\t r)+0.*0.9770488887026971*sin(\t r)},{0.*0.9770488887026971*cos(\t r)+1.*0.9770488887026971*sin(\t r)});
\begin{scriptsize}
\draw [fill=uuuuuu] (0.8775867882917218,0.7020115626196962) circle (1.5pt);
\draw [fill=uuuuuu] (2.104978414833979,0.) circle (1.5pt);
\draw [fill=uuuuuu] (3.299705456240953,0.) circle (1.5pt);
\draw [fill=uuuuuu] (2.390741113058071,0.5096548864647824) circle (1.5pt);
\draw [fill=uuuuuu] (0.8775867882917218,-0.7020115626196963) circle (1.5pt);
\end{scriptsize}
\end{tikzpicture}
\caption{The construction of the points $p_{u}$ and $p_{w}$.}
\label{fig:Puandpw}
\end{figure}
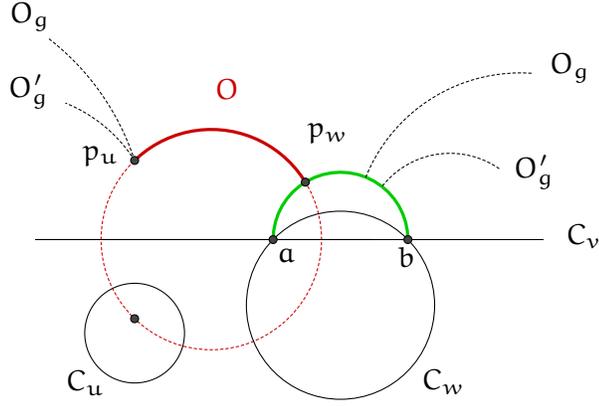
We show that the corresponding black edges of the links $L(v)$ in $G(\mathcal{C})$ and $L(v)'$ in $G(\mathcal{C}')$ have the same length. Let $u$ and $w$ be vertices of adjacent edges $uv$ and $vw$ of $P$ that lie on the face $f$. First by applying an appropriate M\"obius transformation, assume that the \textit{c}-polyhedra $G(\mathcal{C})$ and $G(\mathcal{C}')$ have been placed in a \textit{normalized} position where $C_{u} = C_{u}'$, $C_{v} = C_{v}'$, and $C_{w} = C_{w}'$. Let $O$ be the unique circle orthogonal to $C_{u}$, $C_{v}$, and $C_{w}$ and note that the normalization implies that $O_{f} = O = O_{f}'$. Our claim is that the black edges of $L(v)$ and $L(v)'$ that lie along $O$ are completely determined by the circles $C_{u} = C_{u}'$, $C_{v} = C_{v}'$, and $C_{w} = C_{w}'$. To see this, we define a point $p_{u}$ in the open disk $D^{\circ} = D_{v}^{\circ} = D_{v}^{\prime \circ}$ as follows. The circles $C_{u}$ and $C_{v}$ determine a unique coaxial family $\mathcal{A}_{u}$ and the ortho-circle $O$ is a member of the complementary family $\mathcal{A}_{u}^{\perp}$. There are two cases to consider depending on whether or not $C_{u}$ and $C_{v}$ meet. When $C_{u}$ is disjoint from $C_{v}$, the family $\mathcal{A}_{u}$ is a hyperbolic coaxial family with two foci, one of which is in $D^{\circ}$. This one we call $p_{u}$. On the other hand, when $C_{u}$ and $C_{v}$ meet at two points $a$ and $b$ on $C_{v}$, the coaxial family $\mathcal{A}_{u}$ is elliptic. In this case, let $\lambda$ be the hyperbolic line in $D^{\circ}$ with ideal vertices $a$ and $b$. Note that the circle containing $\lambda$ is the unique member of $\mathcal{A}_{u}$ that is orthogonal to $C_{v}$. Let $p_{u}$ be the point of intersection of $O$ with $\lambda$ in the disk $D^{\circ}$. The observation we make is that in both \textit{c}-links $L(v)$ and $L(v)'$, the black edge determined by $O_{f} = O_{f}'$ is the edge $e= p_{u}p_{w}$, the hyperbolic segment connecting $p_{u}$ and $p_{w}$. This follows from the observation that both ortho-circles $O_{g}$ and $O_{g}'$ for the respective \textit{c}-polyhedra $G(\mathcal{C})$ and $G(\mathcal{C}')$ that are determined by the face $g$ adjacent to $f$ along the edge $uv$ are members of the complementary coaxial family $\mathcal{A}_{u}^{\perp}$. In particular, in the case that $\mathcal{A}$ is hyperbolic, both pass through the same point $p_{u}$, and in the case that $\mathcal{A}$ is elliptic, the unique green edge from $O$ to both $O_{g}$ and $O_{g}'$ lies along the hyperbolic line $\lambda$ and so meets $O$ at $p_{u}$. See Fig.~\ref{fig:Puandpw} for these cases.
\end{proof}

\begin{Remark}\label{rem:complexangleL(v)}
Note that the complex angle at a green vertex or a green edge of a \textit{c}-link $L(v)$ is the complex dihedral angle between the two adjacent \textit{c}-faces $\mathcal{C}_{f}$ and $\mathcal{C}_{g}$ incident to $C_{v}$.
\end{Remark}
This observation along with the preceding lemma gives the following corollary.

\begin{Corollary}\label{CongruenceConditionII}
	Let $G(\mathcal{C})$ and $G(\mathcal{C}')$ be two proper, convex, non-unitary \textit{c}-polyhedra both based on the abstract spherical polyhedron $P$. Then $G(\mathcal{C})$ and $G(\mathcal{C}')$ are M\"obius-congruent with one another if and only if the following conditions hold.
	\begin{enumerate}
		\item[(i)] $G(\mathcal{C})$ and $G(\mathcal{C}')$ have M\"obius-congruent \textit{c}-faces. 
		\item[(ii)] The complex angles at the corresponding green (finite and hyperideal) vertices of corresponding \textit{c}-links agree; equivalently, by the preceding lemma, corresponding \textit{c}-links are congruent.
	\end{enumerate}
\end{Corollary}

% poincare extension, equator fixing acts as MT on disk.

%% file: sec_cauchy.tex
\section{Cauchy's Rigidity Theorem}\label{Section:CRT}

In this section we review Cauchy's celebrated rigidity theorem on the uniqueness of convex, bounded  polyhedra in $\mathbb{E}^{3}$. This serves as the model for our proof of the Main Theorem. In fact, the bulk of the remainder of this paper is dedicated to an appropriate generalization of Cauchy's Arm Lemma from the context of Euclidean and spherical geometry to that of M\"obius geometry.
\subsection{Cauchy's argument recalled}
Suppose we have two convex polyhedra in Euclidean $3$-space $\mathbb{E}^3$ with the same combinatorics and with corresponding faces congruent. Cauchy's theorem states that the two polyhedra must be congruent, meaning that there is a Euclidean isometry mapping one to the other.

Cauchy's proof is split into two components---the one geometric and the other combinatorial. The geometric component is the Discrete Four Vertex Lemma (Lemma~\ref{lem:cauchyFourVertex}), which follows from an application of Cauchy's Arm Lemma (Lemma~\ref{lem:cauchyArm}), presented after a bit of notation. We denote a planar or spherical polygon $P$ merely by listing its vertices in cyclic order, say as $P = p_{1} \dots p_{n}$. The Euclidean or spherical length of the side $p_{i}p_{i+1}$ is denoted as $|p_{i}p_{i+1}|$ and the interior angle at $p_{i}$ is denoted as $\angle p_{i}$, though on occasion, for either emphasis or to avoid confusion, we will use $\angle p_{i-1}p_{i}p_{i+1}$.

\begin{Lemma}[Cauchy's Arm Lemma]\label{lem:cauchyArm}
	Let $P = p_1 \dots p_n$ and $P' = p'_1 \dots p'_n$ be two convex (planar or spherical) polygons such that, for $1\leq i < n$, $|p_i p_{i+1}| = |p_i' p_{i+1}'|$, and for $1\leq i < n-1$, $\angle p_{i+1}  \leq \angle  p'_{i+1} $. Then $|p_n p_1| \leq |p_n' p_1'|$ with equality if and only if $\angle  p_{i+1}  = \angle  p'_{i+1} $ for all $1\leq i < n - 1$. 
\end{Lemma}

Cauchy's original proof of the lemma had a gap that subsequently was filled. A straightforward inductive proof, such as the one in \cite{FuchsTab:2007}, relies on the law of cosines and the triangle inequality. Imagine now that we have convex planar or spherical polygons $P$ and $P'$ with the same number of sides with corresponding sides of equal length. Label each vertex of $P$ with a plus sign $+$ or a minus sign $-$ by comparing its angle with the corresponding angle in $P'$: if the angle at $p_{i}$ is larger than that at $p_{i}'$, label it with a $+$, if smaller, a $-$, and if equal, no label at all. Using the Cauchy Arm Lemma, the proof of the following lemma is straightforward. 

\begin{Lemma}[Discrete Four Vertex Lemma]\label{lem:cauchyFourVertex}
	Let $P$ and $P'$ be as in the preceding paragraph and label the vertices of $P$ as described. Then either $P$ and $P'$ are congruent (and thus no vertex is labeled with $+$ or $-$), or a walk around $P$ encounters at least four sign changes (from $-$ to $+$ or from $+$ to $-$).
\end{Lemma}

\begin{proof}
	First note that because a polygon is a cycle the number of sign changes must be even. Suppose there are no sign changes, i.e., either no vertex is labeled or all of the labeled vertices have the same label. If no vertex is labeled, then the two polygons are congruent. Assume then that some of the vertices are labeled, but all with the same label. Then Cauchy's Arm Lemma implies that there exists a pair of corresponding edges in $P$ and $P'$ with different lengths, a contradiction. 
	
	Assume now that there are exactly two sign changes of the labels of $P$.  Select two edges $p_i p_{i+1}$ and $p_j p_{j+1}$ (oriented counter-clockwise) of $P$ such that all of the $+$ signs are along the subchain from $p_{i+1}$ to $p_j$ and all of the $-$ signs are along the subchain from $p_{j+1}$ back to $p_i$. Subdivide both edges in two by adding a vertex at the respective midpoints $X$ and $Y$ of $p_i p_{i+1}$ and $p_j p_{j+1}$. Similarly, subdivide the corresponding edges $p'_i p'_{i+1}$ and $p'_j p'_{j+1}$ in $P'$ at midpoints $X'$ and $Y'$. Denote the subchain of $P$ from $X$ to $Y$ by $P_{+}$ and the subchain from $Y$ back to $X$ by $P_{-}$. Similarly for $P'_{+}$ and $P'_{-}$ in $P'$. Applying the arm lemma to $P_+$ and $P_+'$ implies that $|X Y| > |X' Y'|$, and, similarly, an application to $P_{-}$ and $P'_{-}$ implies that $|X Y| < |X' Y'|$, a contradiction. 
\end{proof}

This brings us to the combinatorial component of Cauchy's proof. A  nice proof of the following lemma appears in~\cite{FuchsTab:2007} and follows from an argument based on the Euler characteristic of a sphere.

\begin{Lemma}[Cauchy's Combinatorial Lemma]\label{lem:cauchycombinatorial}
	Let $P$ be an abstract spherical polyhedron. Then for any labeling of any non-empty subset of the edges of $P$ with $+$ and $-$ signs, there exists a vertex $v$ that is incident to an edge labeled with a $+$ or a $-$ sign for which one encounters at most two sign changes in labels on the edges adjacent to $v$ as one walks around the vertex.
\end{Lemma}

The proof of Cauchy's Rigidity Theorem is now straightforward. Assume we have convex polyhedra $P$ and $P'$ that have the same combinatorics and congruent corresponding faces. For each edge of $P$ label its dihedral angle with a $+$ or a $-$ depending on whether it is larger or smaller than the corresponding dihedral angle in $P'$. If $P$ and $P'$ are not congruent, Cauchy's Combinatorial Lemma provides a vertex $v$ that is incident to an edge labeled with a $+$ or a $-$ sign, and around which there are at most two sign changes. Intersect $P$ with a small sphere centered at $v$ (one that contains no other vertex of $P$ on its interior) to obtain a convex spherical polygon, and intersect $P'$ with a sphere centered at the corresponding vertex $v'$ and of the same radius. By construction both spherical polygons have the same edge lengths, and the angles between edges are given by the dihedral angles between faces at $v$ and $v'$. An application of the Four Vertex Lemma implies that there are at least four sign changes, contradicting that there are at most two. We conclude that $P$ and $P'$ are congruent. 

\subsection{Observations towards a proof for \textit{c}-polyhedra} The remainder of this paper is concerned with developing a Cauchy-style proof of the global rigidity of convex, proper, non-unitary \textit{c}-polyhedra. To that end we make the following observations. First, Cauchy's Combinatorial Lemma is valid for abstract polyhedral graphs, which means that it holds for \textit{c}-polyhedra. Second, though it is straightforward to extend the Four Vertex Lemma to convex hyperbolic polygons, doing so would only get us halfway to our desired result. Instead we generalize the Four Vertex Lemma to the class of green-black polygons.  

Finally, the crucial point of connection between the combinatorial and geometric components of Cauchy's proof is made by intersecting small spheres centered at the vertices of the polyhedra to obtain spherical polygons, the spherical links of the vertices. In the case of \textit{c}-polyhedra it is not at all obvious how to go about obtaining a similar construction as there is no obvious ``sphere'' with which to intersect a \textit{c}-polyhedron to obtain the analogue of a link of a vertex. But this is precisely where the \textit{c}-links of circles in a \textit{c}-polyhedron come into play, which play the role of links of vertices of Euclidean polyhedra. Given our construction of \textit{c}-links, and our version of the Four Vertex Lemma for green-black polygons, the remainder of the proof is essentially Cauchy's original argument. 

%% file: sec_greenblackpolygonsNEW.tex
\section{The Green-Black Arm and Four Vertex Lemmas}\label{sec:greenblackpolygons}
\subsection{An important property} 

An important property of green-black polygons is stated in the next lemma. 

\begin{Lemma}\label{lem:greenblackcontainment}
Let $P$ be a green-black polygon with ordered vertices $p_1, \dots, p_n$, some of which are green with the remaining ones black. Let $p_i p_{i+1}$ be a green edge. Then the shortest path from $p_i p_{i+1}$ to a vertex $p_j$ meets $p_i p_{i+1}$ at a right angle. Similarly, if $p_j p_{j+1}$ is another green edge, then the shortest path between $p_i p_{i+1}$ and $p_j p_{j+1}$ meets both green edges at right angles. 	
\end{Lemma}

\begin{proof}
\begin{figure}
\definecolor{qqwuqq}{rgb}{0.,0.39215686274509803,0.}
\definecolor{qqccqq}{rgb}{0.,0.8,0.}
\begin{tikzpicture}[line cap=round,line join=round,>=triangle 45,x=1.0cm,y=1.0cm]
\clip(-2.,-1.) rectangle (5.25,4.);
\draw [shift={(1.9394786234077395,2.1176536693744743)},color=qqwuqq,fill=qqwuqq,fill opacity=0.10000000149011612] (0,0) -- (-100.77019484592059:0.5016865568647408) arc (-100.77019484592059:-4.151347413809074:0.5016865568647408) -- cycle;
\draw [line width=2.pt,color=qqccqq] (0.,0.)-- (3.52,0.);
\draw [shift={(-5.8220231213872795,0.03913294797687904)}] plot[domain=-0.0067214364915599845:0.5097291484842211,variable=\t]({1.*5.8221546366946875*cos(\t r)+0.*5.8221546366946875*sin(\t r)},{0.*5.8221546366946875*cos(\t r)+1.*5.8221546366946875*sin(\t r)});
\draw [shift={(-7.191707060063219,-0.2510221285563737)}] plot[domain=0.03132166879572527:0.49547359629527593,variable=\t]({1.*8.015638597472556*cos(\t r)+0.*8.015638597472556*sin(\t r)},{0.*8.015638597472556*cos(\t r)+1.*8.015638597472556*sin(\t r)});
\draw [shift={(10.293837382240689,0.26246817148672774)}] plot[domain=2.7258047763626405:3.1722672891853776,variable=\t]({1.*8.557863255676724*cos(\t r)+0.*8.557863255676724*sin(\t r)},{0.*8.557863255676724*cos(\t r)+1.*8.557863255676724*sin(\t r)});
\draw [shift={(8.455070422535208,0.18746478873239517)}] plot[domain=2.52748928177286:3.1738192429936767,variable=\t]({1.*5.818091359377083*cos(\t r)+0.*5.818091359377083*sin(\t r)},{0.*5.818091359377083*cos(\t r)+1.*5.818091359377083*sin(\t r)});
\draw [shift={(7.133333333333329,-0.04666666666666546)}] plot[domain=2.4201593970929074:3.1286782424502406,variable=\t]({1.*3.613634673781444*cos(\t r)+0.*3.613634673781444*sin(\t r)},{0.*3.613634673781444*cos(\t r)+1.*3.613634673781444*sin(\t r)});
\draw [shift={(-2.1037384074687586,4.328577621576414)}] plot[domain=4.853184530829097:5.678720290449778,variable=\t]({1.*2.8770469526602005*cos(\t r)+0.*2.8770469526602005*sin(\t r)},{0.*2.8770469526602005*cos(\t r)+1.*2.8770469526602005*sin(\t r)});
\draw [shift={(2.4246026908253855,6.256465945568261)},line width=2.pt,color=qqccqq]  plot[domain=4.167174405273455:4.843425735949782,variable=\t]({1.*4.167146795873675*cos(\t r)+0.*4.167146795873675*sin(\t r)},{0.*4.167146795873675*cos(\t r)+1.*4.167146795873675*sin(\t r)});
\draw [shift={(5.447677008600061,2.4287467824118942)}] plot[domain=3.263515450557904:4.416692563291038,variable=\t]({1.*2.497123348841284*cos(\t r)+0.*2.497123348841284*sin(\t r)},{0.*2.497123348841284*cos(\t r)+1.*2.497123348841284*sin(\t r)});
\draw [shift={(-3.4058091536138586,1.6928373235353353)},dash pattern=on 1pt off 1pt]  plot[domain=-0.46128008898560857:0.26623349516403105,variable=\t]({1.*3.80331883948667*cos(\t r)+0.*3.80331883948667*sin(\t r)},{0.*3.80331883948667*cos(\t r)+1.*3.80331883948667*sin(\t r)});
\draw [shift={(5.258648972833552,1.5846704453301186)},dash pattern=on 1pt off 1pt]  plot[domain=2.9098174872212534:3.8806911789669023,variable=\t]({1.*2.3524626821775354*cos(\t r)+0.*2.3524626821775354*sin(\t r)},{0.*2.3524626821775354*cos(\t r)+1.*2.3524626821775354*sin(\t r)});
\draw (-1.7,1.48)-- (-1.5612437602217841,1.6253734856487896);
\draw (-1.5370855896286926,1.3998972267799363)-- (-1.7,1.48);
\draw (-0.7352027126083227,2.8714010748348113)-- (-0.7318132365256445,2.656122097620691);
\draw (-0.538547871780913,2.776912950586148)-- (-0.7352027126083227,2.8714010748348113);
\draw (-0.1385207582551038,3.5572616210187755)-- (-0.14396441876041946,3.3083927036341594);
\draw (0.065406393046373,3.437236280130647)-- (-0.1385207582551038,3.5572616210187755);
\draw (2.465118005293456,3.7190816037167123)-- (2.2637999170176935,3.566079856627135);
\draw (2.4973288994175777,3.48555262131683)-- (2.465118005293456,3.7190816037167123);
\draw (3.695053853292478,3.5329736179828797)-- (3.4556029996102047,3.388919938944464);
\draw (3.6569210878859666,3.2600763624479763)-- (3.695053853292478,3.5329736179828797);
\draw (4.42,2.34)-- (4.164242670340887,2.2212750269450448);
\draw (4.365560758616649,2.0843787269175262)-- (4.42,2.34);
\draw (4.72,0.04)-- (4.574931570423441,0.256410485373608);
\draw (4.454140717457984,0.0067760559116631855)-- (4.72,0.04);
\draw (0.1742774571207839,2.5712948075451303)-- (0.31121438391270795,2.498424200456576);
\draw (0.31121438391270795,2.498424200456576)-- (0.3780516325838259,2.626487821065224);
\draw (2.989501624900678,1.9894304928178237)-- (2.8442917870757563,1.9665805109890278);
\draw (2.8284913828164906,2.1089382069217857)-- (2.8442917870757563,1.9665805109890278);
\draw (-0.01009482043263731,0.17134200267427918)-- (0.17751404798472528,0.17915903885833592);
\draw (0.17751404798472528,0.17915903885833592)-- (0.18533108416878205,0.);
\draw (0.8234927846498713,0.18049033606657264)-- (0.9720164721469501,0.18049033606657264);
\draw (0.9720164721469501,0.18049033606657264)-- (0.9689457237422988,0.011504554793939208);
\draw (1.7437450543690916,0.18049033606657264)-- (1.8922687418661701,0.18049033606657264);
\draw (1.8922687418661701,0.18049033606657264)-- (1.8891979934615188,0.011504554793939215);
\draw (2.6454063691444896,0.18978581353848392)-- (2.7939300566415683,0.18978581353848392);
\draw (2.7939300566415683,0.18978581353848392)-- (2.790859308236917,0.02080003226585053);
\draw (3.5159065510459997,0.17882314928834184)-- (3.37024965990007,0.18049033606657264);
\draw (3.37024965990007,0.18049033606657264)-- (3.3671789114954187,0.011504554793939215);
\draw [fill=black] (0.,0.) circle (1.5pt);
\draw[color=black] (-0.10489664274272599,-0.341588666505697) node {$p_i$};
\draw [fill=black] (3.52,0.) circle (1.5pt);
\draw[color=black] (3.6577525337428307,-0.341588666505697) node {$p_{i+1}$};
\draw [fill=black] (1.74,0.) circle (1.5pt);
\draw[color=black] (1.62592197844063,-0.341588666505697) node {$\sigma(t)$};
\draw[color=black] (-0.9577637894127854,2.710337887754797) node {$r_i$};
\draw[color=black] (0.5800670288931902,3.3959761821366063) node {$r_{i+a}$};
\draw[color=black] (2.904548075699823,3.412699067365431) node {$r_{i+t}$};
\draw[color=black] (3.9748127303446033,3.094964248017763) node {$r_{i+b}$};
\draw[color=black] (4.727342565641714,1.9243622820000392) node {$r_{i+1}$};
\draw [fill=black] (0.26351383005755746,2.6934886295830127) circle (1.5pt);
\draw[color=black] (0.6971272254949631,2.9279040514159907) node {$p_{j+1}$};
\draw [fill=black] (2.969090759162815,2.1250442578439968) circle (1.5pt);
\draw[color=black] (3.4396804030222133,2.3257115274918303) node {$p_j$};
\draw[color=black] (-1.3925588053622273,1.1) node {$r_{j+1}$};
\draw[color=black] (4.661125647472312,0.653423004609368) node {$r_j$};
\draw [fill=black] (1.9394786234077395,2.1176536693744743) circle (1.5pt);
\draw[color=black] (1.6510063062838671,2.5765548059241996) node {$q_t$};
\draw[color=qqwuqq] (2.195512010339514,1.3753262166397314) node {$\alpha(t)$};
\end{tikzpicture}
\caption{The construction for Lemma~\ref{lem:greenblackcontainment}.}
\label{fig:GreenEdges}
\end{figure}
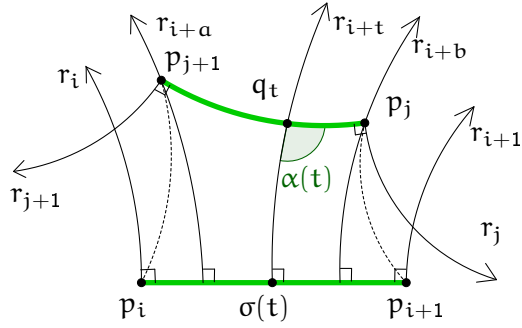
Since the edge $p_i p_{i+1}$ is green, the respective edges $p_{i-1} p_i$ and $p_{i+1} p_{i+2}$ are black and meet the support line $l_i$ of the edge $p_i p_{i+1}$ orthogonally at the respective points $p_{i}$ and $p_{i+1}$. Let $\sigma$ be a unit time parameterization of $p_{i}p_{i+1}$ with $\sigma (0) = p_{i}$ and $\sigma (1) = p_{i+1}$. Extend rays $r_i$ and $r_{i+1}$ outwards from $p_i$ and $p_{i+1}$ along the edges  $p_{i-1} p_i$ and $p_{i+1} p_{i+2}$ and observe that the convex region $R_{i}$ bounded by $r_i$, $r_{i+1}$, and $p_i p_{i+1}$ is the disjoint union of rays $r_{i+t}$ for $0\leq t \leq 1$, where $r_{i+t}$ is the ray in the region $R_{i}$ orthogonal to the segment $p_i p_{i+1}$ at the point $\sigma(t)$. This implies that the nearest point of the edge $p_{i}p_{i+1}$ to any point $q$ of $R_{i}$ is that point $\sigma(t)$ for which $q \in r_{i+t}$, and in this case the shortest path from $q$ to $p_{i}p_{i+1}$ follows the hyperbolic segment from $\sigma(t)$ to $q$, a subsegment of the ray $r_{i+t}$, which therefore meets $p_{i}p_{i+1}$ orthogonally.

Now let $p_{j}p_{j+1}$ be another green edge and extend rays $r_j$ and $r_{j+1}$ outwards along the black edges incident to $p_j$ and $p_{j+1}$. By convexity of the polygon $P$, $p_j$ and $p_{j+1}$ are contained in the convex region $R_i$, and $p_i$ and $p_{i+1}$ are contained in the convex region $R_j$ bounded by $r_j$, $r_{j+1}$, and $p_j p_{j+1}$. This implies that all four interior angles of the hyperbolic quadrilateral $p_{i}p_{i+1}p_{j}p_{j+1}$ are at most $\pi/2$; see Fig.~\ref{fig:GreenEdges}. Let $0\leq a < b \leq 1$ be the $t$-values for which $p_{j+1} \in r_{i+a}$ and $p_{j} \in r_{i+b}$, and for $a\leq t \leq b$, let $q_{t}$ be the point on the edge $p_{j}p_{j+1}$ that lies on the ray $r_{i+t}$. Finally, let $\alpha (t)$ be the measure of the angle $\angle \sigma(t) q_{t} p_{j}$; again see Fig~\ref{fig:GreenEdges}. By our construction and the observation that the angles of the quadrilateral $p_{i}p_{i+1}p_{j}p_{j+1}$ are acute or right, $\alpha (a) \leq \pi/2$ and $\alpha (b)  \geq \pi/2$. By continuity of the function $\alpha$, there is a parameter value $a \leq t \leq b$ such that $\alpha(t) = \pi/2$. The hyperbolic segment from $\sigma(t)$ to $q_{t}$, which lies on the ray $r_{i+t}$, meets both edges $p_{i}p_{i+1}$ and $p_{j}p_{j+1}$ orthogonally. This implies that this segment is the shortest path connecting edges $p_{i}p_{i+1}$ and $p_{j}p_{j+1}$.
\end{proof}

\subsection{Relaxed green-black polygons}

We generalize the concept of green-black polygons slightly. Instead of requiring that green edges meet black edges at right angles, we relax this requirement and allow the measures of these angles to be less than or equal to $\pi / 2$. Call these polygons {\em relaxed green-black polygons}. The proof of Lemma~\ref{lem:greenblackcontainment} may be modified to verify the following.

\begin{Corollary}\label{cor:relaxedgreenblackcontainment}
	Let $P$ be a relaxed green-black polygon with vertices $p_1, \dots, p_n$. Let $p_i p_{i+1}$ be a green edge. Then the shortest path from $p_i p_{i+1}$ to any vertex $p_j$ meets $p_i p_{i+1}$ at a right angle. Similarly, if $p_j p_{j+1}$ is a green edge, then the shortest path between $p_i p_{i+1}$ and $p_j p_{j+1}$ meets both at right angles. 
\end{Corollary}

\subsection{A four vertex lemma for green-black polygons}\label{sec:fourvertex}
When $X$ and $Y$ are points in the hyperbolic plane $\mathbb{H}^{2}$, we use $|XY|$ to denote the hyperbolic distance between the points. We now prove an analog of Cauchy's Discrete Four Vertex Lemma for green-black polygons. Let $P$ and $P'$ be two black-edge-congruent green-black polygons. We label a green vertex (or a green edge) of $P$ with a plus sign $+$ if the angle (or edge-length) is larger than the corresponding angle (or edge-length) in $P'$, a minus sign $-$ if it is smaller, or no label if equal. The statement of the lemma is now the same as Lemma~\ref{lem:cauchyFourVertex}. 

\begin{Lemma}[Green-Black Polygon Four Vertex Lemma]\label{lem:greenblackfourvertex}
Let $P$ and $P'$ be two black-edge-congruent green-black polygons labeled as in the preceding paragraph. Then either $P$ and $P'$ are congruent (in the usual sense), or a walk around $P$ encounters at least four sign changes (from $-$ to $+$ or from $+$ to $-$). 
\end{Lemma}

\begin{proof}
	The proof follows the same sketch as in Cauchy's original. New are the details of extending the arm lemma to green-black polygons, which is accomplished by Lemma~\ref{lem:greenblackarm} below. As before the number of sign changes must be even, so we assume there are either two or zero sign changes. Denote the vertices of $P$ in (counter-clockwise) order by $p_1, \dots, p_n$ and the corresponding vertices of $P'$ by $p'_1, \dots, p'_n$. 

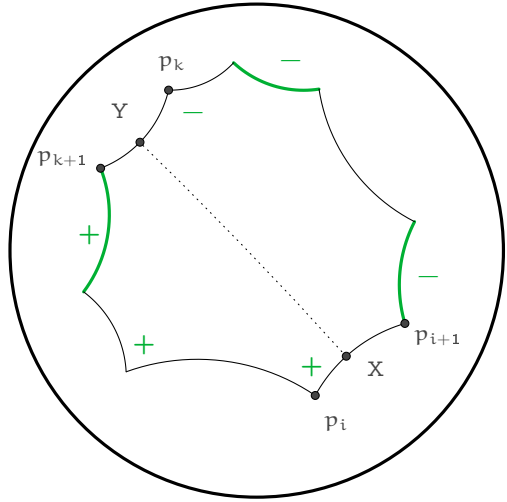
\begin{figure}
\definecolor{greenedge}{rgb}{0.,0.69215686274509803,0.2}

\definecolor{qqccqq}{rgb}{0.,0.69215686274509803,0.2}
\definecolor{uuuuuu}{rgb}{0.26666666666666666,0.26666666666666666,0.26666666666666666}
\begin{tikzpicture}[line cap=round,line join=round,>=triangle 45,x=1.0cm,y=1.0cm]
\clip(-3.5,-3.5) rectangle (3.5,3.5);
\draw [line width=1.2pt] (0.,-0.02) circle (3.28cm);
\draw [shift={(-3.667217832595169,0.4569294489406403)},line width=1.2pt,color=qqccqq]  plot[domain=-0.6438439087145857:0.37090862791196577,variable=\t]({1.*1.7080832330336952*cos(\t r)+0.*1.7080832330336952*sin(\t r)},{0.*1.7080832330336952*cos(\t r)+1.*1.7080832330336952*sin(\t r)});
\draw [shift={(0.622459846128109,3.4772520364127475)},line width=1.2pt,color=qqccqq]  plot[domain=3.9622428628893167:4.863708042174372,variable=\t]({1.*1.3637551342652872*cos(\t r)+0.*1.3637551342652872*sin(\t r)},{0.*1.3637551342652872*cos(\t r)+1.*1.3637551342652872*sin(\t r)});
\draw [shift={(3.7312779859576266,-0.47264848672367893)},line width=1.2pt,color=qqccqq]  plot[domain=2.6675970598300136:3.42672191285405,variable=\t]({1.*1.8354634458428318*cos(\t r)+0.*1.8354634458428318*sin(\t r)},{0.*1.8354634458428318*cos(\t r)+1.*1.8354634458428318*sin(\t r)});
\draw [shift={(-2.611735369123523,2.455413256516329)}] plot[domain=5.083297608296656:6.05145783060787,variable=\t]({1.*1.4800109556579497*cos(\t r)+0.*1.4800109556579497*sin(\t r)},{0.*1.4800109556579497*cos(\t r)+1.*1.4800109556579497*sin(\t r)});
\draw [shift={(-1.1447914011904767,3.26)}] plot[domain=4.689245534654762:5.5330391896842075,variable=\t]({1.*1.1447914011904818*cos(\t r)+0.*1.1447914011904818*sin(\t r)},{0.*1.1447914011904818*cos(\t r)+1.*1.1447914011904818*sin(\t r)});
\draw [shift={(3.187659541298706,2.48888693209953)}] plot[domain=3.292911715379478:4.238393386624912,variable=\t]({1.*2.3868990320691155*cos(\t r)+0.*2.3868990320691155*sin(\t r)},{0.*2.3868990320691155*cos(\t r)+1.*2.3868990320691155*sin(\t r)});
\draw [shift={(2.532800015279226,-2.909259764506492)}] plot[domain=1.8559255860591528:2.640078689604578,variable=\t]({1.*2.0011241601146508*cos(\t r)+0.*2.0011241601146508*sin(\t r)},{0.*2.0011241601146508*cos(\t r)+1.*2.0011241601146508*sin(\t r)});
\draw [shift={(-0.7874568535451342,-4.249797904941598)}] plot[domain=0.9736786681253572:1.9183851272700463,variable=\t]({1.*2.7843991475439607*cos(\t r)+0.*2.7843991475439607*sin(\t r)},{0.*2.7843991475439607*cos(\t r)+1.*2.7843991475439607*sin(\t r)});
\draw [shift={(-3.1525200300893315,-1.7028016674997792)}] plot[domain=0.04999814551398871:0.9269524180803124,variable=\t]({1.*1.418380764200669*cos(\t r)+0.*1.418380764200669*sin(\t r)},{0.*1.418380764200669*cos(\t r)+1.*1.418380764200669*sin(\t r)});
\draw [shift={(-45.222381048938445,-43.38669514833759)},dotted]  plot[domain=0.7350722314458603:0.7982657677826219,variable=\t]({1.*62.569765828427364*cos(\t r)+0.*62.569765828427364*sin(\t r)},{0.*62.569765828427364*cos(\t r)+1.*62.569765828427364*sin(\t r)});
\draw (-1.128186224785348,2.083877593375715) node[anchor=north west] {\color{greenedge}$-$};
\draw (0.17156644600438709,2.7752353969872714) node[anchor=north west] {\color{greenedge}$-$};
\draw (2.000331232550146,-0.10081306603680337) node[anchor=north west] {\color{greenedge}$-$};
\draw (0.44576387959347405,-1.3069855518152937) node[anchor=north west] {\color{greenedge}+};
\draw (-1.781272467674597,-1.0027881182262089) node[anchor=north west] {\color{greenedge}+};
\draw (-2.5109018320084706,0.4522731768524418) node[anchor=north west] {\color{greenedge}+};
\begin{scriptsize}
\draw [fill=uuuuuu] (-1.1712834537667531,2.115515171184019) circle (1.5pt);
\draw[color=uuuuuu] (-1.128186224785348,2.4752353969872714) node {$p_k$};
\draw [fill=uuuuuu] (0.7781028422065807,-1.9472120863712383) circle (1.5pt);
\draw[color=uuuuuu] (1.0565044346271855,-2.313158037593784) node {$p_i$};
\draw [fill=uuuuuu] (1.9699208216319182,-0.9889303871460553) circle (1.5pt);
\draw[color=uuuuuu] (2.4115657297058455,-1.1793312396708313) node {$p_{i+1}$};
\draw [fill=uuuuuu] (-2.0752870147006344,1.0760454332347607) circle (1.5pt);
\draw[color=uuuuuu] (-2.579173392730783,1.1989396047529228) node {$p_{k+1}$};
\draw [fill=uuuuuu] (1.1907646329181594,-1.4248612544406971) circle (1.5pt);
\draw[color=uuuuuu] (1.581936365371972,-1.5941459218377654) node {X};
\draw [fill=uuuuuu] (-1.5518303161886475,1.422440019868509) circle (1.5pt);
\draw[color=uuuuuu] (-1.8195440283969093,1.8349887840755548) node {Y};
\end{scriptsize}
\end{tikzpicture}

\caption{The construction for the contradiction in the proof of Lemma~\ref{lem:greenblackfourvertex}.}
\label{fig:fourvertex}
\end{figure}

Suppose there are exactly two sign changes. Pick a pair of black edges $p_i p_{i+1}$ and $p_k p_{k+1}$ of $P$ such that all of the $-$ signs occur on the counter-clockwise walk from $p_i$ to $p_{k+1}$, and all the $+$ signs occur on the counter-clockwise walk from $p_k$ to $p_{i+1}$. Let $X$, $Y$, $X'$, and $Y'$ be the respective midpoints of the edges $p_i p_{i+1}$,$p_k p_{k+1}$,$p'_i p'_{i+1}$, and $p'_k p'_{k+1}$. Fig.~\ref{fig:fourvertex} shows the construction in $P$. By applying the green-black arm lemma, Lemma~\ref{lem:greenblackarm} below, to the chains $X,p_{i+1},p_{i+2},\dots,p_k,Y$ and $X',p'_{i+1},p'_{i+2}, \dots , p'_k, Y'$, we have $|XY| < |X'Y'|$. Similarly, by applying the green-black arm lemma to $Y, p_{k+1},p_{k+2},\dots,p_i, X$ and  $Y,p_{k+1},p_{k+2},\dots,p_i, X$, we have that $|XY| > |X'Y'|$, a contradiction. Hence there cannot be two sign changes on $P$.
	
Suppose then there are no sign changes but some green angles or edges of $P$ are marked, say with a $-$ sign. Choose any black edge $p_i p_{i+1}$ and choose two distinct points $X$ and $Y$ on the open edge from $p_{i}$ to $p_{i+1}$ ordered as $p_{i}, X,Y,p_{i+1}$. Since the green-black polygons $P$ and $P'$ are black-edge congruent, $|p_{i}p_{i+1}| = |p_{i}'p_{i+1}'|$ and we may choose corresponding points $X'$ and $Y'$ on the open edge from $p_{i}'$ to $p_{i+1}'$, so that $|p_{i}X| = |p_{i}'X'|$, $|XY| = |X'Y'|$, and $|Yp_{i+1}| = |Y' p_{i+1}'|$.  By applying the green-black arm lemma to the chains $Y, p_{i+1}, p_{i+2}, \dots, p_i, X$ and $Y', p'_{i+1}, p'_{i+2}, \dots, p'_i$, X, we have that $|XY| < |X'Y'|$, a contradiction. Hence either $P$ and $P'$ are congruent, or $P$ contains at least four sign changes. 
\end{proof}

\subsection{The green-black arm lemma}In this section we state and prove the Green-Black Arm Lemma, beginning with a definition.

\begin{Definition}[Green-black arm chain] Any subchain of consecutive edges of a relaxed green-black polygon that begins and ends with black edges is called a \textit{green-black arm chain}. The two vertices of unit valence are called the \textit{free vertices}. We define two green-black arm chains to be {\em compatible} and {\em black-edge congruent} in the same way as for green-black polygons. We say that the green-black arm chain is \textit{determined} by its vertices and we use the list of its vertices in order to name the chain. 
\end{Definition}
Notice that Corollary~\ref{cor:relaxedgreenblackcontainment} applies to green-black arm chains, and that connecting the two free vertices of a green-black arm chain with a black segment produces a relaxed green-black polygon, which always is convex. This latter observation, that a green-black arm chain produces a convex polygon when the free vertices are connected by a segment, is important in the proof.
%A \textit{polygonal chain} in $\mathbb{H}^{2}$ is just a finite ordered list of points in $\mathbb{H}^{2}$. The polygonal chain $p_1, p_2, \dots, p_n$ is \textit{convex} if $P = p_1 p_2 \cdots p_n$ is a convex hyperbolic polygon, and the hyperbolic segment $p_{i}p_{i+1}$ is called the $i$th \text{edge} of the chain. The \textit{internal angle} at $p_{i}$ of the convex polygonal chain $p_1, p_2, \dots ,p_n$ is just the interior angle of the convex polygon $P= p_{1}\cdots p_{n}$ at $p_{i}$. Color each edge of the convex polygonal chain $p_1, p_2, \dots ,p_n$ either black or green. We say that this colored chain of edges determines a {\em green-black arm chain} provided 
%\begin{enumerate}
%	\item no two green edges are incident at a vertex;
%	\item the internal angle between any black edge and an adjacent green edge does not exceed $\pi / 2$;
%	\item the first and last edges are black. 
%\end{enumerate}
%Every green-black arm chain is a subchain of the relaxed green-black polygon obtained by adding a black edge from $p_{n}$ back to $p_{1}$. In particular, this means that Corollary~\ref{cor:relaxedgreenblackcontainment} applies to green-black arm chains as well. 

\begin{Lemma}[Green-Black Arm Lemma]\label{lem:greenblackarm}
	Let $p_1, p_2, \dots, p_n$ and $p'_1, p'_2, \dots, p'_n$ determine the two respective compatible, black-edge congruent, green-black arm chains $P$ and $P'$ such that:
\begin{enumerate}
%\item the edges $p_1 p_2$ and $p_{n-1} p_n$ are black,
\item if edge $p_i p_{i+1}$ is green, then $|p_i p_{i+1}| \leq |p'_i p'_{i+1}|$, 
\item if $\angle p_i$ is green, then $\angle p_i  \leq \angle  p'_i $, and otherwise, $\angle  p_i  = \angle  p'_i$,
\item if $p_i p_{i+1}$ is green, then $\angle  p_i$ equals $\pi /2$. 
\end{enumerate}
Then $|p_1 p_n| \leq |p_1' p_n'|$, with equality if and only if all of the corresponding angles and edges are congruent. 
\end{Lemma}

\subsubsection{How not to prove the Green-Black Arm Lemma}
To prove the lemma, we will induct on $N$, the number of green edges of the green-black arm chain. Before the proof, though, we think it instructive to discuss how this induction argument may lead to a subtle mistake if one is not careful. If $N=0$, The Green-Black Arm Lemma is just the classical Cauchy Arm Lemma in the hyperbolic plane. Assume then that the lemma holds for $N-1$ green edges and consider a green-black arm chain with $N$ green edges. Choose a green edge $e$ and lengthen it continuously keeping all other edge-lengths and angles constant. Then the free vertices $p_{1}$ and $p_{n}$ of the arm chain increase their distance. Now change the color of $e$ from green to black and apply the inductive hypothesis. QED?

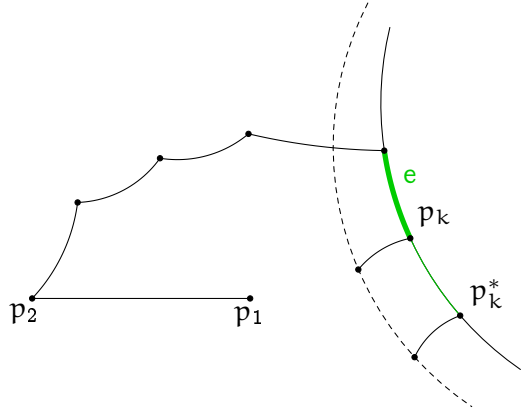
\begin{figure}
\definecolor{qqccqq}{rgb}{0.,0.8,0.}
\begin{tikzpicture}[line cap=round,line join=round,>=triangle 45,x=1.0cm,y=1.0cm]
\clip(-0.5,-1.75) rectangle (7.,4.25);
\draw (0.,0.)-- (2.8978970642754205,0.);
\draw [shift={(-1.6668796107256107,1.570600368565681)}] plot[domain=5.527517327304325:6.154703729122803,variable=\t]({1.*2.290256132923001*cos(\t r)+0.*2.290256132923001*sin(\t r)},{0.*2.290256132923001*cos(\t r)+1.*2.290256132923001*sin(\t r)});
\draw [shift={(0.5544505099798689,2.6840804197465795)}] plot[domain=4.747947108996147:5.662503645760434,variable=\t]({1.*1.4078167836761917*cos(\t r)+0.*1.4078167836761917*sin(\t r)},{0.*1.4078167836761917*cos(\t r)+1.*1.4078167836761917*sin(\t r)});
\draw [shift={(1.9321382217162244,3.3169748203163)}] plot[domain=4.553607976848908:5.4095049694918735,variable=\t]({1.*1.4701597990442734*cos(\t r)+0.*1.4701597990442734*sin(\t r)},{0.*1.4701597990442734*cos(\t r)+1.*1.4701597990442734*sin(\t r)});
\draw [shift={(4.808643703667634,10.41299558177477)}] plot[domain=4.481553943625864:4.697280278014923,variable=\t]({1.*8.44724515632759*cos(\t r)+0.*8.44724515632759*sin(\t r)},{0.*8.44724515632759*cos(\t r)+1.*8.44724515632759*sin(\t r)});
\draw [shift={(5.269203213056361,-0.384375920364719)}] plot[domain=1.772472467164058:2.4521353709337834,variable=\t]({1.*1.2091401113615265*cos(\t r)+0.*1.2091401113615265*sin(\t r)},{0.*1.2091401113615265*cos(\t r)+1.*1.2091401113615265*sin(\t r)});
\draw [shift={(6.059055672741732,-1.2411589202292965)}] plot[domain=1.9201090052153136:2.701104244127561,variable=\t]({1.*1.078229983220804*cos(\t r)+0.*1.078229983220804*sin(\t r)},{0.*1.078229983220804*cos(\t r)+1.*1.078229983220804*sin(\t r)});
\draw [shift={(8.94007485245214,2.595431923674416)}] plot[domain=2.9045864866942277:4.107438847030718,variable=\t]({1.*4.3052084449017505*cos(\t r)+0.*4.3052084449017505*sin(\t r)},{0.*4.3052084449017505*cos(\t r)+1.*4.3052084449017505*sin(\t r)});
\draw [shift={(8.940074852452163,2.595431923674424)},line width=2.pt,color=qqccqq]  plot[domain=3.2881532320137667:3.571709857250105,variable=\t]({1.*4.305208444901776*cos(\t r)+0.*4.305208444901776*sin(\t r)},{0.*4.305208444901776*cos(\t r)+1.*4.305208444901776*sin(\t r)});
\draw [shift={(8.176557658387166,2.0236651549709816)},dash pattern=on 2pt off 2pt]  plot[domain=2.6671932844493353:4.119771475085445,variable=\t]({1.*4.175401036391082*cos(\t r)+0.*4.175401036391082*sin(\t r)},{0.*4.175401036391082*cos(\t r)+1.*4.175401036391082*sin(\t r)});
\draw [shift={(8.940074852452174,2.595431923674442)},color=qqccqq]  plot[domain=3.5717098572501076:3.8568718470072274,variable=\t]({1.*4.305208444901793*cos(\t r)+0.*4.305208444901793*sin(\t r)},{0.*4.305208444901793*cos(\t r)+1.*4.305208444901793*sin(\t r)});
\draw [fill=black] (0.,0.) circle (1.0pt);
\draw[color=black] (-0.10119772587020531,-0.19616906087263028) node {$p_2$};
\draw [fill=black] (2.8978970642754205,0.) circle (1.0pt);
\draw[color=black] (2.8884955132091332,-0.23377526513777927) node {$p_1$};
\draw [fill=black] (0.6044992918951508,1.2771535503503149) circle (1.0pt);
\draw [fill=black] (1.6996844020413246,1.8653085169102976) circle (1.0pt);
\draw [fill=black] (2.875994335161289,2.1898078088054604) circle (1.0pt);
\draw [fill=black] (4.681021646328129,1.966714545627538) circle (1.0pt);
\draw [fill=black] (4.33624114868952,0.38478049763861927) circle (1.0pt);
\draw [fill=black] (5.08374972570614,-0.7814218144902223) circle (1.0pt);
\draw [fill=black] (5.026998204951862,0.8002575853882048) circle (1.0pt);
\draw[color=black] (5.351701892576387,1.101244986275009) node {$p_k$};
\draw [fill=black] (5.690029223252711,-0.22804521669258904) circle (1.0pt);
\draw[color=black] (6.047416671481641,0.06707436898341249) node {$p^*_k$};
\draw[color=qqccqq] (5.022647605256334,1.6183302949208076) node {e};
\end{tikzpicture}
\caption{Convexity is lost in flowing $p_{k}$ to $p_{k}^{*}$.}
\label{fig:convexity}
\end{figure}
Fig.~\ref{fig:convexity} illustrates the problem. When increasing the length of the one green edge $e$, there is no guarantee that the final result is a green-black arm chain to which the inductive hypothesis applies. The problem is that of \textit{convexity}. The final resulting arm after lengthening $e$ might fail to be convex when the free vertices are connected, so the inductive hypothesis fails to apply. This is the green-black version of the mistake Cauchy made in his original argument of 1813 that subsequently was noticed and repaired by Ernst Steinitz over 100 years later. This loss of convexity is the problem that necessitates a more careful analysis, to which we now turn.

\subsubsection{The proof of the Green-Black Arm Lemma} 
%Recall that our color coding is used to indicate fixed and variable quantities. Black edges are fixed in length with green edges variable, and angles at black vertices are fixed with ones at green vertices variable.
Despite the caveat of the preceding paragraph, the proof does proceed by induction on $N$, the number of green edges. The basis of the induction is when $N = 0$, and as stated already, is just the hyperbolic version of the classical Cauchy Arm Lemma. The proof in this case is exactly that of I.\,Schoenberg's reported on pp.\,340-341 of \cite{FuchsTab:2007}, but in which the distances are interpreted as hyperbolic distances. Assume then that for some $N \geq 1$, the lemma is true whenever there are no more than $N-1$ green edges. By reversing the order if necessary, we may assume that the number of vertices between the initial vertex $p_{1}$ and the nearest green edge is at least as large as the number between the terminal vertex $p_{n}$ and its nearest green edge. Let $e=p_{i}p_{i+1}$ denote the last green edge one meets as one traverses the chain $P$ from its first to last vertex and note that $i \leq n - 2$.

In the paragraph following, we are going to produce a new green-black arm chain that will be obtained by fixing the vertices $p_{1}, \dots ,p_{i}$ of $P$ and elongating the edge $e$ to obtain the edge $e^{*} = p_{i}p_{i+1}^{*}$ by translating $p_{i+1}$ to $p_{i+1}^{*}$ along the hyperbolic line $\ell$ supporting $e$. The remaining vertices $p_{j}$ for $j >i$ will rigidly translate via the hyperbolic flow $T_{\ell}$ with axis $\ell$ to the respective vertices $p_{j}^{*}$. This new green-black arm chain determined by $p_{1} \dots, p_{i} , p_{i+1}^{*}, \dots ,p_{n}^{*}$ is denoted as $P^{*}$ and will have the same corresponding interior angles as $P$ as well as the same green edge lengths, except for the single edge $e$ for which $|e| \leq |e^{*} |\leq |e' |$, where $e' = p_{i}'p_{i+1}'$. In particular, $P^{*}$ and $P'$ will be compatible, black-edge congruent, green-black arm chains that satisfy conditions (1)--(3) of the lemma.

To obtain the intermediate green-black arm chain $P^{*}$, continuously increase the length of $e$ by flowing the points $p_{j}$ for $j >i$ via $T_{\ell}$ until one of two possibilities occurs. Either (I) the hyperbolic length of $e^{*} = p_{i}p_{i+1}^{*}$ is equal to that of $e' = p_{i}'p_{i+1}'$ and the polygon formed by connecting the free vertices of $P^{*}$ is convex, or (II) we reach the situation where $p_{1}$, $p_{2}$, and $p_{n}^{*}$ become collinear, all lying on the hyperbolic line supporting the edge $p_{1}p_{2}$, before $e^{*}$ is equal in length to $e'$. In this latter case, to continue the flow until the lengths of $e^{*}$ and $e'$ agree would produce a non-convex polygon. Our first task in either of the cases is to show that (\dag) $|p_{1}p_{n}| \leq  |p_{1} p_{n}^{*}|$ and that the inequality is strict if and only if $e \neq e^{*}$, i.e., the length of $e^{*}$ is strictly larger than that of $e$. After this is verified, we consider the two cases, (I) and (II), in turn.

\textit{Verifying} (\dag): \textit{$|p_{1}p_{n}| \leq  |p_{1} p_{n}^{*}|$ with strict inequality if and only if $e \neq e^{*}$}. We use the following fact of elementary hyperbolic geometry.
\begin{figure}
\definecolor{qqccqq}{rgb}{0.,0.8,0.}
\begin{tikzpicture}[line cap=round,line join=round,>=triangle 45,x=1.0cm,y=1.0cm]
\clip(-4.,-1.) rectangle (2.75,5.6);
\draw [shift={(13.727320799059926,-4.992209165687427)}] plot[domain=2.5326737967699557:2.9203620430202832,variable=\t]({1.*16.735188216209988*cos(\t r)+0.*16.735188216209988*sin(\t r)},{0.*16.735188216209988*cos(\t r)+1.*16.735188216209988*sin(\t r)});
\draw [shift={(2.6716427088845567,10.511143865801573)}] plot[domain=3.92878274186884:4.583166294721245,variable=\t]({1.*7.6953045399545354*cos(\t r)+0.*7.6953045399545354*sin(\t r)},{0.*7.6953045399545354*cos(\t r)+1.*7.6953045399545354*sin(\t r)});
\draw [shift={(0.6029982859424949,6.6334100443398745)},dash pattern=on 2pt off 2pt]  plot[domain=3.7416092351244186:4.935594987233244,variable=\t]({1.*5.407556435137897*cos(\t r)+0.*5.407556435137897*sin(\t r)},{0.*5.407556435137897*cos(\t r)+1.*5.407556435137897*sin(\t r)});
\draw [shift={(2.6716427088845367,10.51114386580153)},line width=2.pt,color=qqccqq]  plot[domain=4.135427445907937:4.372441422746121,variable=\t]({1.*7.695304539954489*cos(\t r)+0.*7.695304539954489*sin(\t r)},{0.*7.695304539954489*cos(\t r)+1.*7.695304539954489*sin(\t r)});
\draw (-0.7572281089805173,3.391276048115255)-- (-0.9228671759279671,3.4927098392512934);
\draw (-0.9228671759279671,3.4927098392512934)-- (-0.8143007638919483,3.6555594573053183);
\draw (-1.6849057577082065,1.5226682128209137)-- (-1.8555712478500135,1.5926666145083266);
\draw (-1.8555712478500135,1.5926666145083266)-- (-1.7710472536614918,1.7810927006807693);
\draw (-2.7610685417716825,1.003468245450142) node[anchor=north west] {$p_1$};
\draw (-1.6535555503912052,1.614509895866956) node[anchor=north west] {$q$};
\draw (-0.2787118369533713,1.2135138127809217) node[anchor=north west] {$p_n$};
\draw (0.8669912575781569,1.2517039159319725) node[anchor=north west] {$p_n^*$};
\draw (1.9554091973831087,1.881840617924312) node[anchor=north west] {$h$};
\draw (1.8026487847789048,3.3139694860887197) node[anchor=north west] {$\ell$};
\draw (-0.6797079200394062,3.428539795541872) node[anchor=north west] {$u$};
\draw (0.3514248650389692,4.020486394383161) node[anchor=north west] {$p_{i+1}$};
\draw (-1.5198901893625267,4.956143921583907) node[anchor=north west] {$p_i$};
\draw (0.16047434928371448,5.318949901518891) node[anchor=north west] {$\ell_1$};
\begin{scriptsize}
\draw [fill=black] (-2.06,0.56) circle (1.5pt);
\draw [fill=black] (-1.6115069903204997,1.7000913864175384) circle (1.5pt);
\draw [fill=black] (-0.6514564188798335,3.5703424603659997) circle (1.5pt);
\draw [fill=black] (-0.002271415995797499,1.2598344017051852) circle (1.5pt);
\draw [fill=black] (1.0167988982366745,1.2417094172217258) circle (1.5pt);
\draw [fill=black] (0.10573843626398416,3.256225036042527) circle (1.5pt);
\draw [fill=black] (-1.5259908420995938,4.06152492091056) circle (1.5pt);
\end{scriptsize}
\end{tikzpicture}
\caption{Flow along a hypercycle increases distance.}
\label{fig:hypercycleFlow}
\end{figure}
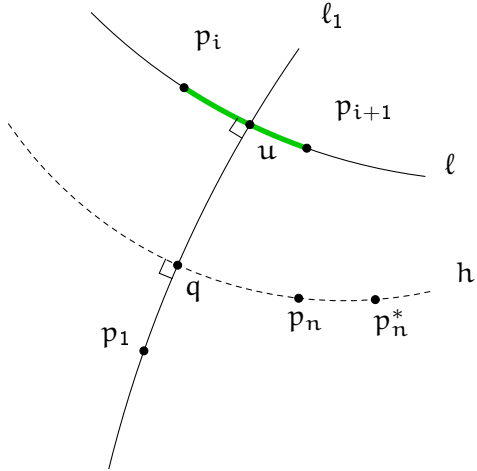
\begin{Lemma}
Fix a point $p$ and a hypercycle $h$\footnote{This includes the case where $h$ is a hyperbolic line.} in the hyperbolic plane $\mathbb{H}^{2}$. Let $m$ be the unique hyperbolic line passing through $p$ and perpendicular to $h$ and let $q$ be the unique point of intersection of $m$ with $h$. Let $r$ and $s$ be points on $h$ in one of the half planes determined by $m$ with $r$ strictly between $q$ and $s$ in the ordering on the hypercycle $h$. Then $|pr| < |ps|$.
\end{Lemma}
Let $\sigma$ be the geodesic arc of shortest length that connects the point $p_{1}$ to the edge $e$. By Lemma~\ref{lem:greenblackcontainment}, $\sigma$ meets $e$ orthogonally, say at the point $u$. Since the polygon formed by adding the edge $p_{1}p_{n}$ to the arm chain $P$ is convex, the line $\ell_{1}$ supporting $\sigma$ meets $P$ only at the vertex $p_{1}$ and the point $u$ of $e$, unless $\sigma = p_{1}p_{2}$ with $u= p_{2}$, in which case $\ell_{1}$ meets $P$ along its first edge $p_{1}p_{2}$. In either case, as the chain of edges from $p_{i+1}$ to $p_{n}$ along $P$ is connected and does not meet $\ell_{1}$, that chain lies in one half-space determined by $\ell_{1}$. Let $h$ be the hypercycle determined by $p_{n}$ and $\ell$, the component containing $p_{n}$ of the set of points whose distance to $\ell$ is equal to the distance from $p_{n}$ to $\ell$. This hypercycle $h$ is exactly the flow line of $p_{n}$ under the hyperbolic flow $T_{\ell}$. Since $\ell_{1}$ is orthogonal to $e$ and therefore to $\ell$, $\ell_{1}$ is orthogonal to all the hypercycles determined by $\ell$, or what is the same, to all the flow lines of $T_{\ell}$. In particular, $\ell_{1}$ is orthogonal to $h$. Let $q$ be the point of intersection of $\ell_{1}$ with $h$. Then since $e$ lengthens to $e^{*}$ when $e \neq e^{*}$ and $p_{i+1}$ and $p_{n}$ are in the same half-plane determined by $\ell_{1}$, $p_{n}$ is strictly between $q$ and $p_{n}^{*}$ in the ordering of $h$; see Fig.~\ref{fig:hypercycleFlow}. Apply the preceding lemma with $p= p_{1}$, $r= p_{n}$ and $s=p_{n}^{*}$ to conclude that $|p_{1}p_{n}| <  |p_{1} p_{n}^{*}|$.

\textit{Case} (I). In this case, the new arm chain $P^{*}$ obtained from lengthening $e$ is convex when the free vertices are joined by a hyperbolic segment, so that it is a green-black arm chain. Moreover, conditions (1)--(3) hold for the two compatible, black-edge congruent, green-black arm chains $P^{*}$ and $P'$ with the additional fact that the lengths of the green edges $e^{*}$ and $e'$ agree. Recolor these green edges, $e^{*}$ and $e'$, black and note that $P^{*}$ and $P'$ with these recolored edges are two compatible, black-edge congruent green-black arm chains that satisfy conditions (1)--(3) and have $N-1$ green edges. Apply the inductive hypothesis for $N$ to conclude that $|p_{1}p_{n}^{*}| \leq |p_{1}'p_{n}'|$ with equality if and only if corresponding angles and edge lengths agree. Applying (\dag), we have $|p_{1} p_{n} | \leq |p_{1}p_{n}^{*}| \leq |p_{1}'p_{n}'|$, again with equality if and only if corresponding angles and edge lengths agree.

\textit{Case} (II). By increasing the size of the edge $e$, we have arrived at a point where $p_{1}$, $p_{2}$ and $p^*_{n}$ are collinear with $p_1$ between $p_2$ and $p^*_n$, so that $|p_1 p^*_n| = |p_2 p^*_n| - |p_1 p_2|$.  Were $e$ to increase its length further, convexity would be violated. Using (\dag) exactly as in the preceding paragraph, we need only verify that (\ddag) $|p_{1} p_{n}^{*}| \leq |p_{1}' p_{n}'|$. We verify (\ddag) by inducting on $M$, the number of black edges between $p_{1}$ and the first green edge encountered in a walk from $p_{1}$ to $p_{n}$ in $P$. The basis of the induction when $M=1$ is a bit more difficult to prove than the inductive step, so we will delay its verification until after we verify the inductive step.

(\ddag) \textit{The inductive step, $M \geq 2$}: Let $M \geq 2$ and assume that the result holds for $M-1$. Since $M$ is at least $2$, the initial two edges, $p_{1}p_{2}$ and $p_{2}p_{3}$, are black.  Consider the chains $p_2, \cdots, p^*_n$ and $p'_2, \cdots, p'_n$ and apply the inductive hypothesis for $M$ to conclude that $|p_2 p^*_n| < |p'_2 p'_n|$. We have 
\begin{equation*}
	|p_1 p^*_n| = |p_2 p^*_n| - |p_1 p_2| \leq |p'_2 p'_n| - |p'_1 p'_2| \leq |p'_1 p'_n|,
\end{equation*}
where the first inequality follows since $|p_{1}p_{2}| = |p_{1}'p_{2}'|$, and the last inequality is the triangle inequality.

(\ddag) \textit{The basis of the induction, $M=1$}: When $M=1$, by our initial requirement that the number of vertices between the initial vertex $p_{1}$ and the nearest green edge is at least as large as the number between the terminal vertex $p_{n}$ and its nearest green edge, both edges $p_2p_3$ and $e = p_{n-2}p_{n-1}$ are green. By Corollary~\ref{cor:relaxedgreenblackcontainment} the shortest path between  $p_2p_3$ and $p_{n-2}p_{n-1}^{*}$ meets both orthogonally. Let $X$ and $Y$ be the endpoints of this path. We obtain Fig.~\ref{fig:actiii} by applying a M\"obius transformation taking $X$ and $Y$ to the respective points $-A$ and $A$ on the $x$-axis, with $p_1$, $p_2$, $p_{n-1}^*$ and $p_n^*$  below the $x$-axis and the remaining vertices of $P$ above the $x$-axis. 

Next, along the lines $k$ and $\ell$ supporting the respective green edges $p_{2}p_{3}$ and $p_{n-2}p_{n-1}^{*}$, stretch both green edges so that their lengths equal the corresponding lengths in $P'$. To do this, flow $p_1p_2$ using the hyperbolic flow $T_{k}$ with axis $k$ to edge $p_{1}^*p_{2}^*$ so that $|p_2^* p_3| = |p_2' p_3'|$. By hypothesis, the edge $p_{1}p_{2}$ meets $k$ orthogonally, and so the edge $p_{1}^*p_{2}^*$ must also meet $k$ orthogonally. Similarly flow $p_{n-1}^*p_n^*$ to $p_{n-1}^{**} p_n^{**}$ using $T_{\ell}$ so that $|p_{n-2} p_{n-1}^{**}| = |p'_{n-2} p'_{n-1}|$.

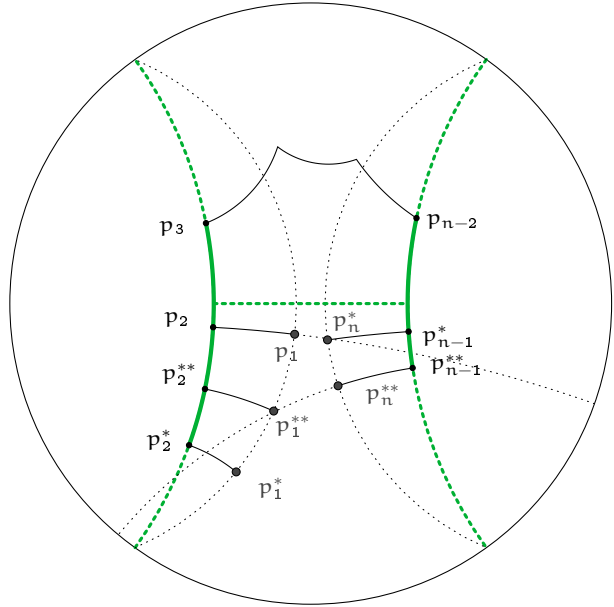
\begin{figure}
\definecolor{greenedge}{rgb}{0.,0.69215686274509803,0.2}
\centering
\definecolor{uuuuuu}{rgb}{0.26666666666666666,0.26666666666666666,0.26666666666666666}
\definecolor{ttffqq}{rgb}{0.,0.69215686274509803,0.2}
\begin{tikzpicture}[line cap=round,line join=round,>=triangle 45,x=4.0cm,y=4.0cm]
\clip(-1.1,-1.1) rectangle (1.1,1.1);
\draw(0.,0.) circle (4.cm);
\draw [shift={(-1.7133103604075548,0.)},line width=1.2pt,dotted,color=ttffqq]  plot[domain=-0.6232354939684139:0.6232354939684144,variable=\t]({1.*1.3911981854070488*cos(\t r)+0.*1.3911981854070488*sin(\t r)},{0.*1.3911981854070488*cos(\t r)+1.*1.3911981854070488*sin(\t r)});
\draw [shift={(1.7133103604075548,0.)},line width=1.2pt,dotted,color=ttffqq]  plot[domain=2.518357159621379:3.7648281475582075,variable=\t]({1.*1.3911981854070488*cos(\t r)+0.*1.3911981854070488*sin(\t r)},{0.*1.3911981854070488*cos(\t r)+1.*1.3911981854070488*sin(\t r)});
\draw [shift={(-0.9303628646300736,0.)},dotted]  plot[domain=-1.1672579846756577:1.167257984675658,variable=\t]({1.*0.8829121177454339*cos(\t r)+0.*0.8829121177454339*sin(\t r)},{0.*0.8829121177454339*cos(\t r)+1.*0.8829121177454339*sin(\t r)});
\draw [shift={(0.9303628646300736,0.)},dotted]  plot[domain=1.9743346689141354:4.308850638265451,variable=\t]({1.*0.882912117745434*cos(\t r)+0.*0.882912117745434*sin(\t r)},{0.*0.882912117745434*cos(\t r)+1.*0.882912117745434*sin(\t r)});
\draw [shift={(-0.5836654135226991,-4.659568838911487)}] plot[domain=1.4549606579456849:1.5142461514325836,variable=\t]({1.*4.588272777363783*cos(\t r)+0.*4.588272777363783*sin(\t r)},{0.*4.588272777363783*cos(\t r)+1.*4.588272777363783*sin(\t r)});
\draw [shift={(-0.5836654135226991,-1.397683746240294)}] plot[domain=1.1539415007234604:1.365194959196559,variable=\t]({1.*1.1376225953482666*cos(\t r)+0.*1.1376225953482666*sin(\t r)},{0.*1.1376225953482666*cos(\t r)+1.*1.1376225953482666*sin(\t r)});
\draw [shift={(-0.5836654135226991,-0.9695902303135197)}] plot[domain=0.8839319288552139:1.225409902477138,variable=\t]({1.*0.5298778440943233*cos(\t r)+0.*0.5298778440943233*sin(\t r)},{0.*0.5298778440943233*cos(\t r)+1.*0.5298778440943233*sin(\t r)});
\draw [shift={(0.5836654135226991,-1.7865142672635885)}] plot[domain=1.7253523966107887:1.8856592596075747,variable=\t]({1.*1.5913197485354662*cos(\t r)+0.*1.5913197485354662*sin(\t r)},{0.*1.5913197485354662*cos(\t r)+1.*1.5913197485354662*sin(\t r)});
\draw [shift={(0.5836654135226991,-3.944004965264104)}] plot[domain=1.6378059401713831:1.708009596069373,variable=\t]({1.*3.859512984946485*cos(\t r)+0.*3.859512984946485*sin(\t r)},{0.*3.859512984946485*cos(\t r)+1.*3.859512984946485*sin(\t r)});
\draw [shift={(0.5836654135226991,-1.7865142672635885)},dotted]  plot[domain=1.8856592596075747:2.4476178472242256,variable=\t]({1.*1.5913197485354662*cos(\t r)+0.*1.5913197485354662*sin(\t r)},{0.*1.5913197485354662*cos(\t r)+1.*1.5913197485354662*sin(\t r)});
\draw [shift={(-0.5836654135226991,-4.659568838911487)},dotted]  plot[domain=1.2315925062128412:1.4549606579456849,variable=\t]({1.*4.588272777363784*cos(\t r)+0.*4.588272777363784*sin(\t r)},{0.*4.588272777363784*cos(\t r)+1.*4.588272777363784*sin(\t r)});
\draw [line width=1.2pt,dotted,color=ttffqq] (-0.32211217500050604,0.)-- (0.32211217500050604,0.);
\draw [shift={(-0.5005721658530542,0.6496815970963868)}] plot[domain=5.092086785990749:5.9655566796008115,variable=\t]({1.*0.41141504759166125*cos(\t r)+0.*0.41141504759166125*sin(\t r)},{0.*0.41141504759166125*cos(\t r)+1.*0.41141504759166125*sin(\t r)});
\draw [shift={(0.05910732233073381,0.7401033212236938)}] plot[domain=4.055404391506428:5.054539823809769,variable=\t]({1.*0.27646165246559895*cos(\t r)+0.*0.27646165246559895*sin(\t r)},{0.*0.27646165246559895*cos(\t r)+1.*0.27646165246559895*sin(\t r)});
\draw [shift={(0.755776587503395,0.8975239148864734)}] plot[domain=3.746873024630914:4.129425647690647,variable=\t]({1.*0.7343806577188957*cos(\t r)+0.*0.7343806577188957*sin(\t r)},{0.*0.7343806577188957*cos(\t r)+1.*0.7343806577188957*sin(\t r)});
\draw [shift={(-1.7133103604075581,0.)},line width=1.6pt,color=ttffqq]  plot[domain=-0.34538642431775823:0.1935357390549599,variable=\t]({1.*1.3911981854070519*cos(\t r)+0.*1.3911981854070519*sin(\t r)},{0.*1.3911981854070519*cos(\t r)+1.*1.3911981854070519*sin(\t r)});
\draw [shift={(1.71331036040756,0.)},line width=1.6pt,color=ttffqq]  plot[domain=2.935685878024264:3.2961487234056843,variable=\t]({1.*1.3911981854070536*cos(\t r)+0.*1.3911981854070536*sin(\t r)},{0.*1.3911981854070536*cos(\t r)+1.*1.3911981854070536*sin(\t r)});
\begin{scriptsize}
\draw [fill=black] (-0.32433605413006195,-0.07863057662781253) circle (1.0pt);
\draw[color=black] (-0.44310608023037346,-0.053280129069811219) node {$p_2$};
\draw [fill=uuuuuu] (-0.05336754074398381,-0.10204415487258145) circle (1.5pt);
\draw[color=uuuuuu] (-0.07764964271487182,-0.1693719368896748) node {$p_1$};
\draw [fill=black] (-0.35141305028708003,-0.2840213156903558) circle (1.0pt);
\draw[color=black] (-0.4360158333607989,-0.23573604626855035) node {$p^{**}_2$};
\draw [fill=black] (-0.40426980825667713,-0.47100448395359273) circle (1.0pt);
\draw[color=black] (-0.5023799427396744,-0.45482837440517704) node {$p^{*}_2$};
\draw [fill=uuuuuu] (-0.24766153422505605,-0.5598684677001787) circle (1.5pt);
\draw[color=uuuuuu] (-0.13119476614965905,-0.6138299353554665) node {$p^{*}_1$};
\draw [fill=uuuuuu] (-0.12305724894724938,-0.3574793008674811) circle (1.5pt);
\draw[color=uuuuuu] (-0.058194245832896105,-0.40810119628095215) node {$p^{**}_1$};
\draw [fill=black] (0.32523444687001724,-0.0931539013503761) circle (1.0pt);
\draw[color=black] (0.45926427294965736,-0.109916540007698774) node {$p^{*}_{n-1}$};
\draw [fill=black] (0.33869530272056525,-0.21416310200419578) circle (1.0pt);
\draw[color=black] (0.4857191495148701,-0.20591758064122502) node {$p^{**}_{n-1}$};
\draw [fill=uuuuuu] (0.09085573460756652,-0.2734256504117719) circle (1.5pt);
\draw[color=uuuuuu] (0.24044424637204326,-0.3019186212747512) node {$p^{**}_n$};
\draw [fill=uuuuuu] (0.055749220350437036,-0.12076746623830655) circle (1.5pt);
\draw[color=uuuuuu] (0.12098884949006751,-0.059901925619626568) node {$p^{*}_{n}$};
\draw [fill=black] (-0.3480853686291794,0.26756889375187126) circle (1.0pt);
\draw[color=black] (-0.4595609567955863,0.25190400688667903) node {$p_3$};
\draw [fill=black] (0.35149985617321833,0.28443723672699334) circle (1.0pt);
\draw[color=black] (0.47271862919810714,0.2718132397003417) node {$p_{n-2}$};
\end{scriptsize}
\end{tikzpicture}
\caption{The construction for the proof of the basis of the induction (\ddag). The dotted green line on the $x$-axis is the shortest path between $p_2 p_3$ and $p_{n-2} p_{n-1}^*$.}
\label{fig:actiii}
\end{figure}

If the resulting polygon $p_1^* p_2^* p_3 \cdots p_{n-2} p_{n-1}^{**} p_n^{**}$ fails to be convex, the line through one of the edges, either $p_{1}^{*} p_{2}^{*}*$ or $p_{n-1}^{**} p_n^{**}$, meets the polygon at a point not on that edge. In this case, we flow the other edge backwards until convexity is restored. The argument is symmetric in the two cases, so to continue we will assume, without loss of generality, that the support line $m$ of $p_{n-1}^{**} p_n^{**}$ meets the polygon at a point not on that edge. Now apply the flow $T_{k}$ in reverse to move the edge $p_{1}^{*} p_{2}^{*}$ back until its image $p_1^{**} p_2^{**}$ under the flow meets $m$ at the point $p_{1}^{**}$; see Fig.~\ref{fig:actiii}. The points  $p_1^{**}, p_2^{**}, p_3, \cdots p_{n-2}, p_{n-1}^{**} p_n^{**}$ determine a green-black arm chain $P^{**}$ that is compatible and black-edge congruent with $P'$. Moreover, by the construction of $P^{**}$, it satisfies conditions (1)--(3) of the lemma and the green edge $e^{**} = p_{n-2}p_{n-1}^{**}$ has the same length as the green edge $e' = p_{n-2}'p_{n-1}'$. As in Case (I), recolor these green edges, $e^{**}$ and $e'$, black and note that $P^{**}$ and $P'$ with these recolored edges are two compatible, black-edge congruent green-black arm chains that satisfy conditions (1)--(3) and have $N-1$ green edges. Apply the inductive hypothesis for $N$ to conclude that $|p_{1}^{**}p_{n}^{**}| \leq |p_{1}'p_{n}'|$ with equality if and only if corresponding angles and edge lengths agree. 

The three inequalities
\begin{equation*}
	|p_{1}p_{n}| \leq |p_{1}p_{n}^{*}|, \quad  |p_{1}p_{n}^{*}| \leq |p_{1}^{**}p_{n}^{**}|, \quad |p_{1}^{**}p_{n}^{**}| \leq |p_{1}'p_{n}'|
\end{equation*}
imply that $|p_{1}p_{n}| \leq |p_{1}'p_{n}'|$, confirming the basis step of induction on $M$. The first inequality is from (\dag), the third is confirmed in the preceding paragraph, and the second is a consequence of the next lemma. That equality holds if and only if corresponding angles and edge lengths agree is straightforward.

This completes the proof of the Green-Black Arm Lemma modulo the verification of the lemma following. In applying the lemma, we set $p_{1} = b$, $p_{n}^{*} = c$, $p_{1}^{**} = B$ and $p_{n}^{**} = C$.

\begin{Lemma}\label{RegionRflow}
	Let $R$ be the open convex region in $\mathbb{H}^{2}$ bounded by the hyperbolic rays $k$ and $\ell$ and the hyperbolic segment $m$, with $m$ orthogonal to both $k$ and $\ell$. Let $c\in R$ and let $a$ be that unique point on $k$ for which the segment $ac$ is orthogonal to $k$. Choose any point $b$ in $ac$ other than one of the endpoints. Let $B$ and $C$ be points in $R$ such that the hyperbolic distance from $B$ to $k$ is equal to that from $b$ to $k$, and the hyperbolic distance from $C$ to $\ell$ is equal to that from $c$ to $\ell$. Assume further that $B$ and $C$ lie in the half-plane bordered by the line through $b$ and $c$ that does not meet $m$. Then $|bc| \leq |BC|$, with equality if and only if $b=B$ and $c=C$. See Fig.~\ref{fig:RegionR}.
\end{Lemma} 
\begin{proof}
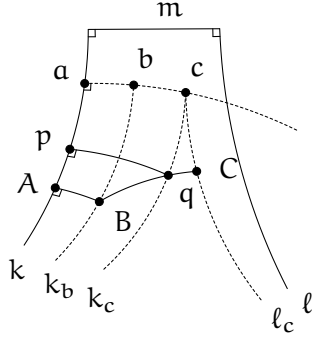
\begin{figure}
\begin{tikzpicture}[line cap=round,line join=round,>=triangle 45,x=1.0cm,y=1.0cm]
\clip(-3.,-0.75) rectangle (1.5,4.1);
\draw (-1.5470344343126554,3.5507329384306305)-- (0.2,3.56);
\draw [shift={(-6.742680297805866,3.532402456953445)}] plot[domain=-0.5811616127601003:0.003528032107170221,variable=\t]({1.*5.195678198790347*cos(\t r)+0.*5.195678198790347*sin(\t r)},{0.*5.195678198790347*cos(\t r)+1.*5.195678198790347*sin(\t r)});
\draw [shift={(8.458128342245985,3.8610160427807476)}] plot[domain=3.178027402461839:3.6141105892300374,variable=\t]({1.*8.26361267092171*cos(\t r)+0.*8.26361267092171*sin(\t r)},{0.*8.26361267092171*cos(\t r)+1.*8.26361267092171*sin(\t r)});
\draw [shift={(-1.4675964094134175,-3.3008717654795223)},dash pattern=on 1pt off 1pt]  plot[domain=1.1166298129769001:1.5914788879154662,variable=\t]({1.*6.1339882302487*cos(\t r)+0.*6.1339882302487*sin(\t r)},{0.*6.1339882302487*cos(\t r)+1.*6.1339882302487*sin(\t r)});
\draw [shift={(-4.628211852533284,3.0548241980375703)},dash pattern=on 1pt off 1pt]  plot[domain=5.511836984691178:6.2170205518231025,variable=\t]({1.*3.6936087592891576*cos(\t r)+0.*3.6936087592891576*sin(\t r)},{0.*3.6936087592891576*cos(\t r)+1.*3.6936087592891576*sin(\t r)});
\draw [shift={(-3.62737783549824,2.845541635127112)},dash pattern=on 1pt off 1pt]  plot[domain=5.456292612973673:6.243490869537996,variable=\t]({1.*3.3778713093128596*cos(\t r)+0.*3.3778713093128596*sin(\t r)},{0.*3.3778713093128596*cos(\t r)+1.*3.3778713093128596*sin(\t r)});
\draw [shift={(4.340361388512763,2.8181428829212747)},dash pattern=on 1pt off 1pt]  plot[domain=3.1648107068472267:3.8186885234406005,variable=\t]({1.*4.593766877722838*cos(\t r)+0.*4.593766877722838*sin(\t r)},{0.*4.593766877722838*cos(\t r)+1.*4.593766877722838*sin(\t r)});
\draw [shift={(-2.7199292423876806,-1.9349055026345947)}] plot[domain=1.1790176832152184:1.3567435846557645,variable=\t]({1.*3.4568243194693786*cos(\t r)+0.*3.4568243194693786*sin(\t r)},{0.*3.4568243194693786*cos(\t r)+1.*3.4568243194693786*sin(\t r)});
\draw [shift={(0.12363778871669442,-1.3534295525964528)}] plot[domain=1.6467123822461038:2.0986244217535197,variable=\t]({1.*3.025142333443171*cos(\t r)+0.*3.025142333443171*sin(\t r)},{0.*3.025142333443171*cos(\t r)+1.*3.025142333443171*sin(\t r)});
\draw [shift={(-2.203030323973693,-2.2352425049597886)}] plot[domain=1.1503064270210954:1.4729713032711869,variable=\t]({1.*4.212340354631171*cos(\t r)+0.*4.212340354631171*sin(\t r)},{0.*4.212340354631171*cos(\t r)+1.*4.212340354631171*sin(\t r)});
\draw (-1.4564875523492489,3.551213240085589)-- (-1.4540887542362542,3.4530561255820897);
\draw (-1.5476091481179175,3.4529814580599876)-- (-1.4540887542362542,3.4530561255820897);
\draw (-1.5116679149299115,2.832958140152935)-- (-1.516491627882574,2.7426234102239886);
\draw (-1.6070812240096872,2.744561039783)-- (-1.516491627882574,2.7426234102239886);
\draw (-1.7053145715600084,1.9475903022479701)-- (-1.732501575119836,1.8497822949766418);
\draw (-1.8220300219332772,1.8643878263400584)-- (-1.732501575119836,1.8497822949766418);
\draw (0.10839140360827071,3.559514066531047)-- (0.11078330797146572,3.455456236106948);
\draw (0.20436186569753723,3.457737819189555)-- (0.11078330797146572,3.455456236106948);
\draw (-0.7798334955123575,4.035953753375385) node[anchor=north west] {$m$};
\draw (-2.1669386599306195,3.198907533467816) node[anchor=north west] {$a$};
\draw (-1.030947361484629,3.43806359629855) node[anchor=north west] {$b$};
\draw (-0.3254369761339614,3.2467387460339627) node[anchor=north west] {$c$};
\draw (-2.3821791164782806,2.337945707277173) node[anchor=north west] {$p$};
\draw (-2.6332929824505524,1.7759289596249483) node[anchor=north west] {$A$};
\draw (-1.3418502431645842,1.2737012276804067) node[anchor=north west] {$B$};
\draw (-0.49284622011547574,1.5606885030772877) node[anchor=north west] {$q$};
\draw (0.045254921253677555,1.9792116130310722) node[anchor=north west] {$C$};
\draw (-2.7409132107243828,0.6279798580374247) node[anchor=north west] {$k$};
\draw (-2.27455888820445,0.44861281091437416) node[anchor=north west] {$k_b$};
\draw (-1.7005843374106866,0.2572879606497869) node[anchor=north west] {$k_c$};
\draw (0.7507653066043453,-0.017741511605557286) node[anchor=north west] {$\ell_c$};
\draw (1.1573306134165944,0.17358333865902997) node[anchor=north west] {$\ell$};
\draw (-2.025793390437933,1.3538594005769427)-- (-1.9329693681532332,1.330062000248358);
\draw (-1.9329693681532332,1.330062000248358)-- (-1.8979482294299626,1.422769524822937);
\begin{scriptsize}
\draw [fill=black] (-1.5944539511737768,2.8318045485719616) circle (1.5pt);
\draw [fill=black] (-1.791614952409546,1.9569584271634552) circle (1.5pt);
\draw [fill=black] (-1.9856241072640834,1.4430273481198492) circle (1.5pt);
\draw [fill=black] (-0.942685037956081,2.8106157507430507) circle (1.5pt);
\draw [fill=black] (-0.25216734550919595,2.711494141503059) circle (1.5pt);
\draw [fill=black] (-1.4,1.26) circle (1.5pt);
\draw [fill=black] (-0.483520445932915,1.61015695966992) circle (1.5pt);
\draw [fill=black] (-0.10579855307615471,1.662999644758378) circle (1.5pt);
\end{scriptsize}
\end{tikzpicture}
\caption{The proof of Lemma~\ref{RegionRflow}.}
\label{fig:RegionR}
\end{figure}
	The proof uses the labeling of Fig.~\ref{fig:RegionR}. The dotted lines labeled $k_{b}$ and $k_{c}$ are the hypercycles determined by the line $k$ and the respective points $b$ and $c$; similarly, $\ell_{c}$ is the hypercycle determined by $\ell$ and $c$. The point $B$ lies on $k_{b}$ and the point $C$ lies on $\ell_{c}$. Let $A$ be the unique point on $k$ for which $AB$ is orthogonal to $k$ and let $q$ be the point of intersection of the hypercycle $k_{c}$ with the segment $BC$. Finally, let $p$ be the unique point on $k$ for which $pq$ is orthogonal to $k$. Then $|ab| + |bc| = |ac|  = |pq| \leq |AB| + |Bq|$. Since $|ab| = |AB|$, we conclude that $|bc| \leq |Bq| \leq |Bq| + |qC| = |BC|$. 
	
	The verification that $|bc| = |BC|$ implies $b=B$ and $c=C$ is left as an easy exercise. 
\end{proof}

%% file: sec_mainresult.tex
\section{The Proof of the Main Theorem}\label{Section:Proof}

The remainder of the argument is Cauchy's. Let $G(\mathcal{C})$ and $G(\mathcal{C}')$ be two proper, convex, non-unitary \textit{c}-polyhedra, both based on the same abstract spherical polyhedron $P$ with $1$-skeleton $G=P^{(1)}$, that have M\"obius-congruent \textit{c}-faces. Assume that $G(\mathcal{C})$ and $G(\mathcal{C}')$ are not M\"obius-congruent to one another. Label the edge $e$ of $P$ that is adjacent to the faces $f$ and $g$ as follows. When the oriented ortho-circles $O_{f}^{+}$ and $O_{g}^{+}$ meet at an angle larger than the angle of intersection of ${O_{f}'}^{+}$ and ${O_{g}'}^{+}$, label with a plus sign, a minus sign if smaller, and no sign if equal; or if ${O_{f}'}^{+}$ and ${O_{g}'}^{+}$ fail to meet, then a plus sign if $\langle  {O_{f}}^{+} , {O_{g}}^{+} \rangle > \langle  {O_{f}'}^{+} , {O_{g}'}^{+} \rangle$, a minus sign if $\langle  {O_{f}}^{+} , {O_{g}}^{+} \rangle < \langle  {O_{f}'}^{+} , {O_{g}'}^{+} \rangle$, and no sign otherwise.\footnote{Notice that in terms of the inversive distance, when $O_{f}^{+}$ and $O_{g}^{+}$ meet, $e$ gets a plus sign when  $\langle  {O_{f}}^{+} , {O_{g}}^{+} \rangle < \langle  {O_{f}'}^{+} , {O_{g}'}^{+} \rangle$.} By Theorem~\ref{Thm:Congruence ConditionI}, there is at least one edge of $P$ labeled with a plus or a minus sign. Apply Cauchy's Combinatorial Lemma~\ref{lem:cauchycombinatorial} to conclude that there is a vertex $v$ that is incident to at least one edge labeled with a plus or minus sign for which one encounters at most two sign changes in labels on the edges adjacent to $v$ as one walks around the vertex.

Consider now the green-black polygon $L(v)$, the \textit{c}-link of circle $C_{v}$ in $G(\mathcal{C})$. Label a green vertex or a green edge of $L(v)$ with a plus sign if the interior angle or edge-length is larger than the corresponding one in $L(v)'$, a minus sign if it is smaller, and no label if equal. By Remark~\ref{rem:complexangleL(v)} the green vertex or edge of $L(v)$ determined by the edge $e$ of $P$ has the same sign that the preceding paragraph assigns to the edge $e$, and so there are at most two sign changes on a walk around the \textit{c}-link $L(v)$. But this contradicts the Green-Black Polygon Four Vertex Lemma~\ref{lem:greenblackfourvertex}, which guarantees at least four sign changes. Therefore, the \textit{c}-polyhedra $G(\mathcal{C})$ and $G(\mathcal{C}')$ are M\"obius-congruent, verifying the Main Theorem.

%% file: sec_hyperideal.tex
\section{Hyperideal Polyhedra in $\mathbb{H}^{3}$}\label{Section:Hyperideal}
Convex hyperbolic polyhedra have been a topic of interest since Poincar\'e's 1881 study~\cite{Poincare:1881} and Dehn's 1905 paper~\cite{Dehn:1905}. Andre'ev in two 1970 papers classified all compact convex hyperbolic polyhedra with acute dihedral angles~\cite{Andreev:1970a} and studied non-compact ones with ideal vertices~\cite{Andreev:1970b}. Thurston~\cite{Thurston:1980} in the Princeton notes for his 1978-79 course on hyperbolic $3$-manifolds recovered and popularized Andre'ev's results in the context of circle packings on the Riemann sphere. In 1993, Hodsgen and Rivin characterized all compact ones~\cite{HodgsonRivin:1993}, even with dihedral angles greater than $\pi/2$, and in 1996 Rivin~\cite{Rivin:1996} characterized the ones with ideal vertices. Bao and Bonahon~\cite{BaoBonahon:2002} in 2002 characterized convex hyperideal hyperbolic polyhedra in terms of their combinatorial type and their dihedral angles. The results of Andre'ev, Thurston, Hodsgen-Rivin, Rivin, and Bao-Bonahon give conditions that guarantee existence of convex polyhedra in $\mathbb{H}^{3}$ as well as their global rigidity---uniqueness up to hyperbolic isometries.

Our Main Theorem implies nothing about existence, but does imply the global rigidity of certain generalized convex hyperideal polyhedra in $\mathbb{H}^{3}$. We first recall the Bao-Bonahon definition of hyperideal polyhedron. We use the Klein projective model of $\mathbb{H}^{3}$ where the hyperbolic $3$-space is identified with the open unit ball $B^{3}$ in $\mathbb{E}^{3} \subset \mathbb{RP}^{3}$. In this model, the hyperbolic lines and totally geodesic hyperbolic planes are the intersections of Euclidean lines and planes of $\mathbb{E}^{3}$ with $\mathbb{H}^{3} = B^{3}$. A \textit{hyperideal polyhedron} $P$ is the intersection of $\mathbb{H}^{3}$ with a compact convex projective polyhedron $P'$ of $\mathbb{RP}^{3}$ with two properties: ($1$) no vertex of $P'$ is in $\mathbb{H}^{3}$, and ($2$) every edge of $P'$ meets $\mathbb{H}^{3}$ nontrivially. Our interest is in a generalization of hyperideal polyhedra to regions $P$ obtained by intersecting $\mathbb{H}^{3}$ with a compact convex projective polyhedron $P'$ of $\mathbb{RP}^{3}$ with the two properties: ($1'$) no vertex of $P'$ is in the closed unit ball $\mathbb{H}^{3}\cup \partial \mathbb{H}^{3} = B^{3} \cup \mathbb{S}^{2}$ and ($2'$) every face of $P'$ meets $\mathbb{H}^{3}$ nontrivially. In particular, we allow that some or all of the edges of $P'$ lie outside the closed unit ball. The terminology \textit{hyperideal polyhedron} suggests a polyhedron whose vertices either lie at infinity (=\textit{ideal}), or lie beyond (=\textit{hyper}) infinity, per condition ($1$). The edges and faces do not lie, at least not entirely, beyond infinity per condition ($2$). We will use the term \textit{strictly-hyperideal polyhedron}\footnote{Bao and Bonahon~\cite{BaoBonahon:2002} use this term to indicate that conditions ($1'$) and ($2$) adhere. We loosen the term by allowing full edges to lie beyond infinity while still requiring condition ($2'$).} to mean those regions $P$ obtained by intersecting $\mathbb{H}^{3}$ with compact convex projective polyhedra $P'$ that satisfy properties ($1'$) and ($2'$). All vertices lie beyond infinity per condition ($1'$), and the edges may or may not lie entirely beyond infinity, but the faces do not lie, at least not entirely, beyond infinity per condition ($2'$). If in addition to conditions ($1'$) and ($2'$), $P'$ satisfies condition ($3'$) no edge of $P'$ is tangent to $\partial \mathbb{H}^{3} = \mathbb{S}^{2}$, we say that the strictly-hyperideal polyhedron $P$ is \textit{non-unitary}.

Let $P$ be a strictly-hyperideal polyhedron obtained from $P'$ and let $G$ be the $1$-skeleton of the polyhedron $P^{*}$ dual to $P'$. We describe a \textit{c}-framework $G(\mathcal{C})$ as follows. Let $u$ be a vertex of $G$. Then $u$ is a face of $P'$ and as such has a support plane $\Pi_{u}$. Define the oriented circle $C_{u}$ by $C_{u} = \Pi_{u} \cap \mathbb{S}^{2}$, oriented so that its companion disk $D_{u}$ is the disk on $\mathbb{S}^{2}$ it bounds that does not meet the interior of $P'$.\footnote{Since all the vertices of $P'$ lie exterior to the closed unit ball exactly one of the two complementary domains of $C_{u}$ in $\mathbb{S}^{2}$ meets the interior of $P'$. Another way to describe the orientation on $C_{u}$ is using the unit normal vector $\mathbf{n}_{u}$ in $\mathbb{E}^{3}$ to the plane $\Pi_{u} \cap \mathbb{E}^{3}$ that points exterior to $P'$. The circle $C_{u}$ is oriented so that $\mathbf{n}_{u}$ points toward $D_{u}$, or what is the same, so that $C_{u}$ is oriented counter-clockwiase looking from the tip of $\mathbf{n}_{u}$.}  Let $\mathcal{C} = \{C_{u} : u \in V(G) \}$. 
\begin{Lemma}
	The \textit{c}-framework $G(\mathcal{C})$ is a convex \textit{c}-polyhedron, and is non-unitary if $P$ is non-unitary.
\end{Lemma}
\begin{proof}
	Let $f$ be a face of $P^{*}$ with vertices $u_{1}, \dots ,u_{n}$. Then $u_{1}, \dots ,u_{n}$ are the faces of $P'$ that are incident at the vertex $f$ of $P'$ with respective support planes $\Pi_{u_{i}}$, for $i = 1, \dots ,n$. Since $P'$ is strictly-hyperideal, the vertex $f$ of $P'$ lies outside the closed ball $B^{3}\cup \mathbb{S}^{2}$. Let $O$ be the circle of intersection of $\mathbb{S}^{2}$ with the cone with vertex $f$ that circumscribes $\mathbb{S}^{2}$. Since each support plane $\Pi_{u_{i}}$ passes through $f$, the circle $C_{u_{i}}= \Pi_{u_{i}} \cap \mathbb{S}^{2}$ meets $O$ orthogonally. It follows that the \textit{c}-face $\mathcal{C}_{f} = \{ C_{u_{i}}: i = 1, \dots ,n \}$ is \textit{c}-planar with ortho-circle $O$. That $\mathcal{C}_{f}$ is not coaxial is a consequence of the fact that the vertices $u_{1}, \dots, u_{n}$ are not collinear and the fact that the \textit{c}-framework $G(\mathcal{C)}$ is edge-uncoupled is a consequence of the convexity of $P'$. It follows that $G(\mathcal{C})$ is a \textit{c}-polyhedron, and it remains to prove that it is convex. 
	
	To see that $G(\mathcal{C})$ is convex, recall that as $P'$ is a convex projective polyhedron, it is the closure of one of the components of the complement of the support planes of $P'$ that does not contain a projective line. Normalize $P'$ by applying, if needed, a projective transformation set-wise fixing the $2$-sphere $\mathbb{S}^{2}$ and moving the ortho-circle $O$ to the equator of $\mathbb{S}^{2}$, which moves the vertex $f$ of $P'$ to infinity, by which we mean that $f$ lies in the complement of $\mathbb{E}^{3}$ in $\mathbb{RP}^{3}$. Of course such a projective transformation acts as a M\"obius transformation of $\mathbb{S}^{2}$. Let $P''$ be the closure of the component of the complement of the $n$ support planes $\Pi_{u_{1}}, \dots , \Pi_{u_{n}}$ that contains $P'$, and note that our normalization implies that there is a projective line containing the vertex $f$ of $P'$ that runs parallel to the $z$ axis and is contained in $P''$. We need to show that $O$ may be oriented to $O^{+}$ so that all the oriented circles of the \textit{c}-framework $G(\mathcal{C})$ are segregated from $O^{+}$. If this is not possible, then there are faces $v$ and $w$ of $P'$ whose corresponding circles $C_{v}$ and $C_{w}$ are not both segregated from $O$ no matter which orientation is assigned to $O$. This means, without loss of generality, that the companion disk $D_{v}$ of $C_{v}$ overlaps the northern hemisphere of $\mathbb{S}^{2}$ more than the southern, and that the companion disk $D_{w}$ of $C_{w}$ overlaps the southern more than the northern. From the definition of the circles $C_{v}$ and $C_{w}$, since these disks $D_{v}$ and $D_{w}$ do not meet the interior of $P'$, it follows that $P'$ is contained in a region $Q$, bounded in $\mathbb{E}^{3}$, cut out by the planes $\Pi_{u_{1}}, \dots, \Pi_{u_{n}}, \Pi_{v}, \Pi_{w}$. But then $f$, which lies at infinity, cannot lie in $P'$ since  $Q$ is bounded in $\mathbb{E}^{3}$, a contradiction.
\end{proof}

Not all strictly-hyperideal polyhedra are rigid. The work of the present paper guarantees that one ingredient that ensures rigidity is that of properness. The strictly-hyperideal polyhedron $P = P'\cap \mathbb{H}^{3}$ is \textit{proper} provided its corresponding \textit{c}-polyhedron $G(\mathcal{C})$ is proper. Recall from the introduction that the \textit{rigidity} of $P$ means that it is determined up to isometries of $\mathbb{H}^{3}$ by the combinatorics of $P'$ and the hyperbolic isometry classes of the faces of $P$, which are determined by the faces of the dual polyhedron $P^{*}$. This means that combinatorics and hyperbolic equivalence of faces determines the complex angles in which adjacent faces meet, which in turn determines how $P'$ meets $\mathbb{H}^{3}$ up to projective equivalence that fixes $\mathbb{S}^{2}$.

\begin{Theorem}\label{RigidStrictlyIdeal}
	Proper, non-unitary, strictly-hyperideal polyhedra in $\mathbb{H}^{3}$ are globally rigid, unique up to hyperbolic isometries.
\end{Theorem}
\begin{proof}
	Apply The Main Theorem of the introduction to the corresponding \textit{c}-polyhedra $G(\mathcal{C})$. 
\end{proof}

Theorem~\ref{RigidStrictlyIdeal} offers an alternate proof of two corollaries that follow from Bao and Bonahon's work. The proofs are obtained from the observation that in the setting of either of the corollaries, the obviously non-unitary strictly-hyperideal polyhedron $P$ automatically is proper. The details of the proofs are left to the reader, but to close out this paper, we do provide a nice ``proof by picture'' in each case.

\begin{Corollary}[Bao-Bonahon~\cite{BaoBonahon:2002}]\label{COR:BB1}
	Let $P$ be a strictly-hyperideal polyhedron in $\mathbb{H}^{3}$ for which no edge of $P'$ meets the boundary $\partial \mathbb{H}^{3}$. Then $P$ is globally rigid.
\end{Corollary}
\begin{proof}
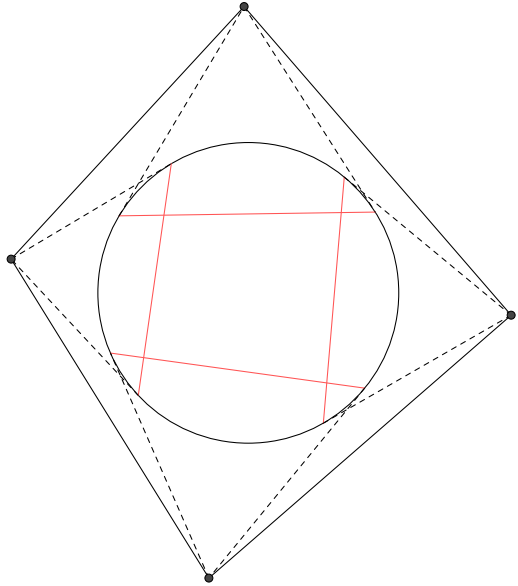
\begin{figure}
\definecolor{uuuuuu}{rgb}{0.26666666666666666,0.26666666666666666,0.26666666666666666}
\definecolor{ffqqqq}{rgb}{1.0,0.3333333333333333,0.3333333333333333}
\begin{tikzpicture}[line cap=round,line join=round,>=triangle 45,x=1.0cm,y=1.0cm]
\clip(-3.3,-3.9) rectangle (3.7,3.95);
\draw(0.,0.) circle (2.cm);
\draw [color=ffqqqq] (-1.4626714298749266,-1.3640352958144586)-- (-1.0239584924970655,1.717995636095487);
\draw [color=ffqqqq] (-1.7174900328529148,1.024806316847673)-- (1.686671988279327,1.0747732802567531);
\draw [color=ffqqqq] (1.2757077792632776,1.5403147931274168)-- (0.9973457692162961,-1.7335805192220968);
\draw [color=ffqqqq] (1.5468733117182991,-1.2677471977857657)-- (-1.8326257984119694,-0.8009261407863348);
\draw [dash pattern=on 2pt off 2pt] (-1.7174900328529148,1.024806316847673)-- (-0.05591609317447944,3.809465850391554);
\draw [dash pattern=on 2pt off 2pt] (-0.05591609317447944,3.809465850391554)-- (1.686671988279327,1.0747732802567531);
\draw [dash pattern=on 2pt off 2pt] (1.2757077792632776,1.5403147931274168)-- (3.494234212191163,-0.2970962618173735);
\draw [dash pattern=on 2pt off 2pt] (3.494234212191163,-0.2970962618173735)-- (0.9973457692162961,-1.7335805192220968);
\draw [dash pattern=on 2pt off 2pt] (1.5468733117182991,-1.2677471977857657)-- (-0.5241885842813877,-3.7948049419749177);
\draw [dash pattern=on 2pt off 2pt] (-0.5241885842813877,-3.7948049419749177)-- (-1.8326257984119694,-0.8009261407863348);
\draw [dash pattern=on 2pt off 2pt] (-1.4626714298749266,-1.3640352958144586)-- (-3.153312621028497,0.4488595581955424);
\draw [dash pattern=on 2pt off 2pt] (-3.153312621028497,0.4488595581955424)-- (-1.0239584924970655,1.717995636095487);
\draw (-3.153312621028497,0.4488595581955424)-- (-0.05591609317447944,3.809465850391554);
\draw (-0.05591609317447944,3.809465850391554)-- (3.494234212191163,-0.2970962618173735);
\draw (3.494234212191163,-0.2970962618173735)-- (-0.5241885842813877,-3.7948049419749177);
\draw (-0.5241885842813877,-3.7948049419749177)-- (-3.153312621028497,0.4488595581955424);
\begin{scriptsize}
\draw [fill=uuuuuu] (-3.153312621028497,0.4488595581955424) circle (1.5pt);
\draw [fill=uuuuuu] (-0.5241885842813877,-3.7948049419749177) circle (1.5pt);
\draw [fill=uuuuuu] (-0.05591609317447944,3.809465850391554) circle (1.5pt);
\draw [fill=uuuuuu] (3.494234212191163,-0.2970962618173735) circle (1.5pt);
\end{scriptsize}
\end{tikzpicture}
	\caption{In Corollary~\ref{COR:BB1}, the \textit{c}-link is a compact hyperbolic polygon. This is the projection of the upper hemisphere of $\mathbb{S}^{2}$ to the equatorial plane. The red lines are the ortho-circles and cut out a compact convex hyperbolic polygon.}\label{fig:Bao-Bon1}
\end{figure}
	Obviously $P$ is non-unitary. We show that $P$ is proper and apply Theorem~\ref{RigidStrictlyIdeal}. Let $u$ be a vertex of $P^{*}$ with corresponding circle $C_{u} = \Pi_{u} \cap \mathbb{S}^{2}$ in the \textit{c}-framework $G(\mathcal{C})$. By applying a projective transformation that fixes $\mathbb{S}^{2}$ if necessary, assume that the vertex $u$ lies at infinity in the direction of the north pole of $\mathbb{S}^{2}$ so that $C_{u}$ is the equator of $\mathbb{S}^{2}$ with $\Pi_{u}$ the equatorial plane and the companion disk $D_{u}$ the upper hemisphere. The vertex $u$ of $P^{*}$ now considered as a face of $P'$ is then a bounded convex Euclidean polygon in the equatorial plane $\Pi_{u}$ that separates the equator from infinity, none of whose edges meets $C_{u}$. Fig.~\ref{fig:Bao-Bon1} is a view of $D_{u}$ from above, projected to the plane $\Pi_{u}$. The ortho-circles that meet $D_{u}$ are precisely the intersections of the cones that inscribe $\mathbb{S}^{2}$ whose vertices are the vertices of the face $u$ of $P'$. This implies that the ortho-circles of the \textit{c}-polyhedron $G(\mathcal{C})$ that meet the companion disk $D_{u}$ cut out a compact convex hyperbolic polygon in $D_{u}$ with its hyperbolic metric, and so the \textit{c}-link of the vertex $u$ is a compact convex hyperbolic polygon and $G(\mathcal{C})$ is proper at $u$. In this case the \textit{c}-link has no green edges.
\end{proof}

\begin{Corollary}[Bao-Bonahon~\cite{BaoBonahon:2002}]\label{COR:BB2}
	Let $P$ be a strictly-hyperideal polyhedron in $\mathbb{H}^{3}$ for which every edge of $P'$ meets $\mathbb{H}^{3}$. Then $P$ is globally rigid.
\end{Corollary}
\begin{proof}
\begin{figure}
\definecolor{qqccqq}{rgb}{0.,0.8,0.}
\definecolor{uuuuuu}{rgb}{0.26666666666666666,0.26666666666666666,0.26666666666666666}
\definecolor{ffqqqq}{rgb}{1.,0.,0.}
\begin{tikzpicture}[line cap=round,line join=round,>=triangle 45,x=1.0cm,y=1.0cm]
\clip(-2.25,-2.2) rectangle (2.35,2.4);
\draw(0.,0.) circle (2.cm);
\draw [color=ffqqqq] (-1.680458119658683,1.0844632350030148)-- (0.4289050917517141,1.9534688178390394);
\draw [color=ffqqqq] (-1.9818250186199544,-0.26901597642522546)-- (-0.7014988683686489,-1.8729386903146363);
\draw [color=ffqqqq] (0.3269497449041427,-1.973094996270356)-- (1.9786625766975872,-0.2913664489546879);
\draw [dash pattern=on 1pt off 1pt] (-1.680458119658683,1.0844632350030148)-- (-0.9274699644751478,2.2512755514956666);
\draw [dash pattern=on 1pt off 1pt] (-0.9274699644751478,2.2512755514956666)-- (0.4289050917517141,1.9534688178390394);
\draw [dash pattern=on 1pt off 1pt] (-1.9818250186199544,-0.26901597642522546)-- (-1.8210241413647705,-1.4536266668201825);
\draw [dash pattern=on 1pt off 1pt] (-1.8210241413647705,-1.4536266668201825)-- (-0.7014988683686489,-1.8729386903146363);
\draw [dash pattern=on 1pt off 1pt] (0.3269497449041427,-1.973094996270356)-- (1.7661380033570506,-1.7346157365998087);
\draw [dash pattern=on 1pt off 1pt] (1.9786625766975872,-0.2913664489546879)-- (1.7661380033570506,-1.7346157365998087);
\draw (-1.207153700958701,1.0916358536189266)-- (-1.0136652974671245,1.1713483004553265);
\draw (-0.9585137800632598,1.3818864849258221)-- (-1.0136652974671245,1.1713483004553265);
\draw (-0.6730026328371006,1.4995100310080327)-- (-0.5737196128117099,1.352595064427541);
\draw (-0.5737196128117099,1.352595064427541)-- (-0.37537511893258985,1.4343081006525793);
\draw (1.1548742551629523,-0.8300925551450591)-- (1.0046639306933367,-0.9830325734700437);
\draw (1.0046639306933367,-0.9830325734700437)-- (1.124672113788065,-1.1608760349829372);
\draw (0.7975671796498527,-1.4939252764244644)-- (0.523940818823323,-1.4724916156380508);
\draw (0.523940818823323,-1.4724916156380508)-- (0.37362631083643427,-1.6255377107530111);
\draw (-0.7700055606049092,-1.535954903296295)-- (-0.8929928084509801,-1.3615005743066437);
\draw (-0.8929928084509801,-1.3615005743066437)-- (-1.1176959104491195,-1.3515497542881205);
\draw (-1.5050937869691894,-0.8662388408610191)-- (-1.4547983599528682,-0.657701245779161);
\draw (-1.4547983599528682,-0.657701245779161)-- (-1.5886267863803945,-0.49004829823212026);
\draw [color=qqccqq] (-1.6389222133967154,-0.6985858933139776)-- (-1.15685827394238,1.3001734487007834);
\draw [color=qqccqq] (-0.4746581389579805,1.581223067233071)-- (1.2749866217749541,-1.0078299398679769);
\draw [color=qqccqq] (0.6472526716629639,-1.6469713715394247)-- (-0.9838674840215943,-1.5192027018351602);
\draw (-1.6389222133967154,-0.6985858933139776)-- (-1.8210241413647705,-1.4536266668201825);
\draw (-1.8210241413647705,-1.4536266668201825)-- (-0.9838674840215943,-1.5192027018351602);
\draw (0.6472526716629639,-1.6469713715394247)-- (1.7661380033570506,-1.7346157365998087);
\draw (1.7661380033570506,-1.7346157365998087)-- (1.2749866217749541,-1.0078299398679769);
\draw (-1.15685827394238,1.3001734487007834)-- (-0.9274699644751478,2.2512755514956666);
\draw (-0.9274699644751478,2.2512755514956666)-- (-0.4746581389579805,1.581223067233071);
\begin{scriptsize}
\draw [fill=uuuuuu] (-0.9274699644751478,2.2512755514956666) circle (1.5pt);
\draw [fill=uuuuuu] (-1.8210241413647705,-1.4536266668201825) circle (1.5pt);
\draw [fill=uuuuuu] (1.7661380033570506,-1.7346157365998087) circle (1.5pt);
\end{scriptsize}
\end{tikzpicture}
	\caption{In Corollary~\ref{COR:BB2}, the \textit{c}-link is a right-angled green-black polygon with alternating green and black sides. This is the projection of the upper hemisphere of $\mathbb{S}^{2}$ to the equatorial plane. The black sides are colored red for emphasis, and are the ortho-circles, and the green ones are the orthogonal segments between successive ortho-circles.}\label{fig:Bao-Bon2}
\end{figure}
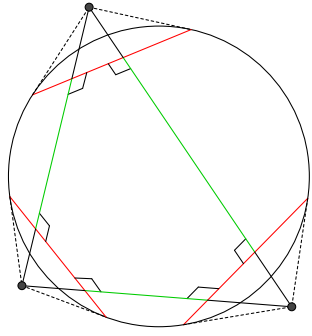
	The proof proceeds exactly as in the preceding lemma, except that all the edges of the face $u$ of $P'$ meet the the open unit disk in $\Pi_{u}$. This implies, as in Fig.~\ref{fig:Bao-Bon2}, that the ortho-circles are pairwise disjoint. In the hyperbolic plane $D_{u}$, the common orthogonal segment to two adjacent ortho-circles with cone points $a$ and $b$, adjacent vertices of the face $u$, project orthogonally to the edge between $a$ and $b$. These give the green segments that alternate with the segments that lie along the ortho-circles to cut out a compact convex green-black polygon with alternating edge colors. Again, the \textit{c}-link of the vertex $u$ is a compact convex hyperbolic polygon and $G(\mathcal{C})$ is proper at $u$.
\end{proof}

%% file: bowersBowersPratt-Cauchy.bbl
\begin{thebibliography}{10}

\bibitem{Andreev:1970a}
E.M. Andre'ev.
\newblock On convex polyhedra in lobachevski spaces.
\newblock {\em Mat. Sbornik}, 81(123):445--478, 1970.

\bibitem{Andreev:1970b}
E.M. Andre'ev.
\newblock On convex polyhedra of finite volume in lobachevski spaces.
\newblock {\em Mat. Sbornik}, 83(125):256--260, 1970.

\bibitem{BaoBonahon:2002}
Xiliang Bao and Francis Bonahon.
\newblock Hyperideal polyhedra in hyperbolic 3-space.
\newblock {\em Bull. Soc. Math. France}, 130(3):457--491, 2002.

\bibitem{bowersBowers:2016}
John~C. Bowers and Philip~L. Bowers.
\newblock {Ma-Schlenker \textit{c}-octahedra in the $2$-sphere}.
\newblock 2016.

\bibitem{Bowers:isoInversive:techReport}
John~C. Bowers and Philip~L. Bowers.
\newblock {Iso-inversive embeddings of circles into the 2-sphere}.
\newblock 2017.

\bibitem{Bowers:rigidityOfCircle:techReport}
John~C. Bowers and Philip~L. Bowers.
\newblock {Rigidity of circle configurations in the planes and the sphere}.
\newblock 2017.

\bibitem{Bowers:2003kr}
Philip~L. Bowers and Monica~K. Hurdal.
\newblock {Planar conformal mappings of piecewise flat surfaces}.
\newblock {\em Visualization and Mathematics III}, (Chapter 1):3--34, 2003.

\bibitem{Bowers:2004bg}
Philip~L. Bowers and Kenneth Stephenson.
\newblock {Uniformizing dessins and Bely{\u{\i}} maps via circle packing}.
\newblock {\em Memoirs of the AMS}, 170(805):1--97, 2004.

\bibitem{Cauchy1813}
A.~Cauchy.
\newblock {Sur les polygones et les polyhedres}.
\newblock {\em J. Ecole Polytechnique XVIe Cahier}, pages 87--98, 1813.

\bibitem{Dehn:1905}
Max Dehn.
\newblock Die eulersche formel im zusammenhang mit dem inhalt in der
  nicht-euklidischen geometrie.
\newblock {\em Math. Ann.}, 61:561--586, 1905.

\bibitem{FuchsTab:2007}
Dmitry Fuchs and Serge Tabachnikov.
\newblock {\em {Mathematical Omnibus: Thirty Lectures on Classic Mathematics}}.
\newblock American Mathematical Society, 2007.

\bibitem{Guo:2011kf}
Ren Guo.
\newblock {Local rigidity of inversive distance circle packing}.
\newblock {\em Transactions of the American Mathematical Society},
  363(9):4757--4776, September 2011.

\bibitem{HodgsonRivin:1993}
C.D. Hodgson and Igor Rivin.
\newblock A characterization of compact convex polyhedra in hyperbolic 3-space.
\newblock {\em Invent. Math.}, 111:77--111, 1993.

\bibitem{Luo:2011ex}
Feng Luo.
\newblock {Rigidity of polyhedral surfaces, III}.
\newblock {\em Geometry {\&} Topology}, 15(4):2299--2319, December 2011.

\bibitem{Ma:2012hl}
Jiming Ma and Jean-Marc Schlenker.
\newblock {Non-rigidity of spherical inversive distance circle packings}.
\newblock {\em Discrete {\&} Computational Geometry}, 47(3):610--617, February
  2012.

\bibitem{Poincare:1881}
Henri Poncar\'e.
\newblock Sur les groupes klein\'eens.
\newblock {\em C. R. Acad. Sci. Paris}, 93:44--46, 1881.

\bibitem{Rivin:1996}
Igor Rivin.
\newblock A characterization of ideal polyhedra in hyperbolic 3-space.
\newblock {\em Ann. of Math.}, 143:51--70, 1996.

\bibitem{Rousset:2004}
Mathais Rousset.
\newblock Sur la rigidit{\'e} de poly{\`e}dres hyperboliques en dimension 3 :
  cas de volume fini, cas hyperid{\'e}al, cas fuchsien.
\newblock {\em Bulletin de la Soci{\'e}t{\'e} Math{\'e}matique de France},
  132(2):233--261, 2004.

\bibitem{Sch:1962}
Hans Schwerdtfeger.
\newblock {\em {Geometry of Complex Numbers}}.
\newblock Number~13 in Mathematical Expositions. Univiversity of Toronto Press,
  1962.

\bibitem{Thurston:1980}
William~P. Thurston.
\newblock The geometry and topology of 3-manifolds.
\newblock Lecture Notes: Princeton University, 1980.

\end{thebibliography}
